\documentclass[12pt,reqno]{amsart}
\usepackage{amsmath}
\usepackage{hyperref}
\usepackage{xfrac}
\usepackage{stmaryrd,mathrsfs,bm,amsthm,mathtools,yfonts,amssymb,color}
\usepackage{xcolor}
\usepackage{courier}
\usepackage{verbatim}
\usepackage[letterpaper,margin=1.1in]{geometry}

\newcommand\sphere{\mathbb{S}}
\newcommand{\eps}{\varepsilon}
\newcommand{\haus}{\mathcal{H}}
\newcommand{\leb}{\mathcal{L}}
\newcommand{\sing}{\mathrm{sing}}
\newcommand{\graph}{\mathrm{graph}}
\newcommand{\dom}{\mathrm{dom}}
\newcommand{\spt}{\mathrm{spt}}

\newcommand{\bC}{\mathbf{C}}

\newcommand{\bT}{\mathbf{T}}
\newcommand{\bY}{\mathbf{Y}}
\newcommand{\bX}{\mathbf{X}}
\newcommand{\R}{\mathbb{R}}
\newcommand{\bbS}{\mathbb{S}}
\newcommand{\cC}{\mathcal{C}}
\newcommand{\cN}{\mathcal{N}}
\newcommand{\interior}{\mathrm{int}}
\newcommand{\reg}{\mathrm{reg}}
\newcommand{\cL}{\mathcal{L}}
\newcommand{\cV}{\mathcal{V}}

\newcommand{\clus}{\mathcal{E}}
\newcommand{\mass}{\mathbf{M}}
\newcommand{\size}{\mathbf{S}}
\newcommand{\ba}{\mathbf{a}}
\newcommand{\bb}{\mathbf{b}}

\newcommand{\axis}{{ \{0\} \times \R^m }}
\newcommand{\refC}{\mathbf{C}}

\newtheorem{theorem}{Theorem}[section]

\newtheorem{proposition}[theorem]{Proposition}

\newtheorem{prop}[theorem]{Proposition}
\newtheorem{lemma}[theorem]{Lemma}
\newtheorem{corollary}[theorem]{Corollary}
\theoremstyle{definition}
\newtheorem{definition}[theorem]{Definition}
\theoremstyle{remark}
\newtheorem{remark}[theorem]{Remark}
\theoremstyle{remark}
\newtheorem{example}[theorem]{Example}
\theoremstyle{remark}

\theoremstyle{remark}
\newtheorem{question}[theorem]{Question}
\theoremstyle{remark}

\numberwithin{equation}{section}

\title{The singular set of minimal surfaces near polyhedral cones}
\author{Maria Colombo, Nick Edelen, and Luca Spolaor}

\address{Institute for Theoretical Studies, ETH Z\"urich, Clausiusstrasse 47, CH-8092 Z\"urich, Switzerland}
\email{maria.colombo@eth-its.ethz.ch}

\address {Massachusetts Institute of Technology, 77 Massachusetts Avenue, Cambridge, MA 02139, USA}
\email{nedelen@mit.edu}

\address {Massachusetts Institute of Technology, 77 Massachusetts Avenue, Cambridge, MA 02139, USA}
\email{lspolaor@mit.edu}

\begin{document}

\begin{abstract}
We adapt the method of Simon \cite{simon1} to prove a $C^{1,\alpha}$-regularity theorem for minimal varifolds which resemble a cone $\bC_0^2$ over an equiangular geodesic net.  For varifold classes admitting a ``no-hole'' condition on the singular set, we additionally establish $C^{1,\alpha}$-regularity near the cone $\bC_0^2 \times \R^m$.  Combined with work of Allard \cite{All}, Simon \cite{simon1}, Taylor \cite{taylor}, and Naber-Valtorta \cite{naber-valtorta}, our result implies a $C^{1,\alpha}$-structure for the top three strata of minimizing clusters and size-minimizing currents, and a Lipschitz structure on the $(n-3)$-stratum. 
\end{abstract}

\maketitle

\tableofcontents


\section{Introduction}



In this paper we are interested in the regularity and fine-scale structure of stationary integral varifolds (and varifolds with bounded mean curvature) which resemble \emph{polyhedral-type} cones.  That is, we address the following question:
\begin{question}\label{question:main}
Suppose $\bC_0^2 \subset \R^{2+k}$ is the cone over an equiangular geodesic net in $\bbS^{1+k}$ (each junction meeting precisely three arcs), and $M^{2+m}$ is a stationary integral varifold weakly close to $\bC = \bC_0\times\R^m$.  Then what can be said about the regular and singular structure of $M$?
\end{question}

Understanding the relationship between the local (singular) structure of a minimal surface $M$ and its tangent cone $\bC$ has been a central question in geometric analysis, even when multiplicity is not a factor.  There are many profound and optimal results concerning the \emph{dimension} of the singular set, in various circumstances (e.g. \cite{All}, \cite{Alm}, \cite{DG}, \cite{Fed3}, \cite{ScSi}, \cite{BDG}), but relatively few works have addressed the \emph{structure} of $M$ near singularities (except notably when the singular set has dimension $0$, in which case it is often known to be locally finite \cite{Fed3, Chang}).

Generally the best results are known when a $\bC$ has smooth cross-section (and multiplicity-one), or when $M$ belongs to a class with very rigid tangent cone structure, with topological obstructions to ``perturbing'' them away.

For example, when $\bC$ is smooth, multiplicity-one the picture is largely complete: using ideas dating back to De Giorgi, Allard \cite{All} and Allard-Almgren \cite{AllAlm} proved that if $\bC$ satisfies an \emph{integrability} hypothesis, then $M$ is a locally $C^{1,\alpha}$-perturbation of $\bC$; later, in huge generality, Simon \cite{simon:loja} proved that for \emph{any} such $\bC$ (not necessarily integrable), then $M$ is locally a $C^{1,\log}$-perturbation; and in Adams-Simon \cite{adams-simon} this decay rate was shown to be sharp.

For $2$-dimensional $(\mass, \eps, \delta)$-minimizing sets in $\R^3$, Taylor \cite{taylor} (see also \cite{david1}, \cite{david2}) has shown the following beautiful structure theorem: $M^2$ decomposes into a union of $C^{1,\alpha}$ manifolds, meeting along various $C^{1,\alpha}$ curves at $120^\circ$, which in turn meet at isolated tetrahedral junctions.  (We remark that, though they may coincide in certain minimization problems, the notions of being $(\mass,\eps,\delta)$-minimizing and having bounded mean curvature are essentially independent.\footnote{For example, a union of $\geq 2$ intersecting lines is stationary but not $(\mass, \eps, \delta)$-minimizing.  Conversely, the $1$-d graph of the curve $f(x) = |x|^3 \sin(1/|x|)$ is $(\mass, c r^2, 1)$-minimizing in $B_1^2(0)$ but its mean curvature does not lie in any $L^p$ for $p > n = 1$.  See Section \ref{sec:clusters} for more background.})  Crucial to Taylor's work is a classification of tangent cones for this class of sets.  For certain $(\mass, 0, \delta)$-minimizing sets, our work generalizes Taylor's theorem to higher dimensions.

Simon \cite{simon1} was the first to consider the singular structure for general \emph{stationary} integral varifolds, and for varifold classes \emph{without} such rigid tangent cone structure.  He considered cones of the form $\bC = \bC_0^\ell\times \R^m$, where $\bC_0$ is smooth and integrable, and proved an ``excess decay dichotomy,'' which loosely says that \emph{either} $M$ has a significant gap in the singular set, \emph{or} the scale-invariant $L^2$-distance $\rho^{-n-2}\int_{M \cap B_\rho} d_{\bC}^2$ decreases when the radius is decreased by a fixed factor.

For certain tangent cones which cannot be ``perturbed away,'' like the union of three half-planes, and in general for classes of $M$ admitting some kind of ``no-hole'' condition on the singular set, Simon's result implies that $\sing(M)$ is locally a $C^{1,\alpha}$ manifold (see Theorem \ref{thm:Y-reg}).  More generally, he used his decay dichotomy to show countable-rectifiability of the singular set for particular ``multiplicity-one'' classes, e.g. for mod-$2$ minimizing flat chains.  

Later, in \cite{simon:rect} Simon used the Lojaciewicz inequality to show countable-rectifiability of each stratum\footnote{There is a subtle difference between the strata of Almgren used in Naber-Valtorta \cite{naber-valtorta}, and the strata used in Simon \cite{simon:rect}: Simon defines $S^m(M)$ to be the set of points for which every tangent cone $\bC$ satisfies $\dim(\sing \bC) \leq m$, rather than asking for every tangent cone to have $\leq m$ dimensions of symmetry.  In paticular, the if $M^2 = \bT$ is the tetrahedral cone (defined in Section \ref{sec:polyhedral}), then $0$ lies in the $0$-stratum for Naber-Valtorta, but only the $1$-stratum for Simon.} of $M$ in \emph{any} ``multiplicity-one class'' (e.g. codimension-$1$ mass-minimizing currents), and almost-everywhere uniqueness of the tangent cones in the singular set.  Just recently Naber-Valtorta \cite{naber-valtorta} proved rectifiability of each stratum for general stationary integral varifolds, and rectifiability \emph{with mass bounds} of the singular set for $M$ in any multiplicity-one class.

\vspace{5mm}

We generalize the seminal results of \cite{simon1} to prove that whenever $M$ admits a certain ``no-holes'' property on the singular set, and $\bC_0$ is integrable, then $M$ as in Question \ref{question:main} must be a $C^{1,\alpha}$-perturbation of $\bC$.  Integrability loosely means that every infintesimal motion through polyhedral cones can be generated by a family rotations, see Section \ref{sec:compatible}.  Both of these conditions are satisfied in several natural circumstances, and for a wide class of cones.

Our main Theorem \ref{thm:main-decay} is an excess decay dichotomy in the spirit of Simon, and is given in Section \ref{sec:main-thm}.  Here we list two consequences, which correspond to two classes of varifold admitting a no-hole condition.
\begin{theorem}[$\eps$-regularity for polyhedral cones]\label{thm:no-spine-reg}
Let $\refC^2 \subset \R^2 \subset \R^{2+k}$ be a polyhedral cone.  There are $\delta(\refC), \mu(\refC) \in (0, 1)$ so that if $M$ is an integral varifold with bounded (generalized) mean curvature $H_M$ and no boundary in $B_1$, satisfying
\begin{align}
\theta_M(0) \geq \theta_{\bC}(0), \quad \mu_M(B_1) \leq \frac{3}{2} \theta_\bC(0), \quad \int_{B_1} \mathrm{dist}(z, \bC)^2 d\mu_M + ||H_M||_{L^\infty(B_1)} \leq \delta^2,
\end{align}
then $\spt M \cap B_{1/2}$ is a $C^{1,\mu}$-perturbation of $\refC$.
\end{theorem}
We remark that the restriction $\bC_0 \subset \R^3$ is due to integrability: Theorem \ref{thm:no-spine-reg} holds for \emph{any} integrable polyhedral $\bC_0$, but we can only verify integrability for those nets in $\bbS^2$ (indeed, we feel integrability may be generally false in higher codimension)

A second class admitting the no-hole condition consists of varifolds with an associated orientation (i.e. current) structure.  These arise naturally as size- and cluster-minimizers, and we correspondingly have the following interior regularity theorem.
\begin{theorem}[Regularity of size-/cluster-minimizers]\label{thm:clusters}
Let $M^n$ be the support of the integral varifold associated to either a minimizing cluster in $U = \R^{n+1}$, or a homologically size-minimizing current in an open set $U$ (e.g. as constructed by Morgan \cite{morgan-size}).  Then we can decompose $M \cap U = M_n \cup M_{n-1} \cup M_{n-2} \cup M_{n-3}$ (disjoint union), where:
\begin{enumerate}
\item $M_{n}$ is a locally-finite union of embedded $C^{1,\alpha}$ $n$-manifolds;

\item $M_{n-1}$ is a locally-finite union of embeded, $C^{1,\alpha}$ $(n-1)$-manifolds, near which $M$ is locally diffeomorphic to $\bY^1 \times \R^{n-1}$;

\item $M_{n-2}$ is a locally-finite union of embedded, $C^{1,\alpha}$ $(n-2)$-manifolds, near which $M$ is locally diffeomorphic to $\bT^2 \times \R^{n-2}$;

\item $M_{n-3}$ is relatively closed, $(n-3)$-rectifiable, with locally-finite $\haus^{n-3}$-measure.
\end{enumerate}
Here $\bY^1$ is the stationary $1$-dimensional cone consisting of three rays, and $\bT^2$ is the stationary $2$-dimensional cone over the tetrahedral net in $\bbS^2$ (see Section \ref{sec:polyhedral} for precise definitions).
\end{theorem}

\begin{remark}
In either of the above cases, standard interior estimates and work of \cite{kns},  \cite{krummel} imply that $M_n$ and $M_{n-1}$ are analytic.  Contrarily, we suspect $C^{1,\alpha}$ may be sharp for $M_{n-2}$, as there exist Jacobi fields on $\bT^2$ which near $0$ are bounded in $C^{1,\alpha}$ but not $C^{2,\alpha}$.
\end{remark}

When $n = 2$ then Theorem \ref{thm:clusters} has been established (for general $(\mass, \eps, \delta$)-minimizing sets) by Taylor \cite{taylor}.  David \cite{david1}, \cite{david2} has given an entirely different proof of Taylor's Theorem, and has proven partial generalizations to higher codimension.

For general $n$, conclusions, 1), 2) are respectively consequences of Allard's \cite{All} and Simon's work \cite{simon1}.  Conclusion 4) follows from parts 1), 2), 3), and the work of Naber-Valtorta \cite{naber-valtorta}.  White \cite{white-announce} has announced a result analogous to Theorem \ref{thm:clusters} parts 2), 3) for general $(\mass, \eps, \delta)$-minimizing sets.  We mention that, in light of the essentially independent natural of general $(\mass, \eps, \delta)$-minimizing sets, and varifolds with bounded mean curvature, our main Theorem is likewise independent from the result asserted by White.

\vspace{5mm}

The very broad strategy of proof is to ``linearize'' the minimal surface operator over $\bC$, and use good decay properties of solutions to the linearized problem (called \emph{Jacobi fields}) to prove decay of the minimal surface.  In general the linear problem may not adequately capture the non-linear problem, and for this reason we must (as in \cite{simon1}) make two running assumptions: first, we require the polyhedral cone $\bC_0$ to be integrable (Definition \ref{def:integrable}), to ensure that every $1$-homogeneous Jacobi field can be realized ``geometrically'' through a family of rotations; second, we require the singular set of $M$ to satisfy a no-holes condition (Definition \ref{def:no-holes}), which prevents the tangent cone from ``gaining'' symmetries not seen in $M$.


Our proof follows Simon \cite{simon1}, but there several complications when dealing with polyhedral cylindrical cones.  Our main contributions are making sense of inhomogeneous blow-ups on cylindrical cones with \emph{singular cross section $\bC_0$}, correspondingly defining a good notion of Jacobi field on polyhedral cones, and extending the various non-concentration and growth estimates of Simon to the singular setting.  We additionally remove the ``multiplicity-one'' hypothesis from the excess decay Theorem (in both our result and Simon's), but we caution the reader that the structural results of Simon still require this hypothesis.

\vspace{5mm}

\textbf{Acknowledgments} We are grateful to Guido De Philippis for many interesting discussions and for introducing us to the problem.  We thank Spencer Becker-Kahn, Leon Simon, and Neshan Wickramasekera for several helpful conversations.  We wish to acknowledge the support of Gigliola Staffilani, whose grant allowed M.C. to visit MIT.  N.E. was supported by NSF grant DMS-1606492.  

\section{Notation and preliminaries}\label{sec:prelim}

Let us fix some notation.  We work in $\R^{n+k} = \R^{\ell + k} \times \R^m \ni (x, y)$.  We denote by capital letters $X \in \R^{n+k}$.  Write $r = |x|$, and $R = |X| = \sqrt{|x|^2 + |y|^2}$.  We shall always write $d_A(x)$ for the Euclidean distance function to a set $A$.  We write $B_r(A) = \{ x : d_A(x) < r \}$ for the open $r$-tubular neighborhood of $A$.  More generally, given a radius function $r_x : A \to \R$, we write $B_{r_x}(A) = A \cup \bigcup_{x \in A} B_{r_x}(x)$.

Given a linear subspace $V \subset \R^{n+k}$, we write $V^\perp$ to denote its orthogonal complement, and $\pi_V$ for the linear projection operator.  Given another linear space $W$, write $<V, W>^2 = \sum_{i, j} e_i \cdot f_j$ for the distance between $V$, $W$, where $\{e_i\}_i$, $\{f_j\}_j$ are choices of orthonormal basis on $V$, $W$.  The $\cdot$ always denotes Euclidean inner product.

We will be working with $n = (\ell + m)$-dimensional integral varifolds in $\R^{n+k} = \R^{(\ell+k)+m}$ with bounded mean curvature, and the reader should always think of them as having (almost-)symmetry in the $\{0\}\times \R^m$ factor.  Any cone $\bC$ will be a rotation of $\bC_0^\ell \times \R^m$ where $\bC_0$ is $\ell$-dimensional, stationary and either smooth or polyhedral (see Definition \ref{def:polyhedral-cone}).

Typically $M$ will denote a general integral varifold, and $\mu_M$ will be its mass measure.  Our integral varifolds will always have bounded mean curvature and no boundary in $B_1$.  This means there is a $\mu_M$-a.e. bounded vector field $H_M$, so that
\begin{gather}\label{eqn:first-variation}
\int_M div_M(Y) = - \int_M H_M \cdot X \quad \forall Y \in C^1_c(B_1, \R^{n+k}).
\end{gather}
Here $div_M(Y)$ is the tangential divergence, defined at $\mu_M$-a.e. point by $div_M(Y) = \sum_i e_i \cdot (D_{e_i} Y)$, for any orthonormal basis $\{e_i\}_i$ of $T_X M$.  Of course $M$ is \emph{stationary} if $H_M \equiv 0$.

We shall aways write $\theta_M(X, R)$ for the Euclidean density ratio in $B_R(X)$:
\begin{gather}
\theta_M(X, R) := r^{-n} \mu_M(B_R(X)).
\end{gather}
When $|H_M| \leq \Lambda_M$, and $M$ has no boundary in $B_1$, the $\theta_M(X, R)$ is almost-monotone in the sense that 
\begin{gather}\label{eqn:monotonicity}
e^{\Lambda_M R} \theta_M(X, R) \text{ is increasing for all $X \in B_1$ and $R < 1-|X|$.}
\end{gather}
In this case the density at $X$ is well-defined
\begin{gather}
\theta_M(X) := \lim_{r \to 0} \theta_M(X, R).
\end{gather}

By the monotonicity \eqref{eqn:monotonicity} any integral varifold having bounded mean curvature and no boundary can be identified with its support plus multiplicity, in the sense that
\begin{gather}
\int f d\mu_M = \int_{\spt M} f \theta d\haus^n
\end{gather}
for some $\mu_M$-measurable, integer-valued function $\theta$.  We shall make this identification.  In particular, we shall use the following shorthand:
\begin{gather}
\int_{M \cap A} f \equiv \int_A f d\mu_M, \quad d_M(x) \equiv d_{\spt M}(x), \quad \phi(M) \equiv \phi_\sharp M,
\end{gather}
Here $\phi : \R^{n+k} \to \R^{n+k}$ is some $C^1$ mapping, and $\phi_\sharp$ is the pushforward.

We write $\reg M$ for the set of points in $M$ for which $M$ locally coincides with a $C^{1,\alpha}$ graph, and $\sing M = M \setminus \reg M$.

We may further stratify $M$ using the quantitative strata of Cheeger-Naber \cite{cheeger-naber}.  We say a varifold cone $\bC$ is $m$-symmetric if it takes the form $q(\bC_0\times \R^m)$ for some $q \in SO(n+k)$.  The $m$-stratum of $M$ consists of points
\begin{gather}
S^m(M) = \{ X \in M : \text{ no tangent cone at $X$ is $(m+1)$-symmetric} \}.
\end{gather}

Fix a metric $d_{\cV}$ on the space of $n$-varifolds which induces varifold convegence.  We say $M$ is $(m,\eps)$-symmetric in some ball $B_r(x)$ if $d_{\cV}(r^{-1} (M - X) \llcorner B_1, \bC  \llcorner B_1) < \eps$ for some $m$-symmetric cone $\bC$.  The $(m, \eps)$-stratum then consists of points
\begin{gather}
S^m_\eps(M) = \{ X \in M : \text{ $M$ is \emph{not} $(\eps, m+1)$-symmetric in $B_r(X)$ for all $r < 1$} \}.
\end{gather}

We will often use the following local Holder-semi-norm.  Suppose $\bC$ is a cone, and $f : \Omega \subset \bC \to \bC^\perp$, and $(x, y) \in \bC \cap B_{1/2}$.  Then we define
\begin{gather}
[f]_{\alpha, \bC}(x, y) = \sup \left\{ \frac{|f(Z) - f(W)|}{|Z - W|^\alpha} : Z, W \in \Omega \cap B_{|x|/4}(x, y) \right\}.
\end{gather}
One can easily verify the following compactness: if $f_i : \bC_i \cap B_1 \to \bC_i^\perp$ is a sequence of functions, and $\bC_i \to \bC$ in $C^{1,\alpha}_{loc}$, and $[f_i]_{\alpha,\bC_i}$ is uniformly bounded on compact subsets, then after passing to a subsequence we have $f_i \to f : \bC \cap B_1 \to \bC^\perp$ in $C^{0,\alpha'}_{loc}$ for any $\alpha' < \alpha$.




\subsection{One-sided excess, holes}

We shall prove decay of the following \emph{one-sided} excess.
\begin{definition}
Let $M$ be an integrable varifold with bounded mean curvature, and $\bC$ a stationary integral varifold cone, and $\delta \in (0, 1]$.  Then we define
\begin{gather}
E_\delta(M, \bC, X, R) = R^{-n-2} \int_{M \cap B_R(X)} d_{\bC}^2 + \delta^{-1} R ||H_M||_{L^\infty(M \cap B_R)} .
\end{gather}
Notice $E$ is scale-invariant, in the sense that $E_\delta(M, \bC, X, R) = E_\delta(R^{-1}(M - X), R^{-1}(\bC - X), 0, 1)$.  When $\delta = 1$ we may write $E$ instead of $E_\delta$, and when $X = 0$ we may simply write $E_\delta(M, \bC, R)$.
\end{definition}

\begin{remark}
The factor of $\delta^{-1}$ effectively ``captures'' the region where $R ||H_M||_{L^\infty(B_R)}$ is much smaller than $L^2$-distance.  Our ultimate blow-up argument must work in this regime -- it cannot hope to see regions where $L^2$-excess is controlled by mean curvature, since in this case excess decay is dominated by the scaling of $H_M$.
\end{remark}

\begin{definition}
Take a fixed cone $\bC = \bC_0^\ell\times \R^m$, and $\eps > 0$.  We let $\cC_{\eps}(\bC)$ be the collection of cones
\begin{gather}
\cC_{\eps}(\bC) = \{ q(C) : q \in SO(n+k) \text{ with } |q - Id| \leq \eps \}.
\end{gather}

Define $\cN_\eps(\refC)$ to be the set of integral $n$-varifolds $M^n \subset \R^{n+k}$, having bounded mean curvature and no boundary in $B_1$, satisfying
\begin{gather}
0 \in M, \quad \mu_M(B_1) \leq \frac{3}{2} \theta_{\refC}(0), \quad E(M, \refC, 0, 1) \leq \eps^2.
\end{gather}
\end{definition}

In general, even an isolated singularity could potentially have a tangent cone with lots of symmetry.  To prove decay of $M$ towards a cone $\refC = \refC_0^\ell \times \R^m$ with a spine of singularities, we must, like in \cite{simon1}, impose a ``no-holes'' condition on $M$.  In certain circumstances the no-hole condition can be deduced for topological reasons (note that no lower density assumptions are made in the class $\cN_\eps$).
\begin{definition}\label{def:no-holes}
Take the cone $\refC^n = \refC_0^\ell\times \R^m$.  We say $M^{n}$ satisfies the \emph{$\delta$-no-holes condition} in $B_r$ w.r.t. $\refC$ if the following holds: for any $y \in B_r^m$, there is some $X \in B_\delta(0, y)$ with $\theta_M(X) \geq \theta_\refC(0)$.
\end{definition}


\subsection{Polyhedral cones}\label{sec:polyhedral}

We are concerned with the following types of cones.
\begin{definition}\label{def:polyhedral-cone}
A $2$-dimensional cone $\bC_0^2 \subset \R^{2+k}$ is \emph{polyhedral} if:
\begin{enumerate}
\item[A)] $\bC_0$ is the cone over some multiplicity-1 geodesic net in $\bbS^{1+k}$, having the property that every junction has precisely three edges meeting at $120^\circ$ (nets with this property are sometimes called \emph{equiangular}), and

\item[B)] the cone $\bC_0$ has no additional symmetries, i.e. we cannot write $\bC_0 = q(\bC_0' \times \R)$ for some $q \in SO(2 + k)$, and some $1$-dimensional cone $\bC_0'$.
\end{enumerate}
We shall often say a cone $\bC_0^2 \times \R^m$ is polyhedral if $\bC_0$ is polyhedral.
\end{definition}

The equiangular geodesic nets in $\bbS^2$ are completely classified, and there are $10$ of them.  However, by our Definition \ref{def:polyhedral-cone} only $8$ of these nets give rise to polyhedral cones.  For a comprehenive list see Section \ref{sec:int}.  Let us remark that, from the work of \cite{AllAlm2}, any integer-multiplicity geodesic net in $\bbS^{1+k}$ (with finite mass) consists of only finitely many geodesic arcs.

We bring the readers attention to two important (non-)examples.  Define $\bY^1 \subset \R^2$ to be the cone consisting of three rays meeting at $120^\circ$:
\begin{gather}
\bY^1 = \{ (x, 0) : x \geq 0 \} \cup \{ (x, -\sqrt{3} x) : x \leq 0 \} \cup \{ (x, \sqrt{3}) : x \leq 0 \} .
\end{gather}
The cone $\bY^1 \times \R$ arises from the geodesic net consisting of three half-great-circles meeting at $120^\circ$.  Though of fundamental importance in this paper, cones $\bY^1 \times \R$ and $\R^2$ are \emph{not} considered polyhedral cones.

Define $\bT^2 \subset \R^3$ to be the cone over the tetrahedral net: $\bT^2 \cap \bbS^2$ is the equiangular net having vertices
\begin{gather}
(1, 0, 0), \quad (-1/3, 2\sqrt{2}/3, 0), \quad (-1/3, -\sqrt{2}/3, \sqrt{6}/3), \quad (-1/3, -\sqrt{2}/3, -\sqrt{6}/3).
\end{gather}
The tetrahedral cone $\bT^2$ is the archtype of polyhedral cone, and is the polyhedral cone of least density in $\R^3$.  It would be interesting to know whether $\bT^2$ is the least density polyhedral cone in any codimension.

\begin{remark}\label{rem:looks-like-Y}
A very important fact is that if $\bC = \bC_0^2 \times \R^m$ is polyhedral, then away from the axis $\axis$ a small neighborhood is (up to rigid motion) either flat $\R^{2+m}$ or the cone $\bY^1 \times \R^{1+m}$.
\end{remark}

\vspace{5mm}

The cone $\bY^1\times \R^{m}$ we shall decompose into three half-planes $H(1) \cup H(2) \cup H(3)$, and write $Q(i)$ for the $m$-plane containing $H(i)$, $n(i)$ for the outer conormal of $\partial H(i) \subset Q(i)$.

To adequately parameterize polyhedral surfaces and cones we require some further notation.  We call a subset of the form
\begin{equation}\label{eqn:wedge-example}
W = \{ re^{i\theta} \in \R^2 : \theta \in [\theta_0, \theta_1] \text{ and } r \in [0, \infty) \} \subset \R^2
\end{equation}
a wedge.  Given a plane $P^2 = q(\R^2) \subset \R^{n+k}$, for $q \in O(n+k)$, a subset $W \subset P^2$ is wedge if $q^{-1}(W) \subset \R^2$ is a wedge.  We shall write $\mathrm{int} W$ for the ``interior'' points
\begin{gather}
\mathrm{int} W = \{ r e^{i\theta} : \theta \in (\theta_0, \theta_1) \text{ and } r \in (0, \infty) \},
\end{gather}
and $\partial W$ for the ``boundary'' points
\begin{gather}
\partial W = \{ r e^{i\theta} : \theta \in \{\theta_0, \theta_1\} \text{ and } r \in [0, \infty)\}.
\end{gather}

We shall decompose our polyhedral cone $\bC_0$ into a union of wedges $W(1), \ldots, W(d)$, meeting along a collection of lines $L_1, \ldots, L(2d/3)$.  Let us write $P(i)$ for the $2$-plane containing $W(i)$, and $n(i)$ for the outer conormal of $\partial W(i)$ in $P(i)$.  A function $v : \bC_0 \to \bC_0^\perp$ is interpreted as a collection of functions $v(i) : W(i) \to P(i)^\perp$.

When dealing with polyhedral cones, it will be convenient to have a notion of annulus which is flat near the junctions.  Given a wedge $W \subset \R^2$ as in \eqref{eqn:wedge-example}, define the star-shaped curve $S_W$ by letting
\begin{gather}
r(\theta) = \left\{ \begin{array}{ l l} 
 \frac{1}{\cos(\theta - \theta_0)}  & \theta \in [\theta_0 - (\theta_1-\theta_0)/4, \theta_0 + (\theta_1 - \theta_0)/4] \\
 \frac{1}{\cos(\theta - \theta_1)} & \theta \in [\theta_1 - (\theta_1 - \theta_0)/4, \theta_1 + (\theta_1 - \theta_0)/4] \\
 \frac{1}{\cos((\theta_1 - \theta_0)/4)} & else \end{array} \right.
\end{gather}
For all intents and purposes $S_W$ is a circle, but because $S_W$ is linear in a neighborhood of $\partial W$, our lives are simplified when dealing with domains whose boundaries are graphs over $\partial W$.  We have
\begin{equation}\label{eqn:circle-inclusions}
S_W \subset B_2 \setminus B_1
\end{equation}

Let us correspondingly define the annular domain
\begin{gather}
A_W(r_1, r_2) = \bigcup_{r_1 < r < r_2} r S_W.
\end{gather}
As before, $A_W(r_1, r_2)$ is essentially just a round annulus, but is slightly adjusted to fit $W$ better.  If $W \subset P^2 = q(\R^2)$ is a wedge, then we define $A_W := q(A_{q^{-1}W})$ in the obvious way.  As in \eqref{eqn:circle-inclusions}, we have
\begin{equation}\label{eqn:wedge-inclusions}
(B_{r_2} \setminus B_{2r_1}) \cap W \subset A_W(r_1, r_2) \subset (B_{2r_2} \setminus B_{r_1}) \cap W.
\end{equation}


\subsection{Compatible Jacobi fields}\label{sec:compatible}

Let us consider the following model scenario: fix a polyhedral cone $\bC^n = \bC_0^2\times \R^m$, and take $M^n_t$ to be a $1$-parameter family of minimal surfaces, continuous in the varifold distance, satifying $M_0 = \bC$.  For any $\tau > 0$ and $t$ sufficiently small (depending on $\tau$), the $M_t \setminus B_\tau(\axis)$ are graphical over $\bC$ by a function $u_t$, in some suitable sense (see Lemma \ref{lem:poly-graph}).

We can define the initial velocity $v = \partial_t u_t$ as a function $v : \bC \to \bC^\perp$.  The resulting $v$ will satisfy $\Delta v = 0$ on each wedge, and certain compatibility conditions on the junction lines.  This PDE system is a notion of \emph{linearization} of the mean curvature operator over $\bC$ (which itself we do not explicitly define).  We call such a $v$ a compatible Jacobi field.

We shall see in Section \ref{sec:blow-up} how general inhomogeneous blow-up sequences give rise to compatible Jacobi fields.

\begin{definition}\label{def:compatible}
Let $\bC = \bC_0^2\times\R^m$ be a polyhedral cone.  We say $v : \refC \cap B_1 \to {\refC}^\perp$ is a \emph{compatible Jacobi field} on $\bC \cap B_1$ if it satisfies the following conditions:

\begin{enumerate}
\item[A)] For each $i$, $v(i)$ is smooth on $((W(i)\setminus\{0\})\times \R^m) \cap B_1$, and satisfise $\Delta v(i) = 0$.

\item[B)] (``$C^0$ compatibility'') For every $z \in ((\partial W(i) \setminus \{0\}) \times \R^m) \cap B_1$, there is a vector $V(z) \in \R^{2+k}$ (\emph{independent} of $i$) so that 
\begin{gather}
v(i)(z) = \pi_{P(i)^\perp}(V(z)).
\end{gather}

\item[C)] (``$C^1$ compatibility'') If $W(i_1)$, $W(i_2)$, $W(i_3)$ share a common edge $\partial W(i_1)$, then
\begin{gather}
\sum_{j=1}^3 \partial_n v(i_j)(z) = 0 \quad \forall z \in \left( (\partial W(i_1) \setminus \{0\}) \times \R^m\right) \cap B_1.
\end{gather}
\end{enumerate}
We say a compatible Jacobi field $v$ is \emph{linear} if there are skew-symmetric matrices $A(i) : \R^{n+k} \to \R^{n+k}$ so that $v(i) = \pi_{P(i)^\perp} \circ A(i)$.  Notice we \emph{do not} require the $A(i)$ to coincide.
\end{definition}

As outlined in the Introduction, we wish to use the decay properties of compatible Jacobi fields to prove excess decay on $M$.  There is a catch however, which is illustrated in the following example: let $M_t$, $\bC$, and $v$ be as in the previous example.  Let $q_t = \exp(tA) \in SO(n+k)$ be a $1$-parameter family of rotations generated by the skew-symmetric matrix $A$. 

Then one can easily verify the initial velocity of the family $M_t$ as graphs over $q_t(\bC)$ is now $v - \pi_{\bC^\perp} \circ A$, which decays at most linearly, and in particular is insufficient to give any kind of excess decay.  We need to know that, by choosing ``good'' reference cones in our blow-up sequence, we can always eliminate first-order growth in the limiting Jacobi field.

As in \cite{AllAlm} and \cite{simon1}, we require an \emph{integrability} condition on our cross sectional cones $\bC_0$, which for us simply asks that every $1$-homogeneous Jacobi field arises from a family of rotations, like the above example.

\begin{definition}\label{def:integrable}
We say a polyhedral cone ${\refC_0}^2 \subset \R^{2+k}$ is \emph{integrable} if every linear, compatible Jacobi field $v : \refC_0 \to {\refC_0}^\perp$ takes the form
\begin{gather}
v = \pi_{{\refC}^\perp} \circ A,
\end{gather}
for some skew-symmetric matrix $A : \R^{2+k} \to \R^{2+k}$.
\end{definition}


\begin{remark}\label{rem:locally-int}
By Proposition \ref{prop:baby-linear},  any $1$-homogeneous compatible Jacobi field on $\refC_0$ is linear in the sense that each component $v(i)$ is the restriction of some skew-symmetric matrix $A(i)$.  Integrability in this definition means that all the $A(i)$'s match up to generate a \emph{global} rotation of $\refC_0$.
\end{remark}

\begin{remark}\label{rem:maybe-non-int}
The polyhedral cones we are most interested in (those arising from equiangular geodesic nets in $\bbS^2$) are integrable in the sense of Definition \ref{def:integrable}.  We prove this in Section \ref{sec:nets}.

We note however that our notion of integrability is \emph{stronger} than the ``usual'' definition (of \cite{AllAlm}, \cite{simon1}), which simply requires every $1$-homogeneous Jacobi field to arise from a $1$-parameter family of stationary cones.

Indeed, although any $1$-homogeneous compatible Jacobi field on a polyhedral cone is \emph{locally} generated by a rotation (Proposition \ref{prop:baby-linear}), it seems plausible to us that there may exist $1$-parameter families of connected, equiangular geodesic nets in some $\bbS^{2+k}$ which are not \emph{global} rotations.  Of course disconnected equiangular net are trivially not integrable by our definition.

We have chosen to write this paper using rotations, but (like in \cite{simon1}), the methods carry over directly to the more general notion of integrability.  See, for example, Remark \ref{rem:blow-up-rot}.
\end{remark}

\begin{remark}\label{rem:general-non-int}
It is also not clear to us that every linear Jacobi field need arise from a $1$-parameter family of nets, being global rotations or otherwise.  That is, there may be non-integrable polyhedral cones even in the more general sense of integrable.
\end{remark}

\vspace{5mm}

It will be convenient to define a general notion of inhomogeneous blow-up sequence.  The following defines sufficient conditions to inhomogeneously blow-up $M_i$ over $\bC_i$ at scale $\beta_i$, so as to obtain a compatible Jacobi field with ``good'' properties (see Proposition \ref{prop:blow-up}).
\begin{definition}
Let $\eps_i, \beta_i$ be two sequences of numbers $\to 0$.  We say $(M_i, \bC_i, \eps_i, \beta_i)$ is a \emph{blow-up sequence w.r.t. $\refC$} if the following holds:
\begin{enumerate}
\item[A)] Each $M_i \in \cN_{\eps_i}(\bC)$, and $\bC_i \in \cC_{\eps_i}(\refC)$;

\item[B)] Each $M_i$ satisfies the $\eps_i$-no-holes condition w.r.t $\refC$ in $B_1$; 

\item[C)] We have $\limsup_i \beta_i^{-2} E_{\eps_i}(M_i, \bC_i, 0, 1) < \infty$.
\end{enumerate}
\end{definition}

\begin{remark}
This last condition ensures that we can inhomogeneously scale the graph by size $\beta_i^{-1}$, and still have uniform $C^{1,\alpha}$ and $L^2$ bounds.
\end{remark}

\begin{remark}
In many cases we will simply take $\bC_i \equiv \refC$, but allowing for slightly tilted $\bC_i$ is what enables us to kill $1$-homogeneous terms in the resulting Jacobi field.  In general it may not hold that $\limsup_i \beta_i^{-1} \eps_i < \infty$.
\end{remark}

\section{Main Theorems}\label{sec:main-thm}

Our main decay Theorem is the following.  Recall the Definition \ref{def:integrable} of integrability.
\begin{theorem}[Excess decay]\label{thm:main-decay}
Take $\refC = \refC_0^2 \times \R^m$, where $\refC_0^2 \subset \R^{2+k}$ is an integrable polyhedral cone, and take $\theta > 0$.  There are numbers $\delta(\refC, \theta)$, $c(\refC)$, $\gamma(\refC, \theta)$, $\mu(\refC)$ so that the following holds: Suppose $M^{2+m}$ is an integral varifold, with bounded mean curvature and no boundary in $B_1$, satisfying
\begin{gather}
0 \in M, \quad \mu_M(B_1) \leq \frac{3}{2}\theta_\refC(0), \quad E_\delta(M, \refC, 1) \leq \delta^2, \\
\text{and the $\delta$-no-holes condition in $B_{1/2}$ w.r.t. $\refC$.}
\end{gather}
Then we can find a rotation $q \in SO(n+k)$ with $|q - Id| \leq \gamma E_\delta(M, \refC, 1)^{1/2}$, so that
\begin{gather}
E_\delta(M, q(\refC), \theta) \leq c \, \theta^\mu E_\delta(M, \refC, 1).
\end{gather}

Note that $\mu$ and $c$ are \emph{independent} of $\theta$.  In particular, we deduce that for $\delta(\refC)$ sufficiently small, we have
\begin{gather}\label{eqn:main-decay-half}
E_\delta(M, q(\refC), \theta) \leq \frac{1}{2} E_\delta(M, \refC, 1).
\end{gather}
\end{theorem}

\begin{remark}
Though we state and prove all our results for bounded mean curvature, they continue to hold with minor modifications for integral varifolds with mean curvature in $L^p$, provided $p > n$.
\end{remark}

An important special case of Theorem \ref{thm:main-decay} is when $m = 0$, where the no-holes condition becomes simply the requirement that \emph{some} point of the correct density exists.  Note we do not assume any kind of minimizing quality to $M$.
\begin{corollary}\label{cor:no-spine-reg}
Let $\refC^2 \subset \R^{2+k}$ be an integrable polyhedral cone.  For example, suppose $\refC^2 \subset \R^3 \subset \R^{2+k}$.  There are $\delta(\refC), \mu(\refC) \in (0, 1)$ so that if $M^2 \in \cN_\delta(\refC)$ satisfies $\theta_M(0) \geq \theta_\refC(0)$, then $M \cap B_{1/2}$ is a $C^{1,\mu}$-perturbation of $\refC$.
\end{corollary}

For certain classes of varifolds we can deduce the no-holes condition whenever $M$ is sufficiently close to $\bC$ in excess.  One important way the no-holes condition arises is by imposing a boundary/orientability structure.  If $T$ is an integral $n$-current, we can write its action on $n$-forms $\omega$ as
\begin{gather}\label{eqn:current-action}
T(\omega) = \int_{M_T} <\omega, \tau_T> \theta_T \haus^n,
\end{gather}
where $M_T$ is some $n$-rectifiable set, $\theta_T$ is a positive, integer-valued, $\haus^n \llcorner M_T$-integrable function, and $\tau_T$ is a $\haus^n \llcorner M_T$-measurable choice of $n$-orientation.

\begin{definition}
Given an open set $U$, we say an integral varifold $V$ has an \emph{associated cycle structure in $U$} if there is a countable collection of integral $n$-currents $T_1, T_2,\ldots$, each without boundary in $U$, so that
\begin{gather}
\mu_V \llcorner U = (\theta\, \haus^n \llcorner  \bigcup_{i=1}^\infty M_{T_i}) \llcorner U 
\end{gather}
where $\theta$ is some \emph{positive}, integer-valued, $\haus^n \llcorner \bigcup_{i=1}^\infty M_{T_i}$-measurable function.
\end{definition}

Varifolds with cycle structure arise naturally when constructing size-minizers, clusters, and more generally $(\mass, \eps, \delta)$-minimizers.  See the following Section \ref{sec:clusters} for details.  For codimension-one varifolds having a cycle structure, we can prove the following (again we note that no minimizing property of $M$ is required).

\begin{theorem}\label{thm:spine-reg}
There are constants $\delta(m), \mu(m) \in (0, 1)$ so that the following holds.  Let $M^{2+m} \subset \R^{3+m}$ be a varifold with an associated cycle structure in $B_1$, and suppose $M \in \cN_{\delta}(\bT^2 \times \R^m)$.  Then $M \cap B_{1/2}$ is a $C^{1,\mu}$ perturbation of $\bT\times \R^m$.
\end{theorem}

\subsection{Clusters and size-minimizers}\label{sec:clusters}

The most dramatic application of our regularity Theorem is seen in (certain classes of) $(\mass,\eps, \delta)$-minimizing sets, as in this case we have a very good classification of tangent cones due to Taylor \cite{taylor}.  For cluster minimizers and size-minimizing currents we can establish $C^{1,\alpha}$-structure of the $n$, $(n-1)$, and $(n-2)$-strata, and thereby (using results of Naber-Valtorta \cite{naber-valtorta}) give finite Lipschitz structure on the $(n-3)$-stratum.

We first give some background definitions and theorems.  First we define precisely the notion of $(\mass, \eps, \delta)$-minimizing set, in the sense of Almgren.
\begin{definition}
Let $U$ be an open set, and $\eps(r) = C r^\alpha$ for some constants $C$, $\alpha > 0$.  A set $S$ is an $n$-dimensional \emph{$(\mass, \eps, \delta)$-minimizer in $U$} if the following hold:
\begin{enumerate}
\item[A)] $S = (\spt \haus^n \llcorner S) \cap U$;

\item[B)] given any ball $B_r(x) \subset U$ with $r < \delta$, and any Lipschitz map $\phi : B_r(x) \to B_r(x)$ satisfying $\spt (\phi - Id) \subset B_r(x) \cap U$, we have
\begin{gather}
\haus^n(\phi(S) \cap B_r(x)) \leq (1 + \eps(r)) \haus^n(S \cap B_r(x)).
\end{gather} 
\end{enumerate}
\end{definition}

In this paper we shall only deal with $(\mass, \eps, \delta)$-minimizers having an associated cycle structure.  There are two classes in particular we shall consider.

\subsubsection{Size-minimizers}

If $T$ is a rectifiable $n$-current, then in the notation of \eqref{eqn:current-action} the \emph{size} of $T$ is given by $\size(T) = \haus^n(M_T)$.  Given an open set $U$, we say $T$ is \emph{homologically size-minimizing in $U$} if $\size(T)\leq \size(T + S)$ for any rectifiable $n$-current $S$ supported in $U$, with $\partial S = 0$.

Given a size-minimizing current $T$ in $U$, then in $U$ its underlying multiplicity-$1$ varifold is stationary, and its support $(\mass, 0, \infty)$-minimizing.  Morgan has demonstrated the following existence Theorem for size-minimizing currents.
\begin{theorem}[\cite{morgan-size}]
Let $B$ be an $(n-1)$-dimensional compact oriented submanifold of the unit sphere in $\R^{n+1}$.  Then there exists a integral $n$-current $T$ with $\partial T = B$, which is size-minimizing in $\R^n\setminus B$.
\end{theorem}

\subsubsection{Clusters}

Given a natural number $N$, an \emph{$N$-cluster} $\clus$ in $\R^{n+1}$ is a partition of $\R^{n+1}$ of disjoint sets $\clus(0), \clus(1), \ldots, \clus(N)$ of finite-perimeter, satisfying $\clus(0) = \R^n \setminus (\clus(1) \cup \ldots \cup \clus(N))$.  Typically the sets $\clus(1), \ldots, \clus(N)$ are understood to be bounded.  We define the volume vector and perimeter scalar as (resp.)
\begin{gather}
\mass(\clus) = ( |\clus(1)|, \ldots, |\clus(i)| ) \in \R^N, \quad P(\clus) = \frac{1}{2} \sum_{i=0}^N \haus^n(\partial [\clus(i)]).
\end{gather}
Here $|\clus(i)| = \leb^{n+1}(\clus(i))$ is the $(n+1)$-volume, and $\haus^n(\partial [\clus(i)]) \equiv \mass(\partial [\clus(i)])$ is the mass of the reduced boundary.  Of course $|\clus(0)| = \infty$.

Given a volume vector $\mathbf{m} \in \R^N$, a \emph{minimizing cluster for $\mathbf{m}$} is an $N$-cluster which realizes the infimum
\begin{gather}
\inf \{ P(\clus) : \text{$\clus$ is an $N$-cluster in $\R^{n+1}$ with $\mass(\clus) = \mathbf{m}$} \}.
\end{gather}
In other words, a minimizing cluster is a solution to the isoperimetric problem of $N$ regions of prescribed volume.  Almgren proved the following existence Theorem for minimizing clusters (see also the modern presentation \cite{maggi}).
\begin{theorem}[\cite{Alm-eps-delta}]\label{thm:background-clusters}
Given any positive volume vector $\mathbf{m} \in \R^N$ (so, each $m_h > 0$), then there is a minimizing $N$-cluster for $\mathbf{m}$ enjoying the following properties:
\begin{enumerate}
\item Each set $\clus(1), \ldots, \clus(N)$ is bounded; 
\item The associated set $\partial \clus(1) \cup \ldots \cup \partial \clus(N)$ is $(\mass, Kr, \delta)$-minimizing, for some constants $K$, $\delta$; 
\item The associated varifold $\haus^n \llcorner (\partial \clus(1) \cup \ldots \cup \partial\clus(N))$ has bounded mean curvature, and no boundary.
\end{enumerate}
Here $\partial \clus(i)$ denotes the \emph{topological} boundary of $\clus(i)$.
\end{theorem}

\begin{remark}
Conclusion 3) is not explicitly stated in \cite{Alm-eps-delta}, \cite{maggi}, but follows directly from \cite[Theorem VI.2.3]{Alm-eps-delta} or \cite[Theorem IV.1.14]{maggi}.  For the reader's convenience we include a proof of part 3) in Section \ref{sec:corollaries}.
\end{remark}

\subsubsection{Interior regularity}

We prove the following general interior regularity theorem, from which Theorem \ref{thm:clusters} is an immediate consequence.
\begin{theorem}\label{thm:main-reg}
Let $M^n = \haus^n \llcorner \spt M$ be a varifold in an open set $U \subset \R^{n+1}$.  Suppose that, in $U$: $M$ has an associated cycle structure, no boundary, bounded mean curvature, and $\spt M$ is $(\mass, \eps, \delta)$-minimizing.

Then \emph{in $U$} we have the following structure:
\begin{enumerate}
\item $S^n(M) \setminus S^{n-1}(M)$ is a locally-finite union of embedded $C^{1,\alpha}$ $n$-manifolds;

\item $S^{n-1}(M) \setminus S^{n-2}(M)$ is a locally-finite union of embeded, $C^{1,\alpha}$ $(n-1)$-manifolds, near which $M$ is locally diffeomorphic to $\bY \times \R^{n-1}$;

\item $S^{n-2}(M) \setminus S^{n-3}(M)$ is a locally-finite union of embedded, $C^{1,\alpha}$ $(n-2)$-manifolds, near which $M$ is locally diffeomorphic to $\bT \times \R^{n-2}$;

\item $S^{n-3}(M)$ is relatively closed, $(n-3)$-rectifiable, with locally-finite $\haus^{n-3}$-measure.
\end{enumerate}

\end{theorem}

\begin{remark}
Theorem \ref{thm:main-reg} holds for any class of $n$-dimensional $(\mass, \eps, \delta)$-minimizers in $\R^{n+1}$, whose associated (multiplicity-one) varifolds have bounded mean curvature, and satisfy a no-holes condition like Proposition \ref{prop:no-holes}.  We could only verify this condition for minimizers having a cycle structure, but it seems plausible one could prove this for more general classes.  See Remark \ref{rem:general-no-holes}.

An obstacle to extending Theorem \ref{thm:main-reg} to higher \emph{co}dimension $\R^{n+k}$ is whether the tetrahedron $\bT^2$ continues to have the least density of any polyhedral cone in $\R^{n+k}$.
\end{remark}

\subsection{Outline of Proof}

The basic idea, which harks back to methods pioneered by De Giorgi, and implemented first more-or-less in this form by Allard-Almgren \cite{AllAlm}, is to use good decay properties of solutions to the \emph{linearized} minimal surface operator over $\bC$ (i.e. Jacobi fields), to prove decay of minimal surfaces close to $\bC$: if we write $M$ as a ``graph'' over $\bC$ by a function $u$, then as $u$ becomes very small it starts to act like a Jacobi field on $\bC$.

We will argue by contradiction.  Let us outline the proof.  For simplicity assume $H_M \equiv 0$.  If, towards a contradiction, the excess decay \eqref{eqn:main-decay-half} failed, we would have a sequence of numbers $\eps_i \to 0$, and minimal surfaces $M_i \in \cN_{\eps_i}(\bC)$, each satisfying the $\eps_i$-no-holes condition w.r.t. $\bC$ in $B_{1/2}$, so that
\begin{gather}\label{eqn:intro-hyp}
\theta^{-n-2} \int_{M \cap B_\theta} d_{q(\bC)}^2 \geq \frac{1}{2} \int_{M \cap B_1} d_{\bC}^2 =: \beta_i^2  \quad \forall q \in SO(n+k) 
\end{gather}

For any $\tau > 0$, and $i >> 1$, we can write $M_i \cap B_1 \setminus B_\tau(\axis)$ as graph over $\bC$ by the function $u_i$ (in a suitable sense, see Section \ref{sec:graph}), where $u_i \to 0$ and
\begin{gather}
\int_{\dom(u_i) \subset \bC} |u_i|^2 \leq (1+o(1))\int_{M \cap B_1} d_{\bC}^2
\end{gather}
The rescaled graphs $\beta_i^{-1} u_i$ have uniform $L^2$ and $C^{1,\alpha}$ bounds, and we can make sense of the limit $\beta^{-1}_i u_i \to v$ as a compatible Jacobi field $v : \bC \cap B_{1/2} \to \bC^\perp$ (see Section \ref{sec:blow-up}, and recall Definition \ref{def:compatible}).

We would then like to make two assertions:
\begin{enumerate}
\item For any ball $B_\rho$, with $\rho \leq 1/4$, we have strong $L^2$ convergence
\begin{gather}\label{eqn:intro-fake-1}
\beta_i^{-2} \int_{M_i \cap B_\rho} d_{\bC}^2 \to \int_{\bC \cap B_\rho} |v|^2.
\end{gather}

\item For any $\theta \leq 1/4$, we have decay
\begin{gather}\label{eqn:intro-fake-2}
\theta^{-n-2} \int_{\bC \cap B_\theta} |v|^2 \leq c(\bC)\theta^\mu \int_{\bC \cap B_{1/2}} |v|^2.
\end{gather}
\end{enumerate}
If both these claim were true, then for $i >> 1$ we could contradict \eqref{eqn:intro-hyp} with $q = id$, and $\theta(\bC)$ sufficiently small.

The first assertion is true but highly non-trivial.  The issue is the non-graphicality of $M_i$ near the spine $\axis$, where (rescaled) $L^2$ distance may accumulate in the limit.  To rule this out we prove a non-concentration estimate like in \cite{simon1} (equation \eqref{eqn:intro-non-conc}), which uses very strongly the no-holes condition.

The second assertion is in general false, even for toy examples like when $\bC$ is a plane.  While it is  true that $v$ grows \emph{at least} $1$-homogeneously (loosely a consequence of scaling), $v$ may have a non-zero $1$-homogeneous component, which would preclude an estimate like \eqref{eqn:intro-fake-2}.  The problem is partly that we may have chosen the wrong cone $\bC$ (e.g. if $M$ were smooth, we would want to pick $\bC$ to be the tangent space at $0$; see also the example of Section \ref{sec:compatible}), but a deeper issue is that, for general $\bC$ there may (and do) exist $1$-homogeneous Jacobi fields on $\bC$ that do not arise geometrically as initial velocities.

Here we use the \emph{integrability} condition on $\bC_0$, which allows us to always select ``good'' cones $q_i(\bC)$, so that if we repeat the above blow-up procedure with $q_i(\bC)$ in place of $\bC$, we can kill the $1$-homogeneous component of the limiting field (Proposition \ref{prop:kill-linear}).  We end up with a decay of the ``non-linear component'' of $v$ (Theorem \ref{thm:linear-decay}).

The ``corrected'' assertions, which still contradict \eqref{eqn:intro-hyp} for $i >> 1$, are:
\begin{enumerate}
\item If we write $v_\theta$ for the component of $v$ that is $L^2(\bC \cap B_\theta)$-orthogonal to the linear fields on $\bC$, then there is a sequence of rotations $q_i$ so that
\begin{gather}\label{eqn:intro-real-1}
\beta_i^{-2} \int_{M_i \cap B_\theta} d_{q_i(\bC)}^2 \to \int_{\bC \cap B_\theta} |v_\theta|^2,
\end{gather}

\item We have the decay
\begin{gather}\label{eqn:intro-real-2}
\theta^{-n-2} \int_{\bC \cap B_\theta} |v_\theta|^2 \leq c(\bC) \theta^\mu \int_{\bC \cap B_{1/2}} |v|^2.
\end{gather}
\end{enumerate}

Let us outline the structure of the paper, and provide some insight into each section.

\subsubsection{Graphicality}

We demonstrate in this section that when $M \in \cN_\eps(\bC)$, with $\eps(\tau, \beta, \bC)$ sufficiently small, then $M \cap B_{3/4} \setminus B_\tau(\axis)$ decomposes as graphical pieces over $\bC$, with scale-invariant $C^{1,\alpha}$ norm controlled by $\beta$.  Most importantly, we show effective estimates on both the graphical and non-graphical parts.

Precisely, we show in Lemma \ref{lem:poly-graph} the following kind of decomposition: there are domains $\Omega(i) \subset P(i) \times \R^m$, each a perturbation of the wedge $W(i)\times \R^m$, and functions $u(i) : \Omega(i) \to (P(i)\times \R^m)^\perp$, so that
\begin{gather}\label{eqn:intro-graph-1}
M(0) := M \cap B_{3/4} \setminus \bigcup_{i=1}^d u(i)(\Omega(i)) \subset B_\tau(\axis).
\end{gather}
This by itself is a straightforward contradiction argument, using Simon's and Allard's regularity Theorems, and the ``irreducibility'' of integrable polyhedral cones (see Section \ref{sec:mult-one}).

The more involved part is establishing effective \emph{global} estimates, for example
\begin{gather}\label{eqn:intro-graph-2}
\int_{M(0)} r^2 + \sum_{i=1}^d \int_{\Omega(i)} r^2 |Du(i)|^2 \leq c(\bC) \int_{M \cap B_1} d_{\bC}^2,
\end{gather}
and, if we write $f(i)$ for the functions defining $\partial\Omega(i)$ as graphs over $\partial W(i)\times \R^m$, then
\begin{gather}\label{eqn:intro-graph-3}
\sum_{i=1}^d \int_{\partial W(i)\times \R^m} r |f(i)|^2 \leq c(\bC) \int_{M\cap B_1} d_{\bC}^2.
\end{gather}

Estimates \eqref{eqn:intro-graph-2}, \eqref{eqn:intro-graph-3} are crucial in controlling density excess by $L^2$ excess (Proposition \ref{prop:density-est}).  Note the RHS is independent of $\tau$: this is because both sides scale the same way, which allows us to sum up local estimates from Allard's or Simon's regularity Theorems.

The strategy to prove these is to start with a non-effective graphical decomposition, of the form \eqref{eqn:intro-graph-1}, and then by a further contradiction argument ``push'' the region of graphicality towards the spine until either: we hit the spine (!), or a localized $L^2$ excess passes some threshold.  This is the content of Lemma \ref{lem:poly-tiny-graph}.

The scheme is similar to \cite{simon1}, but more involved, and we draw the reader's attention to two particular differences: first, the singular nature of the cross-section of the cone requires additional structure and estimates (e.g. \eqref{eqn:intro-graph-3}); second, we can remove Simon's requirement of $M$ lying in some multiplicity-one class (both in our case and his original setting when $\bC_0$ is smooth).

\subsubsection{$L^2$ estimates}

Here we prove key $L^2$ estimates on $M$ and the $u(i)$ of decomposition \eqref{eqn:intro-graph-1}, which guarantee strong $L^2$ convergence and decay of the Jacobi field (minus its linear part).  Various intermediate steps are involved, but the crucial estimates at the end of the day are the following (Theorem \ref{thm:l2-est}): provided $M \in \cN_\eps(\bC)$ satisfies the $\tau/10$-no-holes condition w.r.t. $\bC$ in $B_{1/4}$, and $\eps(\tau, \bC)$ is small, and $\alpha \in (0, 1)$, then
\begin{align}
&\sum_{i=1}^d \int_{\Omega(i) \cap B_{1/10}\setminus B_\tau(\axis)} R^{2-n} |\partial_R (u(i)/R)|^2 \leq c(\bC)\int_{M \cap B_1} d_{\bC}^2 , \label{eqn:intro-hardt-simon}\\
&and \quad \int_{M \cap B_{1/4}} \frac{d_{\bC}^2}{\max(r, \tau)^{2-\alpha}} + \sum_{i=1}^d \int_{\Omega(i) \cap B_{1/4} \setminus B_\tau(L\times \R^m))} \frac{|u(i) - \kappa^\perp|^2}{\max(r, \tau)^{2+2-\alpha}} \label{eqn:intro-non-conc}\\
&\quad\quad\quad \leq c(\bC, \alpha) \int_{M \cap B_1} d_{\bC}^2 .
\end{align}
where $L = \cup_{i=1}^{2d/4} L(i)$ are the singular lines of $\bC_0$, and $\kappa$ is a piecewise-constant function which forms a discrete approximate-parameterization of the singular set of $M$ over the spine $\axis$.

Estimate \eqref{eqn:intro-hardt-simon} says that the blow-up limit field $v$ must grow at least $1$-homogeneously in $R$, and is a key component in proving super-linear growth of $v$ minus-its-linear-part.  Estimate \eqref{eqn:intro-non-conc} is a non-concentration estimate for $L^2$ excess, and gives a growth bound on $v$ which is crucial for characterizing $1$-homogeneous fields.  Notice the RHS of both equations is independent of $\tau$.

To prove \eqref{eqn:intro-hardt-simon}, \eqref{eqn:intro-non-conc} we follow Simon's computations, but the singular nature of $\bC_0$ adds significant complications.  The most delicate estimate controls the density excess of $M$ by its $L^2$-distance to $\bC$ (Proposition \ref{prop:density-est}), which requires heavily the no-holes condition and effective graphical estimates \eqref{eqn:intro-graph-2}, \eqref{eqn:intro-graph-3}.  We additionally exploit heavily the $120^\circ$ angle condition on the geodesic net, and this highlights a technical difference between our paper and \cite{simon1}: we require stationarity of $M$ and $\bC$ \emph{through} the singular set, while Simon only requires stationarity on the regular parts.

\subsubsection{Jacobi fields}

The aim of this section is to prove an $L^2$ decay for Jacobi fields satisfying certain orthogonality and growth conditions (Theorem \ref{thm:linear-decay}).  If $\bC$ had a smooth cross-section, this would follow easily from the Fourier expansion: the discrete powers of decay would show that any $v$ growing $> 1$-homogeneously, must grow at least $(1+\eps)$-homogeneously.  In our case, we adapt the ingenious method devised in \cite{simon1} to handle cylindrical cones.

The basic idea is the following.  On the one hand, we have an upper bound \eqref{eqn:intro-hardt-simon} at any scale $\rho$, which says that $v$ grows at least $1$-homogeneously.  On the other hand, in Theorem \ref{thm:1-homo-linear} we can characterize $1$-homogeneous Jacobi fields satisfying a growth bound like \eqref{eqn:intro-non-conc} as linear (or most specifically, as lying in a subspace of the linear fields).  By a simple contradiction argument, this allows us to say that whenever $v$ is $L^2(B_\rho)$-orthogonal to the linear fields, then $v$ must grow \emph{quantitatively} more than $1$-homogeneously at scale $\rho$.  That is, 
\begin{gather}\label{eqn:intro-hardt-simon-lower}
\int_{\bC \cap B_1 \setminus B_{1/10}} R^{2-n} |\partial_R(v/R)|^2 \geq \frac{1}{c(\bC)} \int_{\bC \cap B_1} |v|^2.
\end{gather}
Chaining \eqref{eqn:intro-hardt-simon} and \eqref{eqn:intro-hardt-simon-lower} with a hole-filling gives the required decay.

Most of this section is analogous to \cite{simon1}, except care must be taken to ensure the argument works with compatible Jacobi fields.  In particular, we demonstrate in Theorem \ref{thm:eigenfunctions} a spectral decomposition for the Jacobi operator system on equiangular geodesic nets.

\subsubsection{Inhomogeneous blow-ups and conclusion of proof}

Here we make sense of inhomogeneous blow-up limits $\beta_i^{-1} u_i \to v$, and prove that the resulting $v$ is a compatible Jacobi field in the sense of Definition \ref{def:compatible}.  The $C^0$ compatibility condition arises from the sheets of $M$ meeting along a common single edge, and is essentially a direct consequence of Simon's $\eps$-regularity for the $\bY\times \R^m$.  The $C^1$ condition requires the stationarity of both $M$ and $\bC$ (\emph{through} the singular set), and depends strongly on the Remark \ref{rem:looks-like-Y} that away from $\axis$, $M$ is locally either $\R^{m+2}$ or $\bY\times \R^{m+1}$.

We then show in Proposition \ref{prop:kill-linear} how integrability allows us to choose new cones $q_i(\bC)$ (for $q_i \in SO(n+k)$) nearby $\bC$ in such a way the the limiting $v$ has no linear component at a given scale.  This allows us to prove the required estimates on $v$ to apply the linear decay Theorem \ref{thm:linear-decay}.

Finally, we can implement the blow-up argument sketched in the initial Proof Outline, to finish proving decay Theorem \ref{thm:main-decay}.

\subsubsection{Equiangular nets in $\bbS^2$}

In this section we establish some background results on nets in $\bbS^2$.  We reprove for the reader's convience the general no-holes principle for $\bY\times \R^m$, and additionally demonstrate a no-holes principle for the tetrahedral cone $\bT\times \R^m$, under natural structure assumptions on $M$.

We prove integrability (in the sense of Definition \ref{def:integrable}) for all equiangular nets in $\bbS^2$, which allows us to apply Theorem \ref{thm:main-decay} to any polyhedral cone $\bC_0 \subset \R^3 \subset \R^{2+k}$.  Unfortunately we are unable to give a general abstract proof, but must appeal to the classification of these nets due to \cite{lamarle}, \cite{heppes}.  It is possible that in general codimension there exist non-integrable equiangular nets (see also Remarks \ref{rem:locally-int}, \ref{rem:maybe-non-int}, \ref{rem:general-non-int}).

\tableofcontents \end{comment}

\section{Graphical estimates}\label{sec:graph}

In this section, and in fact for the duration of the paper, we take $\refC_0^2 \subset \R^{2+k}$ to be a fixed polyhedral cone, composed of wedges $\{W(i)\}_{i=1}^d$.  We set $\refC^n = \refC_0^2 \times \R^m$.  Using the $\eps$-regularity theorems for the plane and $\bY\times \R^{1+m}$, we prove that any $M^n$ sufficiently near $\bC$ must decompose away from the axis as $C^{1,\alpha}$ graphs with effective estimates (though of course the estimates degenerate as $r \to 0$).  Note that $c$ is independent of $\tau$.

\begin{lemma}[Effective graphicality over polyhedral cones]\label{lem:poly-graph}
For any $\beta, \tau > 0$, there is an $\eps(\refC, \beta, \tau)$ and $c(\refC, \beta)$, $\alpha(\bC)$ so that the following holds.  Take $M \in \cN_{\eps}(\refC)$.  Then there is a radius function $r_y : B_1 \cap (\{0\}\times \R^m) \to \R$ with $r_y < \tau$, so that we can decompose
\begin{gather}\label{eqn:poly-graph-1}
M \cap B_{3/4} \setminus B_{r_y}(\{0\}\times \R^m) = M(1) \cup \cdots \cup M(d),
\end{gather}
where for each $i$ there is some domain $\Omega(i) \subset P(i) \times \R^m$, and $C^{1,\alpha}$ function $u(i) : \Omega(i) \to (P(i) \times \R^m)^\perp$, such that
\begin{gather}\label{eqn:poly-graph-2}
M(i) = \{ x + u(i)(x) : x \in \Omega(i) \}.
\end{gather}

Each $\Omega(i)$ is graphical over $W(i)\times \R^m$ in the sense that there exist $C^{1,\alpha}$ functions $f(i) : \partial W(i) \times \R^m \to (\partial W(i)^\perp \cap P(i)) \times \R^m$, 
so that
\begin{gather}\label{eqn:poly-graph-3}
\partial\Omega(i) \cap B_{1/2} \setminus B_{r_y}(\axis) \subset \{ x' + f(i)(x') : x' \in \partial W(i) \times \R^m \} .
\end{gather}

Moreover the functions $u(i)$, $f(i)$ have the following pointwise estimates
\begin{align}
&\sup_ {\partial W(i)\cap (B_{1/2} \setminus B_{r_y}(\axis))}r^{-1} |f(i)| +|Df(i)| + r^\alpha[Df(i)]_{\alpha, C} \leq \beta,\label{eqn:point-est-poly-1}\\
&\sup_ {\Omega(i)\cap (B_{1/2} \setminus B_{r_y}(\axis))}r^{-1} |u(i)| +|Du(i)|+ r^\alpha [Du(i)]_{\alpha, C} \leq \beta,\label{eqn:point-est-poly-11}\\
&\sup_{\partial W(i)\cap (B_{1/2} \setminus B_{5r_y}(\axis))} r^{n+2}\bigl( r^{-1} |f(i)| + |Df(i)| + r^\alpha [Df(i)]_{\alpha, C} \bigr)^2 \leq c\, E(M, \bC, 1),  \label{eqn:point-est-poly-2}\\
&\sup_{\Omega(i)\cap (B_{1/2} \setminus B_{5r_y}(\axis))} r^{n+2}\bigl(r^{-1} |u(i)| + |Du(i)| + r^\alpha ([Du(i)]_{\alpha, C}\bigr)^2 \leq c \, E(M, \bC, 1)\,  .\label{eqn:point-est-poly-22}
\end{align}
and integral estimates
\begin{align}\label{eqn:integral-est-poly-1}
&\sum_{i=1}^d \int_{(\partial W(i) \times \R^m) \cap (B_{1/2}\setminus B_{5r_y}(\{0\}\times \R^m))} r |f(i)|^2 + \sum_{i=1}^d \int_{\Omega(i) \cap B_{1/2} \setminus B_{5r_y}(\{0\}\times \R^m)} r^2 |D u(i)|^2 \\ \label{eqn:integral-est-poly-2}
& \quad+ \sum_{i=1}^d \int_{M(i) \cap B_{10r_y}(\{0\}\times \R^m)} r^2  \leq c \, E(M, \bC, 1)\, .
\end{align}

\end{lemma}

\subsection{Multiplicity-one convergence}\label{sec:mult-one}

We will be working with a one-sided excess, and therefore must restrict our admissible class of cones to those for which one-sided closeness (so, smallness of $E$) gives regularity.  We call these ``atomic.''  This restriction can be easily avoided by considering a \emph{two-sided} excess, similar to that of \cite{Neshan}.  However, we shall see that any \emph{integrable} polyhedral cone (as per our Definition \ref{def:integrable}) is atomic, so for our purposes this is no restriction at all.

\begin{definition}
We say a cone $C_0$ is \emph{atomic} if it cannot be written as the union (i.e. varifold sum) of two non-zero stationary cones.
\end{definition}

\begin{lemma}
Any polyhedral integrable cone $\bC_0^2 \subset \R^{2+k}$ is atomic.  The cone $\bY^1 \times \R$ is atomic.
\end{lemma}

\begin{proof}
Suppose on the contrary we can write $\bC_0 = \bC_0^{(1)} + \bC_0^{(2)}$, where each $\bC^{(i)}_0$ is a non-zero stationary polyhedral cone.  The geodesic nets $\bC^{(i)} \cap \bbS^{2+k}$ are disjoint, and so we can construct a $1$-parameter family of polyhedral cones $\bC_t$ obtained by rotating $\bC_0^{(1)}$ but keeping $\bC_0^{(2)}$ fixed.  Therefore $\bC_0$ is non-integrable, since the deformation $\bC_t$ is not a \emph{global} rotation.  Atomicity of $\bY\times \R$ is obvious.
\end{proof}

The following Lemma is the reason we introduce the notion of atomicity.
\begin{lemma}\label{lem:mult-one}
Let $\bC = \bC_0^\ell\times \R^m$ be a stationary atomic cone, where $\bC_0$ is either smooth or polyhedral (in which case $\ell = 2$).  Let $M_i$ be a sequence of integral varifolds, so that $M_i \in \cN_{\eps_i}(\refC)$ with $\eps_i \to 0$.  Then $M_i \to \bC$ as varifolds with multiplicity $1$.
\end{lemma}

\begin{proof}
After taking a subsequence we have convergence on compact subsets of $B_1$ to some stationary $M$, and by our hypothesis we know $M$ is supported in $\bC$.  We claim that $M$ has constant multiplicity on each subcone of $\refC$.  If $\refC_0$ is smooth this is immediate from the constancy theorem.  If $C_0$ is polyhedral, then the constancy theorem implies $M$ has constant density on each wedge, and by stationarity of the $Y$ junction any three wedges which meet must have the same multiplicity (compare Lemma \ref{lem:vect-vs-scalar-cond}).  This proves the claim.

Therefore $M$ is stationary and supported inside $\refC$, and since $\refC$ is atomic we must have $M = p \refC$ for some integer $p$.  Since $0 \in M$ we have $p \geq 1$.  On the other hand, by our restriction $\theta_M(0, 1) \leq \frac{3}{2} \theta_{\refC}(0)$, we must have $p \leq 1$.  This proves the Lemma.
\end{proof}

\subsection{$\eps$-regularity of Allard and Simon}

Let us recall the $\eps$-regularity results for the plane and $\bY\times \R^m$.

\begin{theorem}[Allard's $\eps$-regularity for the plane \cite{All}]\label{thm:allard}
There are $\eps(n, k)$ and $\mu(n, k)$ so that the following holds.  Suppose $M^n \in \cN_{\eps}(\R^n)$.  Then there is a $C^{1,\mu}$ function $u : \Omega \subset \R^n \to \R^k$, so that
\begin{gather}
M \cap B_{3/4} = \graph_{\R^n}(u), \quad |u|_{C^{1,\mu}} \leq c(n, k) E(M, \R^n, 1)^{1/2}.
\end{gather}
\end{theorem}

In \cite{simon1} Simon proved an $\eps$-regularity theorem for cones of the form $\bY\times \R^m$.  In his original paper, Simon worked in a so-called multiplicity-one class of varifolds, but by using our Lemma \ref{lem:graph-smooth} in place of his Lemma 2.6 one can remove this hypothesis (see Appendix \ref{sec:graphical-smooth}).  A caveat: our Lemma \ref{lem:graph-smooth} is \emph{not} sufficient to remove the multiplicity-one class assumption from Simon's various structure theorems for the singular set.

\begin{theorem}[Simon's $\eps$-regularity for $\bY\times \R^m$ \cite{simon1}]\label{thm:Y-reg}
There are $\eps(m, k)$, $\mu(n, k)$ so that the following holds.  Suppose $M^{1+m} \in \cN_{\eps}(\bY^1\times \R^m)$.  Then $M\cap B_{3/4}$ is $C^{1,\mu}$-close to $\bY\times \R^m$ in the following sense: We can decompose $M\cap B_{3/4} = M(1) \cup M(2) \cup M(3)$, so that for each $i = 1, 2, 3$ there is a domain $\Omega(i) \subset Q(i)$, and a $C^{1,\mu}$ function $u(i) : \Omega(i) \to Q(i)^\perp$, so that
\begin{gather}
M(i) = \{ x + u(i)(x) : x \in \Omega(i) \}, \quad Q(i) \cap B_{1/2} \subset \Omega(i), \quad |u(i)|_{C^{1,\mu}} \leq c(m,k) E(M, \bY\times \R^m, 1)^{1/2} .
\end{gather}

Each $\Omega(i)$ is graphical over $H(i)$, in the sense that there are $C^{1,\mu}$ functions
\begin{gather}
f(i) : \partial H(i) \cap B_{3/4} \to (\partial H(i))^\perp \cap Q(i), \quad |f(i)|_{C^{1,\mu}} \leq c(m,k) E(M, \bY\times \R^m, 1)^{1/2},
\end{gather}
so that
\begin{gather}
\partial \Omega(i) \cap B_{1/2} \subset \{ x' + f(i)(x') : x' \in \partial H(i) \cap B_{3/4} \} .
\end{gather}
\end{theorem}

\begin{proof}[Proof (see \cite{simon1})]
Ensuring $\eps(k, m)$ is sufficiently small, by Lemma \ref{lem:global-graph} we have $\sing M \cap B_{3/4} \subset B_{3/4} \cap B_{1/10}(\axis)$.  Write $\eps = E(M, \bY\times \R^m, 1)$.  Now given $Z \in \sing M \cap B_{3/4}$, we have 
\begin{gather}\label{eqn:Y-small-norm}
E(M, \bY \times \R^m, Z, 1/4) \leq c \,\eps^2 .
\end{gather}
For topological reasons (see Proposition \ref{prop:no-holes-Y}), \eqref{eqn:Y-small-norm} implies $M$ must satisfy the $\delta$-no-holes condition w.r.t. $\bY \times \R^m$ for $\delta(\eps) \to 0$ as $\eps \to 0$.

Therefore, by \cite{simon1} (and our Lemma \ref{lem:graph-smooth}) we have for every $Z \in \sing M \cap B_{3/4}$ a rotation $q_Z \in SO(n+k)$, so that
\begin{gather}
\rho^{-n-2} \int_{M \cap B_\rho(Z)} d_{Z + q_Z(\bY\times \R^m)}^2 \leq c(k, m)  \rho^{2\mu} \eps^2,
\end{gather}
for some fixed $\mu = \mu(k, m) \in (0, 1)$.  In particular, for any other $W \in \sing(M) \cap B_{3/4}$, we have
\begin{gather}\label{eqn:holder-sing-Y}
|q_Z - q_W| \leq c(k, m) \eps |Z - W|^{\mu}, \quad |q_Z - Id| \leq c(k, m) \eps, \quad d(Z, \axis) \leq c(k, m)\eps.
\end{gather}
From \eqref{eqn:holder-sing-Y} we deduce that we can parameterize $\sing M  \cap B_{3/4}$ by a map $F : \axis \to \R^{1+k}\times\{0\}$ having $C^{1,\mu}$ norm bounded by $c(k,m)\eps$.  We define $f(i) := \pi_{Q(i)^T}(F)$.

On the other hand, take now an $X \in \reg M \cap B_{3/4}\cap B_{1/10}(\axis)$, and set $4\rho =  d(X, \sing M) = d(X, Z)$, where $Z \in \sing M$.  Then up to renumbering we have
\begin{gather}\label{eqn:decay-away-Y}
\rho^{-n-2} \int_{M \cap B_\rho(X)} d_{Z + q_Z(Q(1))}^2 \leq \rho^{-n-2} \int_{M \cap B_{4\rho}(Z)} d_{Z + q_Z(\bY\times \R^m)}^2 \leq c \,\eps^2 \rho^{2\mu}.
\end{gather}
Therefore, by Allard we can write $M \cap B_{\rho/2}(X)$ as a graph of $u$ over $Z + q_Z(Q(1))$ with estimates
\begin{gather}
\rho^{-1} |u| + |Du| + \rho^{\mu} [Du]_\mu \leq c(k,m) \eps \rho^\mu.
\end{gather}
Using estimates \eqref{eqn:holder-sing-Y}, we can therefore write $M \cap B_{\rho/2}(X)$ as a graph over $Q(1)$ with uniform $C^{1,\mu}$ norm bounded by $c(k,m)\eps$.  Moreover, if $\tilde q_{X} \in SO(n+k)$ is the rotation taking $Q(1)$ to the tangent space $T_X M$, then \eqref{eqn:decay-away-Y} shows that
\begin{gather}\label{eqn:holder-reg}
|q_Z - \tilde q_X| \leq c(k,m)\eps |Z - \tilde X|^\mu.
\end{gather}
Since $X$ is arbitrary, estimates \eqref{eqn:holder-sing-Y} and \eqref{eqn:holder-reg} show that $u$ extends as a $C^{1,\mu}$ function up to and including the boundary $\pi_{Q(1)}(\sing M)$.
\end{proof}

\begin{definition}\label{def:Y-graph}
For ease of notation, we will write the following to indicate $M$ decomposes as in Theorem \ref{thm:Y-reg}.
\begin{gather}
M \cap B_{3/4} = \graph_{\bY\times \R^m}(u, f, \Omega), \quad B_{1/2} \subset\Omega, \quad |u|_{C^{1,\alpha}} + |f|_{C^{1,\alpha}} \leq c(k,m) E(M, \bY\times \R^m, 1)^{1/2}.
\end{gather}
\end{definition}

\subsection{Graphicality for polyhedra}


We first prove a ``crude'' graphicality for polyhedral cones, from which we push towards the spine as far as possible.

\begin{lemma}\label{lem:poly-global-graph}
Given any $\beta, \tau > 0$,  there is an $\eps_1(\refC, \beta, \tau)$ so that the following holds.  Given $M \in \cN_{\eps_1}(\refC)$, then we can decompose
\begin{gather}
M \cap B_{3/4} \setminus B_{\tau}(\axis) = M(1) \cup \cdots \cup M(d),
\end{gather}
where for each $i$ there is some domain $\Omega(i) \subset P(i) \times \R^m$, and $C^{1,\alpha}$ function $u(i) : \Omega(i) \to (P(i) \times \R^m)^\perp$, so that
\begin{gather}
M(i) = \{ x + u(i)(x) : x \in \Omega(i) \} .
\end{gather}

Each $\Omega(i)$ is graphical over $W(i)\times \R^m$ in the sense that there are $C^{1,\alpha}$ functions $f(i) : \partial W(i) \times \R^m \to (\partial W(i)^\perp \cap P(i)) \times \R^m$, 
so that
\begin{gather}
\partial\Omega(i) \cap B_{1/2}\setminus B_{2\tau}(\axis) \subset \{ x' + f(i)(x') : x' \in \partial W(i) \times \R^m \} .
\end{gather}

Moreover, we have the pointwise estimates
\begin{align}\label{eqn:polyhedral-graph-est-1}
&\sup_{\Omega(i) \cap B_{1/2} \setminus B_{2\tau}(\axis)} r^{-1} |u(i)| + |Du(i)| + r^\alpha [Du(i)]_{\alpha, C} \leq \beta, \\ \label{eqn:polyhedral-graph-est-2}
&\sup_{(\partial W(i) \times \R^m) \cap B_{1/2} \setminus B_{2\tau}(\axis)} \quad r^{-1} |f(i)| + |Df(i)| + r^\alpha [Df(i)]_{\alpha, C} \leq \beta.
\end{align}
\end{lemma}

\begin{proof}

This is essentially a direct Corollary of Lemma \ref{lem:mult-one}, and the $\eps$-regularity of the plane and $\bY$-type cones.  If the Lemma failed, we would have a counter-example sequence $M_i$.  Passing to a subsequence, we have by Lemma \ref{lem:mult-one} multiplicity-$1$ varifold convergence $M_i \to \refC$ on compact subsets of $B_1$.

In any ball avoiding the lines $\partial W(i) \times \R^m$ we can eventually write $M_i$ as a $C^{1,\alpha}$ graph over $\bC$ by Allard's theorem, satisfying the (local, scale-invariant) estimates \eqref{eqn:polyhedral-graph-est-1}.  Similarly, in any ball centered on a line $\partial W(i) \times \R^m$, but disjoint from the axis $\axis$, we can eventually decompose $M_i$ into graphs over $\bC$ as in Theorem \ref{thm:Y-reg}, and having estimates \eqref{eqn:polyhedral-graph-est-2}.
\end{proof}

\begin{definition}\label{def:poly-graph}
For ease of notation, we write
\begin{gather}
M \cap B_{3/4} \setminus B_{\tau}(\axis) = \graph_{\bC}(u, f, \Omega), \quad B_{1/2} \setminus B_{2\tau}(\axis) \subset \Omega, \quad |u|_{C^{1,\alpha}} + |f|_{C^{1,\alpha}} \leq \beta
\end{gather}
to indicate the decomposition as in Lemma \ref{lem:poly-global-graph}.
\end{definition}

\begin{remark}\label{rem:wedge-cont}
One consequence of Lemma \ref{lem:poly-global-graph} is that the number and size of wedges for polyhedral cones is continuous under varifold convergence.
\end{remark}

For $y \in \R^m$, us define the torus
\begin{gather}
U(\rho, y, \gamma) = \{ (\xi, \eta) \in \R^{2+k}\times \R^m: (|\xi| - \rho)^2 + |\eta - y|^2 \leq \gamma \rho^2 \},
\end{gather}
and the ``halved-torus''
\begin{gather}
U_+(\rho, y, \gamma) = U(\rho, y, \gamma) \cap \{ (\xi, \eta) : |\xi| \geq \rho \}.
\end{gather}

The following Lemma gives us a criterion to decide how close to the spine we can push graphicality, and is the key to integral estimates \eqref{eqn:integral-est-poly-1}.  The graphicality assumption in the half-torus allows us to avoid working in a multiplicity-1 class.

\begin{lemma}\label{lem:poly-tiny-graph}
For any $\beta > 0$ there is an $\eps_2(\refC, \beta)$ so that the following holds.  Take $M^n \in \cN_{1/10}(\bC)$.  Pick $\rho \leq 1/16$, and $y \in B_{3/4}^m$.  Suppose we know
\begin{equation}\label{eqn:poly-tiny-u}
M \cap U_+(\rho, y, 1/16) \subset \graph_{\bC}(u,\Omega, f), \quad |u|_{C^{1,\alpha}} + |f|_{C^{1,\alpha}} \leq 1/10, 
\end{equation}
where $\graph_\bC(u, \Omega, f)$ is a decomposition as in Lemma \ref{lem:poly-global-graph}, and
\begin{equation}\label{eqn:poly-tiny-dist}
\rho^{-n-2} \int_{M \cap U(\rho, y, 1/4)} d_{\bC}^2 + \rho||H_M||_{L^\infty(U(\rho, y, 1/4))} \leq \eps_2 .
\end{equation}
Then we have
\begin{gather}\label{eqn:poly-tiny-concl}
M \cap U(\rho, y, 1/8) \subset \graph_{\bC}(u,\Omega, f), \quad |u|_{C^{1,\alpha}} + |f|_{C^{1,\alpha}} \leq \beta.
\end{gather}
\end{lemma}

\begin{proof}
By dilation invariance and monotonicity, we see there is no loss in supposing $\rho = 1/2$.  Suppose the Lemma is false, and consider a counterexample sequence $M_i$, $y_i$, $\eps_i \to 0$, which satisfy the hypothesis of the Lemma and $M_i \cap U(\rho, y_i, 1/8) \neq \emptyset$, but fail \eqref{eqn:poly-tiny-concl}.  Passing to a subsequence, the $y_i \to y \in B_{3/4}^m$, and in $U(\rho, y, 1/5)$ the $M_i$'s converge to a stationary varifold supported in $\refC$.  The multiplicity of the limit in each component of $\refC \cap U(\rho, y, 1/5)$ is constant, but by the graphicality assumption we converge with multiplicity one inside $U_+(\rho, y, 1/16)$.

Therefore the convergence is with multiplicity $1$, and by Theorems \ref{thm:allard}, \ref{thm:Y-reg} we deduce that for $i >> 1$ we satisfy the conclusions of the Lemma.
\end{proof}

Using Lemma \ref{lem:poly-tiny-graph}, we can obtain the finer graphical estimates of Lemma \ref{lem:poly-graph}.

\begin{proof}[Proof of Lemma \ref{lem:poly-graph}]
We can assume $\beta \leq 1/10$.  Ensure $\eps \leq \min\{ \eps_1(\refC, \beta, \tau), \eps_2(\refC, \beta)\}$, the constants from Lemmas \ref{lem:poly-global-graph}, \ref{lem:poly-tiny-graph}.  Recalling Definition \ref{def:poly-graph}, we have the crude decomposition
\begin{gather}
M \cap B_{3/4} \setminus B_{\tau}(\axis) = \graph_{\refC}(u, f, \Omega), \quad B_{1/2} \setminus B_{2\tau}(\axis) \subset \Omega, \quad |u|_{C^{1,\alpha}} + |f|_{C^{1,\alpha}} \leq \beta .
\end{gather}

Given $y \in B_{3/4}^m$, define
\begin{gather}
r_y = \inf \{ r' : \text{\eqref{eqn:poly-tiny-u} holds for all $r' < \rho < 3/4$} \}.
\end{gather}
According to this definition, we can extend $\Omega(i)$, $u(i)$, and $f(i)$ to obtain the decomposition of \eqref{eqn:poly-graph-1}, \eqref{eqn:poly-graph-2}, \eqref{eqn:poly-graph-3}, with estimates \eqref{eqn:point-est-poly-1}, \eqref{eqn:point-est-poly-11}.  Moreoever, by Lemma \ref{lem:poly-global-graph} $r_y \leq \tau$.  If $r_y > 0$, then necessarily by Lemma \ref{lem:poly-tiny-graph}, \eqref{eqn:poly-tiny-dist} must fail at $r_y$, and hence
\begin{gather}
r_y^{n+2} \eps_2 \leq \int_{M \cap U(r_y, y, 1/4)} d_{\refC}^2 + r_y^{n+3} ||H_M||_{L^\infty(U(r_y, y, 1/4))}.
\end{gather}
In particular, by monotonicity we have
\begin{gather}
\int_{M \cap B_{20 r_y}(0, y)} r^2 \leq c(\refC, \beta) \int_{M \cap U(r_y, y, 1/4)} d_{\refC}^2 + c(\bC, \beta) r_y^{n+3} ||H||_{L^\infty(B_1)}.
\end{gather}

Take a Vitali subcover $\{B_{2\rho_j}(0, y_j)\}_j$ of
\begin{gather}
\{B_{10r_y}(0, y) : y \in B_{3/4}^m \text{ and } r_y > 0 \},
\end{gather}
and then by construction $\{B_{10\rho_j}(0, y_j)\}_j$ covers $\mu_M$-a.e. $B_{3/4} \cap B_{10r_y}(\axis)$, and the $U(\rho_j, y_j, 1/4) \subset B_{2\rho_j}(0, y_j)$ are disjoint.  Note that, by disjointness, $\sum_j \rho_j^m \leq c(m)$.  We deduce that
\begin{align}
\int_{M \cap B_{3/4} \cap B_{10r_y}(\axis)} r^2 
&\leq \sum_j \int_{M \cap B_{20\rho_j}(0, y_j)} r^2 \\
&\leq \sum_j c\int_{M \cap U(\rho_j, y_j, 1/4)} d_{\refC}^2 + \sum_j c \rho_j^{n+3} ||H||_{L^\infty(B_1)}\label{eqn:poly-global-graph-vitali} \\
& \leq c(\refC, \beta) \int_{M \cap B_1} d_{\refC}^2 + c(\refC, \beta) ||H||_{L^\infty(B_1)}
\end{align}

We prove the additional pointwise and $L^2$ estimates.  We claim that
\begin{gather}
(x, y) \not\in B_{2r_y}(\axis) \implies B_{|x|/2} \cap B_{r_y}(\axis) = \emptyset.
\end{gather}
Otherwise, suppose the latter intersection is non-empty, and contains some $(x', y')$.  Then we have
\begin{gather}
|x|/2 \leq |x'| \leq 3|x|/2, \quad \text{ and } (x', y') \in B_{r_{y''}}(0, y'') \text{ for some $y''$},
\end{gather}
which implies that
\begin{gather}
|(x, y) - (0, y'')| < |x|/2 + r_{y''} \leq |x'| + r_{y''} < 2r_{y''} ,
\end{gather}
a contradiction.  This proves the claim.

Define $\delta(\refC)$ by
\begin{gather}
\delta = 1/100 \cdot \text{(smallest geodesic length in $\refC_0\cap \bbS^{1+k}$)} \leq 1/20,
\end{gather}
so that $B_{10\delta|x|}(x, y) \cap B_{10\delta |x'|}(x', y') = \emptyset$ whenever $(x, y)$ and $(x', y')$ belong to different triple junctions in $\bC$.

Now provided $\beta(m, k, \delta)$ is sufficiently small, given any $(x, y) = (x' + u(i)(x', y), y) \in \sing M \setminus B_{2r_y}(\axis)$, we can use the above claim and Simon's regularity at scale $B_{10\delta |x|}(x, y)$ to deduce
\begin{align}\label{eqn:poly-graph-6}
\sup_{B_{5\delta |x|}(x, y)} |x|^{n+2}|Du(i)|^2 + |x|^{n}|f(i)|^2 
&\leq c \int_{B_{10\delta |x|}(x, y) \cap M} d_{\refC}^2 + c|x|^{n+3} ||H_M||_{L^\infty(B_1)} \\
&\leq c(\refC, \beta) E(M, \refC, 1).
\end{align}
Of course, on the LHS we could put any of the $C^{1,\alpha}$ estimates for $u$ or $f$ from Theorem \ref{thm:Y-reg}, normalized to scale like $|x|^{n+2}$.

Let $\{(x_j, y_j)\}_j$ be the centers of a Vitali cover of
\begin{gather}\label{eqn:poly-graph-5}
\{ B_{\delta |x|}(x, y)  : (x, y) \in \sing M \cap B_{1/2} \setminus B_{2r_y}(\axis) \}.
\end{gather}
Then the balls $\{B_{5\delta|x_j|}(x_j, y_j)\}_j$ cover \eqref{eqn:poly-graph-5}, and have overlap bounded by a universal constant $c(n)$.  In particular, since $\beta \leq 1/10$ the number of $5\delta|x_i|$-balls meeting $M \cap \{ |x| = r \} \cap B_1$ is bounded by $c(\refC) r^{-m}$.

We can we can sum up the estimates \eqref{eqn:poly-graph-6} to obtain
\begin{align}
\sum_{i=1}^d \int_{(\partial W(i) \times \R^m) \cap (B_{1/2} \setminus B_{5r_y}(\axis))} r |f(i)|^2 
&\leq c \sum_j \left( \int_{M \cap B_{5\delta|x_j|}(x_j, y_j)} d_{\bC}^2 + |x_j|^{n+3} ||H_M||_{L^\infty(B_1)}  \right) \\
&\leq c \int_{M \cap B_1} d_{\bC}^2 + c \int_0^1 r^{5} \frac{dr}{r} ||H||_{L^\infty(B_1)} \\
&\leq c(\bC, \beta) E(M, \bC, 1) .
\end{align}

If instead $(x, y) \in M \setminus (B_{2r_y}(\axis)$ and $d((x, y), \sing M) \geq \delta|x|/5$, then we can apply Allard  to deduce
\begin{gather}\label{eqn:poly-graph-4}
\sup_{B_{\delta |x|/10}(x, y)} |x|^{n+2} |Du(i)|^2  \leq c( \refC, \beta) \int_{B_{\delta |x|/5}(x, y) \cap M} d_{\refC}^2 + c(\bC) |x|^{n+3} ||H_M||_{L^\infty(B_1)}.
\end{gather}

By taking a Vitali cover of
\begin{gather}
\{ B_{\delta |x|/10}(x, y) : (x, y) \in M \setminus B_{2r_y}(\axis) \text{ and } d((x, y), \sing M) \geq \delta |x|/5 \},
\end{gather}
we can use both Vitali covers to sum up estimates \eqref{eqn:poly-graph-6} and \eqref{eqn:poly-graph-4} as before to obtain
\begin{gather}
\sum_{i=1}^d \int_{\Omega(i) \cap (B_{1/2} \setminus B_{5r_y}(\axis))} r^2 |Du(i)|^2 \leq c(\refC, \beta) E(M, \bC, 1). 
\end{gather}

This proves the $L^2$ estimates \eqref{eqn:integral-est-poly-1}, \eqref{eqn:integral-est-poly-2}.  The pointwise estimates \eqref{eqn:point-est-poly-2}, \eqref{eqn:point-est-poly-22} follow directly from from Simon's and Allard's regularity theorems as in \eqref{eqn:poly-graph-6}, \eqref{eqn:poly-graph-4}.
\end{proof}


 \end{comment}

\section{$L^2$ estimates}\label{sec:estimates}



We demonstrate various decay and growth quantities are controlled at the scale of excess.  We require first a 

\begin{definition}\label{def:chunky}
Set $r_\nu = 2^{-\nu}$, where $\nu \in \{0, 1, 2, \ldots \}$.  For each $\nu$, let us choose and fix (for the duration of this paper) a covering of $\R^m$ by disjoint, half-open cubes $\{Q_{\nu \mu}\}_\mu$, each having side length $r_\nu$.

We say $f(r, y) : \R_+ \times \R^m \to \R$ is \emph{chunky} if $f$ is constant on each annular cylinder $[r_{\nu+1}, r_\nu) \times Q_{\nu\mu}$.
\end{definition}

We introduce this class of functions is because of the following compactness result.
\begin{lemma}
Let $\{\kappa_i\}$ be a sequence of chunky functions with $||\kappa_i||_{L^\infty(U)} \leq c(U)$ for all $U \subset\subset \R_+\times \R^m$.  Then we can find a chunky function $\kappa$, admitting the same bounds $||\kappa||_{L^\infty(U)} \leq c(U)$, and a subsequence $i'$, so that  $\kappa_{i'} \to \kappa$ pointwise, and uniformly on compact sets.
\end{lemma}

\begin{proof}
Obvious.
\end{proof}

We work towards the following Theorem.  As before we fix a polyhedral cone $\refC_0^2 \subset \R^{2+k}$, composed of wedges $\{W(i)\}_{i=1}^d$ and lines $\{L(i)\}_{i=1}^{2d/3}$.  We take $\refC = \refC_0 \times \R^m$.  Recall the Definition \ref{def:no-holes} of the ``no-holes'' condition.
\begin{theorem}\label{thm:l2-est}
For any $\tau > 0$ and $\alpha \in (0, 1)$, there is an $\eps(\refC, \tau)$ so that the following holds.  Let $M \in \cN_{\eps}(\refC)$ and decompose
\begin{gather}
M \cap B_{3/4} \setminus B_{\tau}(\axis) = \graph_{\bC}(u, f, \Omega)
\end{gather}
as in Definition \ref{def:poly-graph}/Lemma \ref{lem:poly-global-graph}.

Then provided $\theta_M(0) \geq \theta_{\bC}(0)$, the following decay/growth estimates hold:
\begin{align}\label{eqn:thm-point-est}
&\int_{M \cap B_1} \frac{d_{\bC}^2}{R^{n+2-\alpha}} + \sum_{i=1}^d \int_{\Omega(i) \cap B_{1/10} \setminus B_\tau(\axis)} R^{2-n} |\partial_R (u(i)/R)|^2  + \int_{M \cap B_{1/10}} \frac{|X^\perp|^2}{R^{n+2}} \\
&\quad \leq c(\refC, \alpha) E(M, \bC, 1), 
\end{align}
where $X = \pi_{N_XM}(X)$ is the projection to the normal bundle of $M$.

Provided $M$ satisfies the $\tau/10$-no-holes condition w.r.t. $\refC$ in $B_{1/4}$, then we have decay along the spine:
\begin{align}\label{eqn:thm-spine-est}
&\int_{M \cap B_{1/4}} \frac{d_{\bC}^2}{\max(r, \tau)^{2-\alpha}} + \sum_{i=1}^d \int_{\Omega(i) \cap B_{1/4} \setminus B_\tau(L\times \R^m))} \frac{|u(i) - \kappa^\perp|^2}{\max(r, \tau)^{2+2-\alpha}} \\
&\quad \leq c(\refC, \alpha) E(M, \bC, 1) ,
\end{align}
where we write $L = \cup_{i=1}^{2d/3} L(i)$ for the lines of $\bC_0$, and $\kappa : (0, 1] \times B_1^m \to \R^{2+k} \times \{0\}$ is a chunky function admitting the bound
\begin{gather}\label{eqn:thm-k-est}
\sup_{(0, 1]\times B_1^m} |\kappa|^2 \leq c(\refC, \alpha) E(M, \bC, 1). 
\end{gather}
\end{theorem}

Let us give a brief
\begin{proof}[Outline of proof]
We will first show that testing the stationarity of $M$ with a radial vector field proportional to $d_{\bC}^2$, the following estimate holds
\begin{gather}\label{eqn:outline-decay}
\int_{M \cap B_1} \frac{d_{\bC}^2}{R^{n+2-\alpha}} \leq c(\bC, \alpha) E(M,\bC, 1) +c(\bC, \alpha) \int_{M \cap B_{1/10}} \frac{|X^\perp|^2}{R^{n+2}} . 
\end{gather}
Due to the fact that $\partial_R( X \equiv (x', y) + u(i)(x', y))$ is tangent to $M$, we also have a pointwise inequality
\begin{gather}\label{eqn:outline-hardt-simon}
|\partial_R(u(i)/R)|^2 \leq 2 |X^\perp|^2/R^4
\end{gather}
on each $\Omega(i) \cap B_{1/10}$.

Next, using a cylindrical vector field of the form $\phi^2(R) (x, 0)$, and the effective graphical estimates of Lemma \ref{lem:poly-graph}, we will show that whenever $\theta_M(0) \geq \theta_{\refC}(0)$, we have
\begin{gather}\label{eqn:outline-density}
\int_{M \cap B_{1/10}} \frac{|X^\perp|^2}{R^{n+2}} \leq c(\refC) E(M, \bC, 1).
\end{gather}
This estimate controlling density excess by $L^2$ excess is very important, and is by far the most involved.  Combining \eqref{eqn:outline-density} with \eqref{eqn:outline-decay}, \eqref{eqn:outline-hardt-simon} gives estimate \eqref{eqn:thm-point-est}.

To obtain the estimates \eqref{eqn:thm-spine-est}, we can apply \eqref{eqn:thm-point-est} at each singular point $Z$ satisfying $\theta_M(Z) \geq \theta_\bC(0)$, which by assumption form $\tau/10$-dense set in a $\tau$-neighborhood of the spine.  Now sum all these estimates up over cubes centered in $\axis$.
\end{proof}

\subsection{Decay estimate} We bound the decay and growth rates in terms of $L^2$ distance and density drop.  We first need a helper Lemma, which says we can find a good $C^1$ approximation to the distance function to our polyhedral cone.

\begin{lemma}\label{lem:smoothing-d}
We can find a $1$-homogeneous function $\tilde d$, which is $C^1$ on $\{\tilde d > 0\}$, and satisfies
\begin{equation}\label{eqn:tilde-d}
\frac{1}{c(\refC)} d_{\bC} \leq \tilde d \leq c(\refC) d_{\bC}, \quad |D\tilde d| \leq c(\refC).
\end{equation}
\end{lemma}

\begin{proof}
We first consider a $1$-dimensional stationary cone $\tilde \bC_0^1 \subset \R^{1+k}$, so that $\partial \tilde \bC_0 \cap \bbS^{k}$ is a finite collection of points.  By smoothing $d_{\tilde \bC_0}/|x|$ at the approrpriate cut-locii, we can easily obtain a $\phi : \bbS^{k} \to \R$ so that $\phi$ is $C^1$ on $\{ \phi > 0 \}$, and
\begin{gather}
\frac{1}{2}  \frac{d_{\bC_0}}{|x|} \leq \phi(x/|x|) \leq 2 \frac{d_{\bC_0}}{|x|} , \quad \text{ and } \quad |D\phi| \leq 10.
\end{gather}

Now we consider the polyhedral cone $\bC_0^2 \subset \R^{2+k}$.  Recall that by \cite{AllAlm2}, $\bC_0 \cap \bbS^{1+k}$ consists of finitely many geodesic arcs.  By applying the previous paragraph to a small neighborhood of every vertex, we can construct a function $\psi : \bbS^{1+k} \to \R$ which satisfies:
\begin{gather}
\frac{1}{c(\refC)} \frac{d_{\bC_0}}{|x|} \leq \psi(x/|x|) \leq c(\refC) \frac{d_{\bC_0}}{|x|}, \quad \psi \text{ is $C^1$ on } \{ \psi > 0 \}, \quad |D\psi| \leq c(\refC).
\end{gather}
Now set $\tilde d(x, y) = |x|\psi(x/|x|)$.
\end{proof}

\begin{prop}\label{prop:decay-growth-est}
For any $\tau > 0$ and $\alpha \in (0, 1)$, there is an $\eps(\refC, \tau)$ so that the following holds.  Let $M \in \cN_{\eps}(\refC)$, and let us decompose
\begin{gather}
M \cap B_{3/4} \setminus B_\tau(\axis) = \cup_{i=1}^d M(i) = \graph_{\bC}(u, f, \Omega)
\end{gather}
as in Lemma \ref{lem:poly-global-graph}.  Then we have
\begin{align}
&\int_{M \cap B_1} \frac{d_{\bC}^2}{R^{n+2-\alpha}} + \sum_{i=1}^d \int_{\Omega(i) \cap B_{1/10} \setminus B_\tau(\axis)} R^{2-n} |\partial_R (u(i)/R)|^2 \\
&\quad \leq c(\refC, \alpha) \int_{M \cap B_{1/10}} \frac{|X^\perp|^2}{R^{n+2}} +  c(\refC, \alpha) E(M, \bC, 1). 
\end{align}
Here $X^\perp = \pi_{N_XM}(X)$ is the projection to the normal boundle of $M$.
\end{prop}

\begin{proof}
Let $\zeta$ be a smoothing of the function which is $\equiv 1$ on $[\delta, 1/20]$, $0$ at $0$ and $\equiv 0$ on $B_1 \setminus B_{1/10}$, and linearly interpolates in between.  Consider the vector field
\begin{gather}
V(X) = \zeta^2 (\tilde d/R)^2 R^{-n+\alpha} X,
\end{gather}
where $\tilde d$ as in Lemma \ref{lem:smoothing-d}.

Since $(\tilde d/R)^2$ is $C^1$ and homogeneous degree-$0$, we have
\begin{gather}
X \cdot D (\tilde d/R)^2 = 0.
\end{gather}
We also have $|D(\tilde d/R)| \leq c(\refC)/R$.  We therefore calculate
\begin{align}
div(V) &\geq -2\zeta |\nabla^T\zeta| (\tilde d/R)^2 R^{-n+1+\alpha} -2\zeta^2 (\tilde d/R) |\nabla^\perp (\tilde d / R)| |(x, y)^\perp| \\
&\quad + \zeta^2(\tilde d/R)^2 (n-\alpha) R^{-n-2+\alpha} |(x, y)^\perp|^2 + \zeta^2 (\tilde d/R)^2 R^{-n+\alpha} \alpha .
\end{align}
And so, using \eqref{eqn:tilde-d}, we have for any $\eta$
\begin{align}
\frac{\alpha}{c(\refC)} \int_M \zeta^2 (d_{\bC}/R)^2 R^{-n+\alpha} 
&\leq \int_M |V| |H_M| +  \int_M \eta \zeta^2 (d_{\bC}/R)^2 R^{-n+\alpha} + c(\eta, \refC) \zeta^2 |(x, y)^\perp|^2 R^{-n-2+\alpha} \\
&\quad + \int_M \eta \zeta^2 (d_{\bC}/R)^2 R^{-n+\alpha} + c(\eta, \refC) d_{\bC}^2 R^{-n+\alpha} |\nabla^T \zeta|^2 .
\end{align}
Take $\eta(\refC, \alpha)$ sufficiently small, and use that $|V| \leq R^{-n+1}$ in a computation similar to \eqref{eqn:integrable-on-M}.  We obtain
\begin{gather}
\int_{M \cap B_1} d^2_{\bC} R^{-n-2+\alpha} \leq c ||H_M||_{L^\infty(B_1)} + c \int_{M \cap B_{1/10}} |(x, y)^\perp|^2 R^{-n-2+\alpha} + c \int_M d_{\bC}^2 R^{-n+\alpha} |\nabla^T \zeta|^2,
\end{gather}
for $c = c(\alpha, \bC)$.

We analyze the last term:
\begin{align}
\int_M d_{\bC}^2 R^{-n+\alpha} (\zeta')^2 |(x, y)^T|^2/R^2
&\leq c\int_{M \cap B_\delta} d_{\bC}^2 R^{-n+\alpha} \delta^{-2} + c \int_{M \cap (B_{1/10}\setminus B_{1/20})} d_{\bC}^2 R^{-n+\alpha} \\
&\leq \int_{M \cap B_\delta} R^{-n+\alpha} + c 100^{\alpha-n} \int_{M \cap (B_1 \setminus B_{1/20})} d_{\bC}^2 ,
\end{align}
where $c$ is an absolute constant.  Using the standard layer cake formula, and the monotonicity $\theta_M(0, r) \leq c(\refC) r^n$, we have
\begin{align}\label{eqn:integrable-on-M}
\int_{M \cap B_\delta} R^{-n+\alpha}
&\leq \int_0^\infty \haus^n( M \cap B_\delta \cap \{r^{-n+\alpha} > t\} ) dt \\
&\leq c(\refC) \int_{\delta^{\alpha-n}}^\infty t^{-n/(n-\alpha)} dt \\
&= c(\alpha, \refC) \delta^\alpha \\
& \to 0 \quad \text{ as } \delta \to 0.
\end{align}
This proves the first inequality.

We now prove the second.  It will suffice to prove the pointwise bound
\begin{gather}
|\partial_R(u(i)/R)|^2 \leq 2 |X^\perp|^2/R^4.
\end{gather}
Let us write $(x, y) \in M(i)$ as $(x, y) = (x', y) + u(i)(x', y)$, for $(x', y) \in \Omega(i) \cap B_{1/10} \setminus B_\tau(\axis)$.  For ease of notation let us drop the $i$ index from now on.

We compute
\begin{gather}
\partial_R (u(x', y)/R) = \partial_R \left( \frac{ (x', y) + u(x', y))}{R} \right) = \frac{(x', y) + (x', y) \cdot Du(x', y)}{R^2} - \frac{(x, y)}{R^2},
\end{gather}
and observe that $\partial_R ( (x', y) + u(i)(x', y)) \in T_{(x, y)}M$.  Therefore, we deduce
\begin{gather}
\pi_{N_{(x, y)}M}\left( \partial_R( u(x', y)/R ) \right) = - \frac{\pi_{N_{(x, y)}M}(x, y)}{R^2}.
\end{gather}

This is the required expression, but only for the normal component.  Since $\partial_R (u(x', y)/R) \in T_{(x', y)} \bC$, and $M$ is $C^1$-close to $\bC$, we can show the tangential component is negligable:
\begin{align}
|\partial_R(u(x', y)/R)|^2
&= |\pi_{N_{(x, y)}M}(\partial_R(u(x', y)/R))|^2 + |(\pi_{T_{(x, y)}M} - \pi_{T_{(x', y)}\bC})(\partial_R(u(x', y)/R))|^2 \\
&\leq |\pi_{N_{(x, y)}M}(x, y)|^2/R^4 + c(n, k) |Du(x', y)|^2 |\partial_R(u(x', y)/R)|^2 \\
&\leq |\pi_{N_{(x, y)}M}(x, y)|^2/R^4 + \frac{1}{2} |\partial_R(u(x', y)/R)|^2,
\end{align}
provided $\eps(n, k)$ is sufficiently small.
\end{proof}

\subsection{Density drop} By far the trickiest and most important estimate is estimating the density drop in terms of $L^2$ distance.  We are in effect bounding $W^{1,2}$ by $L^2$.

\begin{proposition}\label{prop:density-est}
There is an $\eps(\refC)$ so that the following holds.  Let $M \in \cN_{\eps}(\refC)$, and suppose $\theta_M(0) \geq \theta_{\refC}(0)$.  Then we have
\begin{gather}
\int_{M \cap B_{1/10}} \frac{|X^\perp|^2}{R^{n+2}} \leq c(\refC) E(M, \bC, 1) .
\end{gather}
\end{proposition}

\begin{proof}
Let $\phi$ be any smooth function, with $\phi' \leq 0$, $\phi = 1$ on $[0, 1/10]$, and $\phi = 0$ on $[2/10, \infty)$.  By the first variation formula, the structure $\bC = \bC_0 \times \R^m$, and our assumption that $\theta_M(0) \geq \theta_{\refC}(0)$, we have the following inequalities:
\begin{gather}\label{eqn:lem-density-1}
\frac{1}{2} n (1/10)^n \int_{M \cap B_{1/10}} \frac{|X^\perp|^2}{R^{n+2}} \leq \int_M \phi^2(R) - \int_{\bC} \phi^2(R) + c(\bC, \phi) ||H_M||_{L^\infty(B_1)},
\end{gather}
and
\begin{align}\label{eqn:lem-density-2}
\ell \left( \int_M \phi^2(R) - \int_{\bC} \phi^2(R) \right) &\leq \left( \int_M 2\phi |\phi'| r^2/R - \int_{\bC} 2\phi |\phi'| r^2/R \right) \\
&\quad  + \int_M 2\phi (\phi')^2 |(x, 0)^\perp|^2 + c(\bC, \phi) ||H_M||_{L^\infty(B_1)}.
\end{align}
See \cite[pages 613-615]{simon1} for a derivation; relations \eqref{eqn:lem-density-1}, \eqref{eqn:lem-density-2} require no special structure on $\bC_0$.  Note that \cite{simon1} proves \eqref{eqn:lem-density-1}, \eqref{eqn:lem-density-2} for stationary surfaces, but the modification for bounded mean curvature is straightforward -- for completeness we provide a brief sketch in Section \ref{sec:variation}.

The Proposition will now follow by Lemma \ref{lem:density-workhouse}, because if we write
\begin{gather}
F(x, y) \equiv F(R) = \phi(R) |\phi'(R)| /R,
\end{gather}
then on $\spt F$ we have that $F$ is a smooth function of $x, y$.
\end{proof}

\begin{lemma}\label{lem:density-workhouse}
There is an $\eps(\refC)$ so that the following holds.  Let $M \in \cN_{\eps}(\refC)$, and let $F(x, y) \equiv F(R)$ be a non-negative $C^1$ function supported on $R \in [1/10, 2/10]$.

Then we have
\begin{equation}\label{eqn:F-diff-est}
\int_M r^2 F - \int_{\bC} r^2 F \leq c(\refC, |F|_{C^1}) E(M, \bC, 1). 
\end{equation}

And relatedly, we have
\begin{equation}\label{eqn:F-single-est}
\int_{M \cap B_{1/10}} |(x, 0)^\perp|^2 \leq c(\refC) E(M, \bC, 1). 
\end{equation}
\end{lemma}

\begin{proof}
Let us prove \eqref{eqn:F-diff-est}.  Choosing $\eps$ sufficiently small, we have that $M \cap B_{3/4} \setminus B_{r_y}(\axis)$ decomposes as graphs over $\refC$ as in Lemma \ref{lem:poly-graph}, with $r_x \leq 1/100$ and $\beta \leq 1/10$.  So we have
\begin{align}
&\int_M F r^2 - \int_{\bC} F r^2  \nonumber\\
&\leq \sum_{i=1}^d \int_{\Omega(i) \cap (B_{1/2} \setminus B_{2r_y}(\axis))} F( (x' + u(i), y) |x' + u(i)|^2 Ju(i) - \sum_i \int_{W(i)}  F(x', y) |x'|^2 \nonumber \\
&\quad + \sum_{i=1}^d \int_{M(i) \cap B_{10r_y}(\axis)} |F|_{C^0} r^2 \label{eqn:diff-F-1} .
\end{align}

Since each domain $\Omega(i)$ is flat, each Jacobian is bounded by
\begin{gather}
Ju(i) \leq 1 + c |Du(i)|^2 \leq c.
\end{gather}
Further, since $u(i)(x', y) \in N_{(x', y)} \bC$, we have
\begin{align}
&F(x' + u(i), y) |x' + u(i)|^2 - F(x', y) |x'|^2  \\
&\leq |F|_{C^0} |u(i)|^2 + |F|_{C^1} \left( \sqrt{|x'|^2 + |u(i)|^2 + |y|^2} - \sqrt{|x'|^2 + |y|^2}\right) |x'|^2 \\
&\leq c|F|_{C^1} |u(i)|^2.
\end{align}

Using the above calculations, and $|u(i)(x', y)| \leq |x'|/10$, we have
\begin{align}
&\sum_{i=1}^d \int_{\Omega(i) \cap (B_{1/2} \setminus B_{2r_y}(\axis))} F(x' + u(i), y) |x' + u(i)|^2 Ju(i) \nonumber \\
&\leq c |F|_{C^1} \sum_{i=1}^d \int_{\Omega(i) \cap (B_{1/2} \setminus B_{2r_y}(\axis))} r^2|Du(i)|^2 + |u(i)|^2  + F(x', y). \label{eqn:diff-F-2}
\end{align}

Now by construction, if $x' \in \partial W(i)$, then
\begin{equation}\label{eqn:n-f-relation}
n(i)(x') \cdot f(i)(x') = n(i)(x') \cdot (z - x'),
\end{equation}
where $z \in \partial M(i)$.  In particular, if $W(1)$, $W(2)$, $W(3)$ all share a common edge, and $x$ lies in this edge, then
\begin{gather}
\sum_{i=1}^3 n(i)(x') \cdot f(i)(x') = 0.
\end{gather}
This follows simply because the $M(1)$, $M(2)$, $M(3)$ all share a common edge (and in particular $z$ in \eqref{eqn:n-f-relation} is independent of $i = 1, 2, 3$), and $\sum_{i=1}^3 n(i)(x') = 0$.

Let us fix a $y$, and recall that the annular region $A_W(r_y, 1/4)$ satisfies
\begin{equation}\label{eqn:A-inclusion}
(B_{1/4} \setminus B_{2r_y}) \cap W \subset A_W(r_y, 1/4) \subset (B_{1/2} \setminus B_{r_y}) \cap W.
\end{equation}
Let us write $\Omega_y(i) \equiv \Omega(i) \cap (\R^{\ell+k}\times \{y\})$, so that $\Omega_y(i)$ is a $2$-dimensional approximate-wedge.

By a similar argument as above, if $x' = s + f(i)(s) \in \partial\Omega_y(i)$ for $s \in \partial W(i)$, then
\begin{align}
&F(s + tf(i), y) |s + tf(i)|^2 - F(s, y) s^2 \\
&\leq |F|_{C^1} \left( \sqrt{ s^2 + t^2|f(i)|^2 + |y|^2} - \sqrt{s^2 + |y|^2} \right) s^2 + |F|_{C^0} t^2 |f(i)|^2 \\
&\leq c |F|_{C^1} t^2 |f(i)|^2.
\end{align}

For this fixed $y$, recalling that $\spt F \subset \{ R \in [1/10, 2/10]\}$, we have
\begin{align}
&\left| \sum_{i=1}^d \int_{\Omega_y(i) \cap A_{W(i)}(r_y, 1/4)} F(x', y)r^2 - \sum_{i=1}^d \int_{W(i) \cap A_{W(i)}(r_y, 1/4)} F(x', y)r^2 \right| \\
&= \left| \sum_{i=1}^d \int_{\partial W(i) \cap A_{W(i)}(r_y, 1/4)} (f(i)(s)\cdot n(i)(s)) \int_0^{1} F(s + tf(i)(s), y)|s + tf(i)(s)|^2 dt ds \right| \\
&= \left| \sum_{i=1}^d \int_{\partial W(i) \cap A_{W(i)}(r_y, 1/4)} (f(i)(s) \cdot n(i)(s)) \left( \int_0^{1} F(s + tf(i), y) |s + tf(i)|^2 - F(s, y) s^2 dt \right) ds  \right| \\
&\leq c \sum_{i=1}^d  \int_{\partial W(i) \cap A_{W(i)}(r_y, 1/4)} r |f(i)|^2 .
\end{align}

Integrating over $y$ (remember both $\Omega(i)$ and $W(i)$ are flat), and using \eqref{eqn:A-inclusion}, gives
\begin{align}
&\sum_{i=1}^d \int_{\Omega(i) \cap (B_{1/2}\setminus B_{2r_y}(\axis))} F(x', y) r^2 \label{eqn:diff-F-3} \\
&\leq \sum_{i=1}^d \int_{ (W(i) \times \R^m) \cap (B_{1/2}\setminus B_{r_y}(\axis))} F(x', y) r^2 + c\sum_{i=1}^d \int_{(\partial W(i) \times \R^m) \cap (B_{1/2} \setminus B_{r_y}(\axis))} r |f(i)|^2 . \nonumber
\end{align}

Combining the calculations \eqref{eqn:diff-F-1}, \eqref{eqn:diff-F-2}, \eqref{eqn:diff-F-3} with the effective estimates of Lemma \ref{lem:poly-graph}, we have
\begin{align}
&\int_M F r^2 - \int_C F r^2 \\
&\leq c \sum_{i=1}^d \int_{\Omega(i) \cap (B_{1/2} \setminus B_{r_y}(\axis))} |u(i)|^2 + r^2 |Du(i)|^2 \\
&\quad+ c \sum_{i=1}^d \int_{(\partial W(i) \times \R^m) \cap (B_{1/2} \setminus B_{r_y}(\axis))} r |f(i)|^2 \\
&\quad + \sum_{i=1}^d \int_{M(i) \cap B_{10r_y}(\axis)} F(x, y)r^2 - \sum_{i=1}^d \int_{(W(i)\times \R^m) \cap B_{r_y}(\axis)} F(x',y) r^2 \\
&\leq c E(M, \bC, 1) + c \sum_{i=1}^d \int_{M(i) \cap B_{10r_y}(\axis)} r^2 \\
&\leq c E(M, \bC, 1). 
\end{align}
This establishes \eqref{eqn:F-diff-est}.

We prove \eqref{eqn:F-single-est}.  Take $\eps$ as before.  We make an initial computation.  Suppose $(x', y) \in \Omega(i)$, and $(x, y) = (x', y) + u(i)(x', y) \in M(i)$.  Write $\pi_{M(i)^\perp}$ for the orthogonal projection onto $N_{(x, y)}M(i)$, and $\pi_{P(i)^\perp}$ for the orthogonal projection to $P(i)^\perp$.  Then we have
\begin{align}
|\pi_{M(i)^\perp}(x, 0)|
&= |\pi_{M(i)^\perp}(x, 0) - \pi_{P(i)^\perp}(x, 0)| + |\pi_{P(i)^\perp}(x - x', 0)| \\
&\leq c|x| |Du(i)(x', y)| + |u(i)(x', y)|.
\end{align}

We deduce that
\begin{align}
&\int_{M \cap B_{1/10}} |(x, 0)^\perp|^2  \\
&\leq \sum_i \int_{\Omega(i) \cap (B_{1/2} \setminus B_{r_y}(\axis))} |\pi_{M(i)^\perp}(x' + u(i), 0)|^2 J u(i) + \sum_i \int_{M(i) \cap B_{10r_y}(\axis)} r^2 \\
&\leq \sum_i \int_{\Omega(i) \cap (B_{1/2} \setminus B_{r_y}(\axis))} c r^2 |Du(i)|^2 + c|u(i)|^2 + \sum_i \int_{M(i) \cap B_{10 r_y}(\axis)} r^2 \\
&\leq c E(M, \bC, 1) .  
\end{align}
This completes the proof Lemma \ref{lem:density-workhouse}.
\end{proof}

\subsection{Moving the point} We localize the $L^2$-decay estimate to a given singular point, and demonstrate that the singular set must lie close to $\axis$ at the scale of the excess.

\begin{prop}\label{prop:Z-est}
For any $\tau > 0$ and $\alpha \in (0, 1)$, there is an $\eps(\refC, \tau)$ so that the following holds.  Take $M \in \cN_{\eps}(\refC)$.  Then for any $Z = (\zeta, \eta) \in \sing(M) \cap B_{1/4}$ with $\theta_M(Z) \geq \theta_{\refC}(0)$, we have
\begin{gather}\label{eqn:est-Z-M}
|\zeta|^2 + \int_{M \cap B_{1/4}} \frac{d_{Z + \bC}^2}{|X - Z|^{n+2-\alpha}} \leq c(\refC, \alpha) E(M, \bC, 1) , 
\end{gather}
and if we write $L = \cup_i L(i)$ for the lines of $\bC_0$, then
\begin{gather}\label{eqn:est-Z-u}
\int_{\Omega(i) \cap B_{1/2} \setminus B_{\tau}(L \times \R^m)} \frac{|u(i) - \zeta^\perp|^2}{|X - Z|^{n+2-\alpha}} \leq c(\refC, \alpha) E(M, \bC, 1) .
\end{gather}
Here $M(i), \Omega(i)$ is the decomposition as in Lemma \ref{lem:poly-global-graph}
\end{prop}

\begin{proof}
We first show the estimate
\begin{gather}
|\zeta|^2 \leq c(\refC) E(M, \bC, 1). 
\end{gather}

Take $\eps_2(\refC, \tau, \beta)$ as in Theorem \ref{lem:poly-global-graph}, with $\tau$ and $\beta \leq \tau/10$ to be specified.  So in particular, if we write $L = \cup_i L(i)$ for the union of lines of $\bC_0$, then we decompose
\begin{gather}\label{eqn:est-Z-graph}
M \cap B_{1/2} \setminus B_{10 \tau}(L \times \R^m) = \cup_i M(i),
\end{gather}
where each $M(i)$ is a graph of $u(i)$ over $W(i)$, with $|u(i)| \leq \beta |x|$.  An important but obvious consequence of Theorem \ref{lem:poly-global-graph} is that $|\zeta| \leq \tau$.

For simplicity let us take $\beta = \tau/10$.  Since $\bC$ is flat away from $B_{10\tau}(L \times \R^m)$, and $|\zeta| \leq \tau$, and each $M(i)$ has small $C^0$ norm, we have
\begin{gather}
|d_{\bC}(x, y) - d_{Z + \bC}(x, y)| = |\pi_{\bC^\perp}(\zeta)|
\end{gather}
for any $(x, y) \in M(i)$.  Here $\pi_{\bC^\perp}$ denotes the projection onto $N_{(x', y)} \bC$, where $(x, y) = (x', y) + u(i)(x', y)$.

Because we assume $\bC_0$ to have no additional symmetries, using a contradiction argument, can prove the existence of some $\delta_0(\refC) > 0$ so that, provided $\eps \leq \delta_0$, we have
\begin{gather}\label{eqn:no-sym-bound}
\int_{\bC \cap B_{1/4} \setminus B_{\delta_0}(L \times \R^m)} |a^\perp|^2 \geq 10 \delta_0 |a|^2 \quad \forall a \in \R^{2+k}\times \{0\},
\end{gather}
where $a^\perp$ at $(x', y) \in \bC$ is simply the projection to $N_{(x', y)}\bC$.

Pick $\rho$ small, but arbitrary.  Ensure $\tau \leq 2\delta_0(\refC) \rho \leq \rho/10$ and we obtain
\begin{align}
\int_{\cup_i M(i) \cap B_\rho(Z)} |\pi_{\bC^\perp}(\zeta)|^2
&\geq \frac{1}{10} \int_{\bC \cap B_{\rho/2}(Z) \setminus B_{\delta_0 \rho}(L \times \R^m)} |\zeta^\perp|^2 \\
&\geq \frac{\rho^n}{10} \int_{\bC \cap B_{1/4} \setminus B_{\delta_0}(L \times \R^m)} |\zeta^\perp|^2\\
&\geq \delta_0(\refC) |\zeta|^2 \rho^n,
\end{align}
where $\delta_0$ is \emph{independent} of $\rho$.  The first inequality follows from the graphical decomposition \eqref{eqn:est-Z-graph}.  The second inequality holds since $|\zeta| \leq \tau \leq \rho/10$.  The third inequality is \eqref{eqn:no-sym-bound}.

We can apply Propositions \ref{prop:decay-growth-est}, \ref{prop:density-est} to the point $Z$, and the cone $\bC + Z$, to deduce
\begin{gather}
\int_{M \cap B_{1/10}(Z)} \frac{d_{Z + \bC}^2}{|X - Z|^{n+2-\alpha}} \leq c \int_{M \cap B_1} d_{Z + \bC}^2 + c ||H_M||_{L^\infty(B_1)} \leq c E(M, \bC, 1) + c |\zeta|^2 .
\end{gather}

Combine the above two relations, to deduce
\begin{align}
\delta_0 \rho^n |\zeta|^2 
&\leq \int_{\cup_i M(i) \cap B_\rho(Z)} |\pi_{\bC^\perp}(\zeta)|^2  \\
&\leq \int_{M \cap B_\rho(Z)} d_{Z + \bC}^2 + \int_{M \cap B_\rho(Z)} d_{\bC}^2 \\
&\leq c E(M, \bC, 1) + c \rho^{n+2-\alpha} |\zeta|^2 .
\end{align}
where $c$ depends on $(\refC, \alpha)$ only (so, is independent of $\rho$ and $\tau$).  Choose $\rho$ small, and correspondingly ensure $\tau = \beta$ is sufficiently small, and we obtain the first part of \eqref{eqn:est-Z-M}.

To obtain the second estimate, apply Propositions \ref{prop:decay-growth-est}, \ref{prop:density-est} at $Z$, and then use the first part of \eqref{eqn:est-Z-M}:
\begin{align}
\int_{M \cap B_{1/4}} \frac{d_{Z + \bC}^2}{|X - Z|^{n+2-\alpha}} \leq c E(M, \bC, 1) + c |\zeta|^2 \leq c E(M, \bC, 1).
\end{align}

We prove the last estimate \eqref{eqn:est-Z-u}.  Take $(x, y) = (x', y) + u(i)(x', y) \in M(i) \cap B_{1/2} \setminus B_{\tau}(L\times \R^m)$.  From the bounds $|f(i)| \leq \beta |x|$ and $|u(i)| \leq \beta|x|$, we know that $(x', y) \in W(i)$ and
\begin{gather}
|u(i)(x', y) - \pi_{\bC^\perp}(\zeta)(x', y)| = d_{Z + \bC}(x, y).
\end{gather}
Now use the second part of \eqref{eqn:est-Z-M}, and the fact that the Jacobian has bound $1/2 \leq Ju(i) \leq 2$.
\end{proof}

\subsection{Estimates on the spine} Using the $\delta$-no-holes condition we can sum the estimates \ref{prop:Z-est} along the spine $\axis$.

\begin{prop}\label{prop:spine-est}
Given $\tau > 0$ and $\alpha \in (0, 1)$, there is an $\eps(\refC, \tau)$ so that the following holds.  Let $M \in \cN_{\eps}(\refC)$, and suppose that $M$ satisfies the $\tau/10$-no-holes condition w.r.t. $\refC$ in $B_{1/4}$.  Take $\alpha \in (0, 1)$.	

Then we have
\begin{gather}\label{eqn:spine-d-est}
\int_{M \cap B_{1/4}} \frac{d_{\bC}^2}{\max(r, \tau)^{2-\alpha}} \leq c(\refC, \alpha) E(M, \bC, 1),
\end{gather}
and if we write $L = \cup_{i=1}^{2d/3} L(i)$ for the lines of $\bC_0$, then 
\begin{gather}\label{eqn:spine-u-est}
\sum_{i=1}^d \int_{\Omega(i) \cap B_{1/4} \setminus B_\tau(L\times \R^m))} \frac{|u(i) - \kappa^\perp|^2}{\max(r, \tau)^{2+2-\alpha}} \leq c(\refC, \alpha) E(M, \bC, 1).
\end{gather}
Here $\kappa : (0,1]\times B_1^m \to \R^{2+k}\times\{0\}$ is a chunky function satisfying the bound $|\kappa|^2 \leq c(\refC, \alpha) E(M, \bC, 1)$.
\end{prop}

\begin{proof}
Let $r_\nu$, $Q_{\nu\mu}$ be as in Definition \ref{def:chunky}.  Whenever $r_\nu > \tau/2$, by the no-holes condition there is $Z_{\nu\mu} = (\zeta_{\nu\mu}, \eta_{\nu\mu}) \in \sing M \cap (B_{\tau/10}^{2+k}(0)\times Q_{\nu\mu})$ with $\theta_M(Z_{\nu\mu}) \geq \theta_{\refC}(0)$.  By Proposition \ref{prop:Z-est}, we we have
\begin{align}
\int_{\Omega(i) \cap (B_{r_\nu}^{2+k}(0) \times Q_{\nu\mu}) \setminus B_\tau(L\times \R^m)} |u(i) - \zeta_{\nu\mu}^\perp|^2
&\leq \int_{ \Omega(i) \cap B_{c(n,k) r_\nu}(Z_{\nu\mu}) \setminus B_\tau(L\times \R^m)} |u(i) - \zeta_{\nu\mu}^\perp|^2 \\
&\leq c(\refC,\alpha) r_\nu^{n+2-\alpha} E(M, \bC, 1) .
\end{align}

Define $\kappa$ by
\begin{gather}
\kappa(r, y) = \zeta_{\nu \mu} \text{ for } (r, y) \in [r_{\nu+1}, r_\nu) \times Q_{\nu\mu}.
\end{gather}
Then since the number of cubes $\{Q_{\nu\mu}\}_\mu$ intersecting $B_1^m$ is bounded by $c(m) r_\nu^{-m}$, we have for any $r_\nu > \tau/2$:
\begin{align}
\int_{\Omega(i) \cap B_{1/4} \cap \{ r_{\nu + 1} \leq r < r_\nu \} \setminus B_\tau(L \times \R^m)} |u(i) - \kappa^\perp|^2 
&\leq \sum_{\mu : Q_{\nu\mu} \cap B^m_1 \neq \emptyset} \int_{\bC \cap (B_{r_\nu}^{2+k}(0) \times Q_{\nu\mu}) \setminus B_{\tau}(L\times \R^m)} |u - \zeta_{\nu\mu}^\perp|^2 \\
&\leq c(\refC,\alpha) r_\nu^{2+2-\alpha} E(M, \bC, 1).
\end{align}

Now given any $\rho > \tau$, choose $\nu$ so that $r_{\nu+1} \leq \rho < r_\nu$.  We have
\begin{align}
\int_{\Omega(i) \cap B_{1/4} \cap \{ \tau \leq r < \rho\} \setminus B_\tau(L \times \R^m)} |u(i) - \kappa^\perp|^2 
&\leq c r_\nu^{2+2-\alpha} \left( \sum_{j = 0}^\infty 2^{-j (\ell+2-\alpha)} \right) E(M, \bC, 1) \\
&\leq c \rho^{2+2-\alpha} E(M, \bC, 1).
\end{align}
Multiply by $\rho^{-2-3+2\alpha}$ and integrate in $\rho \in [\tau, 1/4]$, to obtain \eqref{eqn:spine-u-est} with $2\alpha$ in place of $\alpha$.

Let us prove \eqref{eqn:spine-d-est}.  Take $Q_{\nu\mu}$, $Z_{\nu\mu}$ as before.  Then for each $\nu,\mu$, we have by the same reasoning as above
\begin{align}
\int_{M \cap Q_{\mu\nu}} \frac{d_{\bC}^2}{r_\nu^{n-\alpha}} 
&\leq 2\int_{M \cap B_{c(n,k)r_\nu}(Z_{\mu\nu})} \frac{d_{Z_{\nu\mu}+\bC}^2}{|X - Z_{\nu\mu}|^{n+2-\alpha}} + 2|\zeta|^2 \int_{M \cap B_{c(n,k)r_\nu}(Z_{\nu\mu})} |X - Z_{\nu\mu}|^{-n+\alpha} \\
&\leq c(\bC_0, \alpha) E(M, \bC, 1).
\end{align}
In this last inequality we used Proposition \ref{prop:Z-est} centered at $Z$, and the mass bound $\mu_M(B_R) \leq c(\refC) R^n$.

Therefore, given any $\rho$, we can choose an appropriate $\nu$ and sum over $\mu$ as before to deduce
\begin{gather}
\int_{M \cap B_{1/4} \cap B_\rho(\{0\}\times \R^m)} d_{\bC}^2 \leq c \rho^{\ell-\alpha} E(M, \bC, 1).
\end{gather}
Now multiply by $\rho^{-\ell-1+2\alpha}$ and integrate in $\rho \in [\tau, 1/4]$, to obtain \eqref{eqn:spine-d-est} with $2\alpha$ in place of $\alpha$.
\end{proof}


\end{comment}

\section{Jacobi fields}\label{sec:jacobi}


The aim of this prove, under suitable assumptions, a superlinear decay on Jacobi fields whose linear part has been removed.  To state this we require some additional notation.

\begin{definition}
Let $\cL$ be the subspace of linear compatible fields $v : \refC \to {\refC}^\perp$ of the form
\begin{gather}
\cL = \left\{ v(x, y) = \pi_{{\refC}^\perp}(Ay) + v_0(x) : \begin{array}{c}\text{$A$ is a linear map $\axis \to \R^{2+k}\times \{0\}$,} \\ \text{and $v_0 : \refC_0 \to {\refC_0}^\perp$ is linear compatible} \end{array} \right\}.
\end{gather}
We will find these are precisely the $1$-homogeneous Jacobi fields arising from our blow-up procedure.

Now given an arbitrary compatible Jacobi field $v : \refC \to {\refC}^\perp$, and a scale $\rho > 0$, let us define $\psi_\rho \in \cL$ to be the element of $\cL$ minimizing
\begin{gather}
\min \left\{ \int_{\refC \cap B_\rho} v\cdot \psi \text{ among } \psi \in\cL \right\} ,
\end{gather}
and then define
\begin{gather}
v_\rho = v - \psi_\rho,
\end{gather}
so that $v_\rho$ is $L^2(\refC \cap B_\rho)$-orthogonal to every field in $\cL$.
\end{definition}

Our main Theorem of this section is the following.
\begin{theorem}[Linear decay]\label{thm:linear-decay}
Let $v : \refC \cap B_{1} \to {\refC}^\perp$ be a compatible Jacobi field, and fix $\theta \in (0, 1/4]$.  Suppose that for every $\rho \in [\theta, 1/4]$, there is a chunky function $\kappa_\rho : (0,\rho] \times B_\rho^m \to \R^{2+k}\times\{0\}$ so that we have the following two estimates:

A) Non-concentration estimate:
\begin{gather}\label{eqn:linear-decay-hyp1}
\rho^{2+2-\alpha} \int_{\refC \cap B_{\rho/4}} \frac{|v_\rho - \kappa_\rho^\perp|^2}{r^{2+2-\alpha}} \leq \beta \int_{\refC \cap B_\rho} |v_\rho|^2,
\end{gather}
with the pointwise bound $|\kappa_\rho| \leq \beta \rho^{-n} \int_{\refC \cap B_\rho} |v_\rho|^2$;

B) Hardt-Simon growth estimate:
\begin{gather}\label{eqn:linear-decay-hyp2}
\int_{\refC \cap B_{\rho/10}} R^{2-n} |\partial_R (v/R)|^2 \leq \beta \rho^{-n-2} \int_{\refC \cap B_\rho} |v_\rho|^2.
\end{gather}

Then there are constants $c_2$, $\mu$, depending only on $(\refC, \beta, \alpha)$, so that
\begin{gather}
\theta^{-n-2} \int_{\refC \cap B_\theta} |v_\theta|^2 \leq c_2 \theta^\mu \int_{\refC \cap B_1} |v_{1}|^2.
\end{gather}
\end{theorem}

Since the argument is somewhat involved, we provide a short outline.
\begin{proof}[Outline of Proof]
The biggest hurdle is to show that any $1$-homogeneous compatible Jacobi field $v$ satisfying \eqref{eqn:linear-decay-hyp1} must lie in $\mathcal{L}$.  This is proven in Theorem \ref{thm:1-homo-linear} as follows. 
		
First, we decompose $v(r\theta,y)=\sum_{i = 0}^\infty v_i(r, y)$, where each $v_i$ is the projection of $v$ onto the $i$-th eigenfunction of the Jacobi operator on the geodesic net $\Gamma=\bC_0\cap \sphere^{1+k}$. Thanks to the compatibility conditions, this operation is well defined on the net (see Section \ref{sec:eigenfunctions}) and each $v_i$ is smooth (Proposition \ref{lem:jacobi-apriori-est}). 
		
Next we observe that by $1$-homogeneity of $v$, we can write $v_i(r, y) = r \phi_i(y/r)$ and each $\phi_i$ satisfies the equation
\begin{gather}
\sum_{j,k=1}^m (\delta_{jk} + z^j z^k) D_j D_k \phi_i - \sum_{j=1}^m z^j D_j \phi_i + (1-\lambda_i)\phi_i = 0 .
\end{gather}
with $\lambda_i$ the eigenvalue associated to $\phi_i$. Moreover \eqref{eqn:linear-decay-hyp1}, becomes
\begin{align*}
&\int_{1}^\infty \int_{S^{m-1}} t^{-1-\alpha} |\phi_i(t\omega)|^2 d\omega dt < \infty \quad \text{when }\lambda_i \neq 0, \\
&\int_1^\infty \int_{S^{m-1}} t^{1-\alpha} | t^{-1} \phi_i(t\omega) - \tilde \kappa_i(t\omega)|^2 d\omega dt < \infty \quad \text{when }\lambda_i = 0 . 
\end{align*}

Now one exploits the fact that in polar coordinates the PDE of $\phi$ has a divergence structure, so that we can test it with a logarithmic cut-off function and use the above inequalities to estimate the RHS to prove in Lemma \ref{lem:pde-analysis} that:
\begin{enumerate} 
\item when $\lambda_i = 0$, then $\phi_i(z) = a\cdot z$ for some $a \in \R^m$ (corresponding to a rotation of the spine),
\item when $\lambda_i = 1$, then $\phi_i(z) \equiv const$ (corresponding to an action on $\bC$ that fixes the spine), 
\item otherwise, $\phi_i(z) = 0$ ($v$ cannot act on the spine in any other fashion).
\end{enumerate}
However we cannot do this directly, since a reverse Poincar\'e inequality might not be true for $\phi_i$, and indeed, we will need to study the equation for the radial part of each fourier mode of $\phi_i$ separately (see Lemma \ref{lem:ode-est}).

The underlying reason this works is because of the no-holes condition: in assuming the existence of a singular set (i.e. points of good density) arbitrarily near $\axis$, we enforce the infintesimal motion to act on $\axis$ by rotation.

At this point a simple contradiction argument allows us to prove that whenever $v_\rho$ ($ = $ component of $v$ orthogonal at scale $B_\rho$ to $\cL$) satisfies the non-concentration estimate \eqref{eqn:linear-decay-hyp1}, then the following quantitative growth estimate holds
\begin{gather}
\int_{\refC \cap B_\rho\setminus B_{\rho/10}} R^{2-n}|\partial_R(v/R)|^2 \equiv \int_{\bC \cap B_\rho \setminus B_{\rho/10}} R^{2-n}|\partial_R(v_\rho/R)|^2 \geq \frac{1}{c(\refC)} \rho^{-n-2} \int_{\refC \cap B_\rho} |v_\rho|^2.
\end{gather}
This can be combined with the Hardt-Simon inequalty \eqref{eqn:linear-decay-hyp2} to prove a decay of $\int_{\bC\cap B_\rho} R^{2-n} |\partial_R(v/R)|^2$, and hence a decay of $\rho^{-n-2} \int_{\bC \cap B_\rho} |v_\rho|^2$ also.
\end{proof}


\subsection{Elementary facts}

Let us prove some elementary properties of compatible Jacobi fields.  First, we demonstrate smoothness and a priori estimates up to and including the wedge boundaries.
\begin{lemma}\label{lem:jacobi-apriori-est}
Suppose $v : \refC \cap B_1 \to {\refC}^\perp$ is $C^{1,\alpha}$, satisfies the $C^0$- and $C^1$-compatibility conditions of Definition \ref{def:compatible}, and each $v(i)$ is harmonic on $\mathrm{int} W(i)\times \R^m$.

Then $v$ is a compatible Jacobi field in $B_1$ (so, is smooth up to and including the wedge boundaries), and for every non-negative integer $k$ and $\rho < 1$ we have the pointwise bound
\begin{gather}\label{eqn:jacobi-apriori-est}
\sup_{B_{\rho} \cap \left(W(i) \times \R^m\right)} |x|^{2k+n} |D^k v(i)|^2 \leq c(\refC, \rho, k) \int_{\refC \cap B_1} |v|^2.
\end{gather}
\end{lemma}

\begin{proof}
Away from $\partial W(i) \times \R^m$ smoothness follows from harmonicity.  Let assume assume $W(1)$, $W(2)$, $W(3)$ share a common boundary line $L$.  By the $C^1$ compatability condition we can perform an even extension of $v(1) + v(2) + v(3)$ across $L \times \R^m$ near $z$, and deduce that $v(1) + v(2) + v(3)$ is smooth up to $L \times \R^m$.

Let $P$ be the plane spanned by the conormals $n(1), n(2), n(3)$, and denote by $v(i)^T$ and $v(i)^\perp$ the orthogonal projections to $P$, $P^\perp$ respectively.  We can identify $P$ with $\R^2$, and the $n(i)$ with $1, e^{2\pi i/3}, e^{4\pi i/3}$.  By the $C^1$ compatability condition one can easily verify that
\begin{gather}
\partial_{n(i)} v(i)^T = \alpha e^{i\pi/2} n(i),
\end{gather}
for some $\alpha \in\R$.  In other words, up to a fixed scaling factor, each $\partial_{n(i)} v(i)^T$ is a $90^0$ rotation of $n(i)$.

We deduce that $\partial_{n(i)} v(i) \cdot n(i) = \partial_{n(j)} v(j)\cdot n(j)$ along $L$, and so using an even extension we deduce $v(i)^T - v(j)^T$ is smooth up to $L$.  Similarly, by the $C^0$ compatability condition, we have that $v(i)^\perp = v(j)^\perp$ along $L$, and so using an odd extension we deduce $v(i)^\perp - v(j)^\perp$ is smooth up to $L$.

Combining the above relations gives that each $v(i)$ extends smoothly to $L$.

We now prove \eqref{eqn:jacobi-apriori-est}.  Observe that near $L$, each $v(i)$ can be written as the sum of harmonic functions which extend smoothly across $L$ (by either an even or odd reflection).  Therefore, at \emph{any} point $(x, y) \in \refC\cap B_\rho$, we can scale up $|x| \to 1$ and use standard interior estimates to bound
\begin{gather}\label{eqn:jacobi-apriori-1}
\sup_{B_{\delta |x|}(x, y) \cap (W(i)\times \R^m)} |x|^{2k+n} |D^k v(i)|^2 \leq c(\refC, \delta, k) \int_{B_{2\delta |x|}(x, y) \cap (W(i)\times \R^m)} |v|^2 \leq c \int_{M \cap B_1} |v|^2.
\end{gather}
Here $\delta = \delta(\refC, \rho)$ is chosen to be
\begin{gather}
\delta = \min \{ 1/100 \cdot \text{ (smallest geodesic length in $\refC_0 \cap \bbS^{1+k}$)}, (1-\rho)/2 \}.
\end{gather}
The Lemma follows directly. 
\end{proof}

The following Proposition demonstrates that a compatible, $1$-homogeneous field on a polyhedral cone generates a rotation locally.  Unfortunately, it's not always clear if the local-rotations can patch together for form a global movement of the net.  Note this Proposition concerns the \emph{cross-section} $\refC_0$, not the full cone $\refC = \refC_0 \times \R^m$.
\begin{prop}\label{prop:baby-linear}
Let $v : \refC_0 \to {\refC_0}^\perp$ be a compatible $1$-homogeneous Jacobi field.  Then $v$ is linear: there are skew-symmetric matrices $A(i) : \R^{n+k} \to \R^{n+k}$ so that $v(i) = \pi_{P(i)^\perp} \circ A(i)$.

Morover, the $A(i)$ are locally compatible in the following sense: if $W(i_1), W(i_2), W(i_3)$ share a common boundary line $L$, then there is a skew-symmetric matrix $A_L$ so that
\begin{gather}
\pi_{P(i_j)^\perp} \circ A(i_j) = \pi_{P(i_j)^\perp} \circ A_L \quad \text{ for each } j = 1, 2, 3.
\end{gather}
\end{prop}

\begin{proof}
On each wedge $W(i)$, $v(i)$ is harmonic and $1$-homogeneous.  Since $W(i)$ is $2$-dimensional, it follows immediately that $v(i)$ is a linear map $W(i) \to W(i)^\perp$.  Since the domain and range of $v(i)$ are orthogonal, we can extend it to a skew-symmetric linear mapping on $\R^{n+k}$.

Let us prove local compatibility.  Fix a line $L \equiv L(1)$ of $\refC_0$, and suppose without loss of generality that the wedges $W(1), W(2), W(3)$ all meet at $L$.  For each such wedge, write $n(i)$ for the unit outward conormal of $L \subset W(i)$, and $\ell$ for the unit vector defining $L$.

On each piece $W(i)$, by assumption we can write the field $v(i)$ as
\begin{gather}
v(i)(x) = a(i) (x \cdot n(i)) + b(i) (x \cdot \ell),
\end{gather}
where $a(i), b(i) \in P(i)^\perp \subset \R^{2+k}$.  Here $\cdot$ denotes the standard Euclidean inner product.

By the $C^0$ compatibility condition, we have that
\begin{gather}
b(i) = \pi_{P(i)^\perp}(b),
\end{gather}
where $b$ is a fixed vector in $L^\perp \subset \R^{\ell+k}$.

From the $C^1$ compatability condition, we have $\sum_{i=1}^3 a(i) = 0$.  Therefore, by Lemma \ref{lem:sum-is-zero} we can choose an anti-symmetric $A'$ so that $A'(n(i)) = a(i)$.  Define the linear mapping
\begin{gather}
A(x) = A' x + (b - A'(\ell))(x \cdot \ell).
\end{gather}
Then $A$ is anti-symmetric, since $x^T A x = 0$ for every $x$, and by construction we have $v(i) = \pi_{P(i)^\perp} \circ A$.
\end{proof}

\subsection{Eigenfunctions on a net}\label{sec:eigenfunctions}

We require some additional notation.  Write $\Gamma = \refC_0 \cap \bbS^{1+k}$ to be the corresponding equiangular geodesic net of $\refC_0$ composed of geodesc segments $\cup_{i=1}^d \ell(i)$.  Here each $\ell(i) = W(i) \cap \bbS^{1+k}$.  We write a function $u : \Gamma \to {\refC_0}^\perp$ as a collection of functions $u(i) : \ell(i) \to W(i)^\perp$.

Define the norms
\begin{gather}
||u||_0^2 = \int_\Gamma |u|^2,\quad ||u||^2_1 = \int_\Gamma |u|^2 + |u'|^2, \quad ||u||^2_2 = \int_\Gamma |u|^2 + |u'|^2 + |u''|^2
\end{gather}
and let $L^2(\Gamma)$, $W^{1,2}(\Gamma)$, $W^{2,2}(\Gamma)$ be the completion of $C^\infty(\Gamma, {\refC_0}^\perp)$ with respect to these norms.  By Sobolev embedding, we have $W^{1,2} \subset C^0(\Gamma, {\refC_0}^\perp)$ and $W^{2,2} \subset C^1(\Gamma, {\refC_0}^\perp)$.

We say $u \in C^0(\Gamma)$ is \emph{$C^0$-compatible} if for every $p \in \partial \ell(i)$, there is a vector $V$ \emph{independent of $i$} so that $u(i)(p) = \pi_{\ell(i)^\perp}(V)$.  We say $u \in C^1(\Gamma)$ is \emph{$C^1$-compatible} if:
\begin{gather}
\partial_n u(i_1)(p) + \partial_n u(i_2)(p) + \partial_n u(i_3)(p) = 0
\end{gather}
whenever $\ell(i_1), \ell(i_2), \ell(i_3)$ share a common vertex $p$ ($n$ being the outward conormal).  Clearly, a Jacobi field $v : \refC_0 \to {\refC_0}^\perp$ is compatible if and only if each slice $v(r \equiv r_0)$ is compatible on the net $r_0 \Gamma$.

We aim to show the following:
\begin{theorem}\label{thm:eigenfunctions}
There is a sequence $0 = \lambda_1 \leq \lambda_2 \leq \ldots \to \infty$, and a collection $u_i \in C^\infty(\Gamma, {\refC_0}^\perp)$, so that
\begin{gather}
u_i'' + \lambda_i u_i = 0, \quad u_i \text{ is $C^0$- and $C^1$-compatible}, 
\end{gather}
and the $\{u_i\}_i$ form an orthonormal basis in $L^2(\Gamma)$.
\end{theorem}

\begin{remark}
If $v : \refC_0 \to {\refC_0}^\perp$ is a $1$-homogeneous compatible Jacobi field, then $v(r = 1)$ is an eigenfunction of $u \mapsto -u''$ with eigenvalue $1$.

If $V \in \R^{2+k}$ is a fixed vector, then $u(i)(x) = \pi_{W(i)^\perp}(V)$ is an eigenfunction of $-u''$ with eigenvalue $0$.  If $A : \R^{2+k} \to \R^{2+k}$ is a fixed linear map, then $u(i)(x) = \pi_{W(i)^\perp}(Ax)$ is an eigenfunction with eigenvalue $1$.
\end{remark}

Let us define the spaces
\begin{align}
H_1 &= \{ u \in W^{1,2}(\Gamma) \subset C^0(\Gamma, {\refC_0}^\perp) : \text{ $u$ is $C^0$-compatible} \}, \\
H_2 &= \{ u \in H_1 \cap W^{2,2}(\Gamma) \subset C^1(\Gamma, {\refC_0}^\perp) : \text{ $u$ is $C^1$-compatible } \}.
\end{align}
By Sobolev embedding and linearity of compatability conditions each $H_i$ is a well-defined closed (Hilbert) subspace of $W^{i,2}(\Gamma)$.  Our key Lemma is the following.

\begin{lemma}\label{lem:compact-mapping}
The mapping $H_2 \to L^2(\Gamma)$ sending $u \mapsto -u'' + u$ has a bounded inverse map $S : L^2(\Gamma) \to H_2$, which is self-adjoint as a map $L^2 \to L^2$.
\end{lemma}

\begin{proof}
The bilinear form $A : H_1 \times H_1 \to \R$ defined by
\begin{gather}
A(u, \phi) = \int_\Gamma u' \cdot \phi' + u\cdot \phi
\end{gather}
coincides with the inner product on $H_1$.  So by the Reisz representation theorem, there is a bounded solution operator $S : L^2(\Gamma) \to H_1$, which solves
\begin{gather}\label{eqn:basis-1}
A(S(f), \phi) = \int f \cdot \phi \quad \forall \phi \in H_1.
\end{gather}

We show that $S$ maps into $H_2$.  Let us fix some $u = S(f)$.  By standard arguments, and since $\Gamma$ is $1$-dimensional, we can take various $\phi$ supported in a fixed segment to deduce $u \in W^{2,2}$.  In particular, $u$ solves
\begin{gather}\label{eqn:basis-2}
-u'' + u = f \quad \haus^1-a.e. \text{ in } \Gamma.
\end{gather}
We just need to verify $u$ is $C^1$-compatible.

Fix a vertex $p$, and WLOG we can assume the segment $\ell(1), \ell(2), \ell(3)$ meet at $p$.  Choose $\phi$ to be supported in a neighborhood of $p$, then integrate by parts \eqref{eqn:basis-1} and use \eqref{eqn:basis-2} to obtain, for some fixed vector $V$, 
\begin{gather}
\sum_{i=1}^3 \phi(i)(p) \cdot \partial_n u(i)(p) = \sum_{i=1}^3 \pi_{<n(i)>^\perp}(V) \cdot \partial_n u(i)(p),
\end{gather}
where in the equality we used the $C^0$-compatibility of $\phi$.  Here we write explicitly $n(i)$ for the outward conormal of $\ell(i)$.  Since $V$ and $p$ are arbitrary, by Lemma \ref{lem:sum-is-zero} we deduce that $u$ is $C^1$-compatible.  This proves the claim.

Using that $u$ solves $-u'' + u = f$ at $\haus^1$-a.e. point, we can test against $u$ and $u''$, and use the $C^0$-/$C^1-$ compatibility conditions to integrate by parts, to obtain
\begin{gather}
||u||_2^2 \leq 2 \int_\Gamma |u|^2 + 2 \int_\Gamma |f|^2 + 2\int_\Gamma |f|^2 \leq 10 ||f||_0^2.
\end{gather}
So $S : L^2 \to H_2$ is bounded.

Let us demonstrate self-adjointness.  Take $v = S(g)$, and then by Lemma \ref{lem:sum-is-zero} the $C^0$- and $C^1$-compatability conditions ensure we can integrate by parts without picking up any boundary terms:
\begin{align}
\int_\Gamma f \cdot S(g) = \int_\Gamma (-u'' + u) \cdot v = \int_\Gamma u' \cdot v' + u\cdot v = \int_\Gamma u \cdot (-v'' + v) = \int_\Gamma S(f) \cdot g.
\end{align}
This completes the proof of Lemma \ref{lem:compact-mapping}.
\end{proof}

\begin{proof}[Proof of Theorem \ref{thm:eigenfunctions}]
From Lemma \ref{lem:compact-mapping}, the solution operator $S : L^2(\Gamma) \to L^2(\Gamma)$ is compact and self-adjoint, and therefore has a countable eigenbasis $u_i$ with eigenvalues $\mu_i \to 0$.  For each $u_i$, we have $u_i \in H_2$, and
\begin{gather}
u_i'' + (\mu_i^{-1} - 1) u_i = 0
\end{gather}
weakly in $H_1$, and strongly in $H_2$.  It's now straightforward to check each $u_i$ is smooth, and non-negativity of the $\lambda_i$ follows from integration by parts.
\end{proof}

\subsection{$1$-homogeneous implies linear}

For a polyhedral cone without any spine, we easily have that any $1$-homogeneous Jacobi field is linear in the sense of Definition \ref{def:compatible}.  However this argument fails in the presence of a spine.  Following Simon \cite{simon1}, we show that any compatible Jacobi field on $\refC = \refC_0\times \R^m$ with appropriate decay splits into rotation of the spine plus some linear component on $\refC_0$.

\begin{theorem}\label{thm:1-homo-linear}
Let $v : \refC \cap B_1 \to {\refC}^\perp$ be a $1$-homogeneous compatible Jacobi field, satisfying
\begin{gather}\label{eqn:1-homo-decay}
\int_{\refC \cap B_1} \frac{|v - \kappa^\perp|^2}{r^{2+2-\alpha}} < \infty,
\end{gather}
for some $\alpha \in (0, 1)$, and some bounded chunky function $\kappa : (0, 1] \times B_1^m \to \R^{2+k}\times \{0\}$.

Then there is a linear map $A : \{0\}\times \R^m \to \R^{2+k}\times \{0\}$, and a linear compatible Jacobi field $v_0 : \refC_0 \to {\refC_0}^\perp$, so that
\begin{gather}
v(x, y) = \pi_{{\refC}^\perp}(Ay) + v_0(x).
\end{gather}
In other words, $v \in \cL$.
\end{theorem}

\begin{proof}
Write $\Gamma = \refC_0 \cap \bbS^{1+k}$, and let $\Psi_i(\theta)$ be the eigenfunction expansion of $L^2(\Gamma)$ from Theorem \ref{thm:eigenfunctions}, with associated eigenvalues $\lambda_i$.  Write
\begin{gather}\label{eqn:1-homo-1}
v_i(r, y) = \int_\Gamma v(r\theta, y)\cdot \Psi_i(\theta) d\theta, \quad \kappa_i(r, y) = \int_\Gamma \kappa(r, y)^\perp \cdot \Psi_i(\theta) d\theta,
\end{gather}
so that $v(r\theta, y) = \sum_i v_i(r, y) \Psi_i(\theta)$ and $\kappa^\perp(r\theta, y) = \sum_i \kappa_i(r, y) \Psi_i(\theta)$.  Notice that, since
\begin{gather}
\kappa_i(r, y) \equiv \kappa(r, y) \cdot \int_\Gamma \Psi_i(\theta) d\theta,
\end{gather}
we have $\kappa_i \equiv 0$ unless $\lambda_i = 0$.

Since both $v$ and $\Psi_i$ are smooth and compatible, and $v$ is harmonic on each wedge, we can integrate \eqref{eqn:1-homo-1} by parts to deduce each $v_i(r, y)$ is smooth and solves:
\begin{gather}
\partial_r^2 v_i + \frac{1}{r} \partial_r v_i + \Delta_y v_i - \frac{\lambda_i}{r^2} v_i = 0.
\end{gather}

Let us define $\phi_i : \R^m \to \R$ by 
\begin{gather}
\phi_i(z) = v_i(1, z),
\end{gather}
so that by $1$-homogeneity we have $v_i(r, y) = r \phi_i(y/r)$.  By direct calculation, we see that $\phi_i$ satisfies the equation
\begin{gather}\label{eqn:spine-pde}
\sum_{j,k=1}^m (\delta_{jk} + z^j z^k) D_j D_k \phi_i - \sum_{j=1}^m z^j D_j \phi_i + (1-\lambda)\phi_i = 0 .
\end{gather}

We aim to show that any $\phi_i$ satisfying \eqref{eqn:spine-pde} and a decay condition guaranteed by \eqref{eqn:1-homo-decay}, must be either linear or constant, depending on the value of $\lambda_i$.  Let us first find the correct decay condition on each $\phi_i$.

Using the orthonormality of the $\Psi_i$, we can write
\begin{align}
\int_{\refC \cap B_1} \frac{|v - \kappa^\perp|^2}{r^{2+2-\alpha}}
&= \sum_i \int_{B_1^m} \int_0^{\sqrt{1-|y|^2}} r^{\alpha-3} |v_i(r, y) - \kappa_i(r, y)|^2 dr dy \label{eqn:1-homo-2} \\
&= \sum_i \int_{0 \ni s}^1 s^{m-1} \int_{0 \ni r}^{\sqrt{1-s^2}} \int_{\bbS^{m-1}} r^{\alpha-3} |v_i(r, s\omega) - \kappa_i(r, s\omega)|^2 d\omega dr ds \\
&= \sum_i \int_{0 \ni s}^1 s^{m-1} \int_{1/\sqrt{1-s^2} \ni t}^\infty  \int_{\bbS^{m-1}} t^{1-\alpha} |v_i(t^{-1}, s\omega) - \kappa_i(t^{-1}, s\omega)|^2 d\omega dt ds.
\end{align}
Therefore, by choosing an appropriate $s_0 \in (1/3, 1/2)$, we have
\begin{align}
&\int_{1}^\infty \int_{\bbS^{m-1}} t^{-1-\alpha} |\phi_i(t\omega)|^2 d\omega dt < \infty \quad \text{when }\lambda_i \neq 0, \\
&\int_1^\infty \int_{\bbS^{m-1}} t^{1-\alpha} | t^{-1} \phi_i(t\omega) - \tilde \kappa_i(t\omega)|^2 d\omega dt < \infty \quad \text{when }\lambda_i = 0 ,
\end{align}
where $\tilde \kappa_i(t\omega) := \kappa_i(t^{-1}, s_0\omega)$ is uniformly bounded.

We can now apply Lemma \ref{lem:pde-analysis} (proved just below) to deduce that $\phi_i(z) = a_i \cdot z$ when $\lambda_i = 0$, $\phi_i(z) \equiv b_i$ when $\lambda_i = 1$, and $\phi_i \equiv 0$ otherwise.  So we can write
\begin{align}
v(r\theta, y) 
&= \sum_{ \{ i : \lambda_i = 0 \} } r \phi_i(y/r) \Psi_i(\theta) + \sum_{ \{ i : \lambda_i = 1 \} } r \phi_i(y/r) \Psi_i(\theta) \\
&= \sum_{j=1}^{m} y^j  w_j(\theta) + r v_0(\theta),
\end{align}
where each $w_j(\theta)$ lies in the $\lambda = 0$ eigenspace of $L^2(\Gamma)$, and $v_0(r\theta) \equiv rv_0(\theta)$ is a $1$-homogeneous compatible Jacobi field on $\refC_0$.

By Proposition \ref{prop:baby-linear}, we know $v_0$ is linear.  We must show each $w_j(\theta)$ lies in the space
\begin{gather}
\cV = \{ \pi_{{\refC_0}^\perp}(v) : v \in \R^{2+k}\times\{0\} \}.
\end{gather}
Let $P$ be the $L^2(\Gamma)$ orthogonal projection to $\cV^\perp \subset L^2(\Gamma)$.

Since $\kappa^\perp \in \cV$ for each $(r, y)$, we have from \eqref{eqn:1-homo-decay} and $L^2(\Gamma)$-orthogonality of $w_j(\theta), v_0(\theta)$ that
\begin{align}
\int_0^1 \int_{B^m_{\sqrt{1-r^2}}} \int_\Gamma r^{\alpha-3} |\sum_{j=1}^{m} y^j w_j(\theta) - \kappa^\perp(r, y)|^2 d\theta dy dr
&\geq \int_0^1 \int_{B^m_{\sqrt{1-r^2}}} \int_\Gamma r^{\alpha-3} |P(\sum_{j=1}^{m} y^j w_j(\theta))|^2 d\theta dy dr
\end{align}
is finite, which necessitates that $P(\sum_{j=1}^{m} y^j w_j(\theta)) \equiv 0$ on $B^m_1 \times \Gamma$.  Hence, every $w_j \in \cV$ as required.
\end{proof}

To prove Lemma \ref{lem:pde-analysis} we shall need the following $W^{1,2}$ estimate.  We note that \eqref{eqn:pde-ineq} fails for general solutions of \eqref{eqn:spine-pde}, so in our analysis of Lemma \ref{lem:pde-analysis} we must consider each term of the Fourier expansion separately.

\begin{lemma}\label{lem:ode-est}
Suppose $\gamma : \R_+ \to \R$ satisfies the ODE
\begin{gather}\label{eqn:main-ode}
(1+r^2)\gamma'' + \left( (m-1)/r - r \right) \gamma' + (-\mu/r^2 + 1-\lambda)\gamma = 0,
\end{gather}
where $m \geq 1$.  Then for any $4 \leq \rho$ we have
\begin{gather}\label{eqn:ode-ineq}
\int_{\rho}^{2\rho} (\gamma')^2 dr \leq c(\mu, \lambda) \int_{\rho/2}^{4\rho} \gamma^2/r^2 dr.
\end{gather}

In particular, if $m \geq 2$, and $\phi = \gamma(r) \tilde\phi(\omega)$ solves \eqref{eqn:spine-pde} in $\R^m$, where $\tilde\phi(\omega)$ is an eigenfunction of $-\Delta_{\bbS^{m-1}}$ with eigenvalue $\mu$, then
\begin{gather}\label{eqn:pde-ineq}
\int_\rho^{2\rho} \int_{\bbS^{m-1}} |D\phi(r\omega)|^2 d\omega dr \leq c(\mu, \lambda) \int_{\rho/2}^{4\rho} \int_{\bbS^{m-1}} \phi(r\omega)^2/r^2 d\omega dr .
\end{gather}
\end{lemma}

\begin{proof}
The ODE \eqref{eqn:main-ode} can be written in the divergence form:
\begin{gather}\label{eqn:main-ode-div}
\partial_r (h(r) \partial_r \gamma(r)) + \frac{h(r)}{1+r^2} (r^{-2} \mu + 1 - \lambda)\gamma(r) = 0,
\end{gather}
where $h(r) = r^{m-1} (1+r^2)^{1-(2+m)/2}$.  Take $\eta(r)$ a cutoff which is $\equiv 0$ outside $[\rho/2, 4\rho]$, $\equiv 1$ on $[\rho, 2\rho]$, and linearly interpolates in between.  If we multiply \eqref{eqn:main-ode-div} by $\gamma \eta^2$, then we obtain
\begin{align}
\int (\gamma')^2 \eta^2 h dr \leq 5\int \frac{h}{1+r^2} \eta^2 \gamma^2 (|\mu| + 1 + |\lambda|)  + (\eta')^2 \gamma^2 h dr.
\end{align}
where we used that $r^{-2} |\mu| \leq |\mu|$ on $\spt \eta$.

Since $\rho \geq 4$, then we have
\begin{gather}
\frac{1}{2} r^{-1} \leq h(r) \leq 2 r^{-1},
\end{gather}
and therefore
\begin{align}
\rho^{-1} \int_\rho^{2\rho} (\gamma')^2
&\leq 50 (1 + |\mu| + |\lambda|) \rho \int_{\rho/2}^{4\rho} \gamma^2 ,
\end{align}
which proves the required relation \eqref{eqn:ode-ineq}.  

Let us now take $\phi(r\omega) = \gamma(r) \tilde\phi(\omega)$ solving \eqref{eqn:spine-pde}, with $\Delta_{\bbS^{m-1}} \tilde\phi + \mu \tilde\phi = 0$.  By direct computation, we see that $\gamma$ solves the ODE \eqref{eqn:main-ode}. Therefore, we can use \eqref{eqn:ode-ineq} to compute that
\begin{align}
\int_\rho^{2\rho} \int_{\bbS^{m-1}} |D\phi|^2  d\omega dr = \left( \int_{\bbS^{m-1}} \tilde\phi^2 d\omega \right) \int_\rho^{2\rho} (\gamma')^2 + \mu \gamma^2/r^2 dr \leq c(\mu, \lambda) \int_{\rho/2}^{2\rho} \int_{\bbS^{m-1}} \phi^2/r^2 d\omega dr
\end{align}

\end{proof}

\begin{lemma}\label{lem:pde-analysis}
Let $\phi: \R^m \to \R$ is a smooth function satisfying \eqref{eqn:spine-pde}, and take a fixed $\lambda \geq 0$.  Assume $\phi$ satisfies the decay bound
\begin{align}
&\int_1^\infty \int_{\bbS^{m-1}} r^{1-\alpha}  | r^{-1} \phi(r\omega) - k(r\omega)|^2 d\omega dr < \infty \quad \text{ if $\lambda = 0$}, \label{eqn:pde-bound-1} \\
&\int_1^\infty \int_{\bbS^{m-1}} r^{-1-\alpha} |\phi(r\omega)|^2 < \infty \quad \text{ if $\lambda > 0$}, \label{eqn:pde-bound-2}
\end{align}
where $k : \R^m \to \R$ is some bounded measurable function.

Then:

A) when $\lambda = 0$, then $\phi(z) = a\cdot z$ for some $a \in \R^m$, 

B) when $\lambda = 1$, then $\phi(z) \equiv const$, 

C) otherwise, $\phi(z) = 0$.
\end{lemma}

\begin{proof}
Consider the case $m \geq 2$, and let us first suppose $\phi$ takes the special form $\phi(r\omega) = \gamma(r) \psi(\omega)$, where $\psi$ is an eigenfunction of $-\Delta_{\bbS^{m-1}}$ with eigenvalue $\mu$.  Let $u = D_k \phi$ for any fixed $k$.  Then by direct computation $u$ solves
\begin{gather}\label{eqn:pde-analysis-1}
(\delta_{ij} + z^i z^j) D_i D_j u + z^i D_i u - \lambda u = 0.
\end{gather}
In polar coordinates \eqref{eqn:pde-analysis-1} becomes
\begin{gather}
(1+r^2) \partial_r^2 u + ((m-1)/r - (\ell-3)r ) \partial_r u + \Delta_{\bbS^{m-1}} u/r^2 - \lambda u = 0,
\end{gather}
which can be written in the divergence form
\begin{gather}\label{eqn:pde-analysis-2}
\partial_r (g(r) \partial_r u) + \frac{g(r)}{1+r^2} (\Delta_S u / r^2 - \lambda u) = 0.
\end{gather}
where $g(r) = r^{m-1} (1 + r^2)^{-(m-2)/2}$  (this should not be surprising, since the original Jacobi equation is in divergence form).

If we multiply \eqref{eqn:pde-analysis-2} by $\zeta(r)^2 u$, where $\zeta \in C^\infty_0(\R_+)$, then we otain
\begin{gather}
\int_0^\infty \int_{\bbS^{m-1}} \frac{g}{r^2(1+r^2)} |\nabla u|^2 \zeta^2 + g \zeta^2 (\partial_r u)^2 + \frac{g\zeta^2 u^2 \lambda}{1+r^2}  d\omega dr \leq 10 \int_0^\infty \int_{\bbS^{m-1}} g(r) (\zeta')^2 u^2 d\omega dr.
\end{gather}
Here $\nabla$ indicates the covariant derivative on $\bbS^{m-1}$.  Since $r^{-2} |\nabla u|^2 \leq |Du|^2$ is bounded as $r \to 0$, we can in fact plug in any $\zeta \in C^\infty_0([0, \infty))$.  In particular, let us take $\zeta$ to be the usual log cutoff
\begin{gather}\label{eqn:log-cutoff}
\zeta(r) = \max\left\{ 2 - \frac{\log (\max\{r, \rho\})}{\log \rho}, 0 \right\}, \quad \rho \geq 4.
\end{gather}

Since $g(r) \leq 2r$ on $\spt \zeta'$, we can use Lemma \ref{lem:ode-est} to obtain
\begin{align}
\int_0^\rho \int_{\bbS^{m-1}} g(r) \left( \frac{r^{-2}|\nabla u|^2 + \lambda u^2}{1+r^2} + (\partial_r u)^2 \right) d\omega dr
&\leq \frac{c}{(\log \rho)^2} \int_\rho^{\rho^2} \int_{\bbS^{m-1}} r^{-1} |D\phi|^2 d\omega dr \\
&\leq \frac{c(\lambda, \mu)}{(\log \rho)^2} \int_{\rho/2}^{2\rho^2} \int_{\bbS^{m-1}} r^{-3} \phi^2 d\omega dr. \label{eqn:pde-analysis-3}
\end{align}

If $\lambda > 0$, then since $r^{-3} \leq r^{-1-\alpha}$ the integral in \eqref{eqn:pde-analysis-3} is bounded as $\rho \to \infty$.  This shows that $u = D_k\phi \equiv 0$ for any $k$, and hence $\phi$ is constant.  Using \eqref{eqn:spine-pde}, we see that the only constant solution when $\lambda \neq 1$ is $\phi \equiv 0$.

If $\lambda = 0$ then we can instead estimate \eqref{eqn:pde-analysis-3} as
\begin{gather}
\eqref{eqn:pde-analysis-3} \leq \frac{c}{(\log \rho)^2} \int_{\rho/2}^{2\rho^2} r^{-1} \int_{\bbS^{m-1}} |r^{-1} \phi - k|^2 d\omega dr + \frac{c}{(\log \rho)^2} \int_{\rho/2}^{2\rho^2} r^{-1}  dr \leq \frac{c}{\log \rho},
\end{gather}
for some constant $c$ independent of $\rho$.  Taking $\rho \to \infty$ gives that $\phi = a\cdot z + b$, but from \eqref{eqn:spine-pde} we see that necessarily $b = 0$.

Now for a general $\phi$, we can decompose $\phi = \sum_i \gamma_i(r) \phi_i(\omega)$ where each $\gamma_i(r) \phi_i(\omega)$ extends to a $C^\infty$ solution of \eqref{eqn:spine-pde} on $\R^m$, and continues to satisfy bounds \eqref{eqn:pde-bound-1}, \eqref{eqn:pde-bound-2}.  Therefore we can apply the previous logic to each $\gamma_i \phi_i$ to deduce the required result.

\vspace{5mm}
Now consider $m = 1$.  This is essentially the same, but easier.  We observe that $u = \phi'$ satisfies the ODE
\begin{gather}
(1 + z^2) u'' + z u' - \lambda u = 0,
\end{gather}
which can be written in divergence form as
\begin{gather}\label{eqn:pde-analysis-4}
(g(z) u')' - \frac{\lambda g(z)}{1 + z^2} u = 0,
\end{gather}
where $g(z) = (1+z^2)^{1/2}$.

Multiply \eqref{eqn:pde-analysis-4} by $u(z) \zeta^2(|z|)$, where $\zeta$ is the log cutoff \eqref{eqn:log-cutoff}, and observe that $\phi(|z|)$ solves \eqref{eqn:main-ode} on $\R \setminus \{0\}$.  Using Lemma \ref{lem:ode-est}, we obtain as before that
\begin{align}
\int_{-\rho}^\rho (u')^2 g + \frac{\lambda u^2 g}{1+z^2} dz
\leq \frac{10}{(\log \rho)^2} \int_{|z| \in [\rho, \rho^2]} |z|^{-1} (\phi')^2  dz 
\leq \frac{c(\lambda)}{(\log \rho)^2} \int_{ |z| \in [\rho/2, 2\rho^2]} |z|^{-3} \phi^2 dz,
\end{align}
and the proof proceeds as in the case $m \geq 2$.
\end{proof}

\subsection{Linear decay}

We first demonstrate the lower bound: if $v$ is orthogonal to linear fields, then at that scale $v$ must grow quantitatively more than $1$-homogeneously.

\begin{lemma}\label{lem:ortho-implies-growth}
Suppose $v : \refC \cap \overline{B_1} \to {\refC}^\perp$ is a smooth compatible Jacobi field, which is $L^2(\refC \cap B_1)$ orthogonal to every element in $\cL$, and satisfies the decay estimate
\begin{gather}\label{eqn:ortho-hyp}
\int_{\refC \cap B_{1/4}} \frac{|v - \kappa^\perp|^2}{r^{2+2-\alpha}} \leq \beta \int_{\refC \cap B_1} |v|^2,
\end{gather}
where $\kappa : (0, 1] \times B_1^m \to \R^{2+k}\times \{0\}$ is a chunky function with bound $|\kappa|^2 \leq \beta \int_{\refC \cap B_1} |v|^2$.

Then we have
\begin{gather}\label{eqn:ortho-growth}
\int_{\refC \cap B_1\setminus B_{1/10}} |\partial_R(v/R)|^2 \geq \frac{1}{c(\refC, \beta, \alpha)} \int_{\refC \cap B_1} |v|^2.
\end{gather}
\end{lemma}

\begin{proof}
Suppose, towards a contradiction, the Lemma fails: we have a sequence of smooth, compatible Jacobi field $v_i$ on $\refC \cap \overline{B_1}$, and associated chunky functions $\kappa_i$, which both satisfy the hypothesis of Lemma \ref{lem:ortho-implies-growth}, but each $v_i$ admits the bound
\begin{gather}\label{eqn:ortho-1}
\int_{\refC \cap B_1 \setminus B_{1/10}} R^{2-n} |\partial_R (v_i/R)|^2 \leq \eps_i \int_{\refC \cap B_1} |v_i|^2,
\end{gather}
with $\eps_i \to 0$.

Define the rescaled $\tilde v_i := ||v_i||_{L^2(\refC \cap B_1)}^{-1} v_i$.  Then $||\tilde v_i||_{L^2(\refC\cap B_1} = 1$ for all $i$, and using Lemma \ref{lem:jacobi-apriori-est} we can pass to a subsequence, and deduce the $\tilde v_i$ converge smoothly to on compact subsets of $\refC \cap B_1 \setminus \axis$ to some limit $\tilde v$.  We have strong convergence in $L^2(\refC \cap B_{1/4})$, since \eqref{eqn:ortho-hyp} implies
\begin{gather}\label{eqn:ortho-2}
\int_{\refC \cap B_{1/4} \setminus B_\delta(\axis)} |v|^2 \leq c(\refC, \beta) \delta^{2-\alpha} \quad \forall \delta > 0 .
\end{gather}
By compactness of chunky functions, we can assume $||v_i||_{L^2(\refC \cap B_1)}^{-1} \kappa_i \to \tilde\kappa$ uniformly on compact subsets of $\refC \cap B_1 \setminus \axis$.

The resulting $\tilde v$ is a compatible Jacobi field, which is $L^2(\refC \cap B_1)$-orthogonal to the linear fields, and satisfies the bound
\begin{gather}\label{eqn:ortho-3}
\int_{\refC \cap B_{1/4}} \frac{|\tilde v - \tilde\kappa^\perp|^2}{r^{2+2-\alpha}} < \infty,
\end{gather}
where $\tilde\kappa : (0, 1] \times B_1^m \to \R^{2+k}\times \{0\}$ is bounded and chunky.

Moreover, by our hypothesis \eqref{eqn:ortho-1}, $\tilde v$ extends to a $1$-homogeneous field on $\refC$.  By Theorem \ref{thm:1-homo-linear} and our bound \eqref{eqn:ortho-3} we deduce $\tilde v$ is linear, but this contradicts our orthogonality assumption unless $\tilde v \equiv 0$.

So $\tilde v_i \to 0$ uniformly on compact subsets of $B_1 \cap \refC \setminus \axis$.  But, by radial integration and \eqref{eqn:ortho-3}, one can show that
\begin{gather}
\int_{\refC \cap B_1 \setminus B_{1/10}} |\partial_R(\tilde v_i/R)|^2 \geq \frac{1}{c(n)} - c(\refC)(\eps^2 + \beta \delta^{2-\alpha}),
\end{gather}
whenever $\sup_{\refC \cap B_{1/10} \setminus B_{\delta}(\axis)} |\tilde v_i| \leq \eps$.  For $i >> 1$, this is a contradiction.
\end{proof}

\vspace{5mm}

We now prove Theorem \ref{thm:linear-decay}.
\begin{proof}[Proof of Theorem \ref{thm:linear-decay}]
From Lemma \ref{lem:ortho-implies-growth} and \eqref{eqn:linear-decay-hyp2} there is a constant $\beta_2 = \beta_2(\refC, \beta, \alpha)$ so that, for every $\rho \in [\theta, 1/10]$, 
\begin{gather}
\int_{\refC \cap B_{\rho/10}} R^{2-n} |\partial_R (v/R)|^2 \leq \beta\int_{\refC \cap B_\rho} |v_\rho|^2 \leq \beta \beta_2 \int_{\refC \cap B_\rho \setminus B_{\rho/10}} R^{2-n} |\partial_R(v/R)|^2.
\end{gather}
Therefore by hole-filling we obtain
\begin{gather}
\int_{\refC \cap B_{\rho/10}} R^{2-n} |\partial_R(v/R)|^2 \leq \frac{\beta\beta_2}{1+\beta\beta_2} \int_{\refC \cap B_\rho} R^{2-n} |\partial_R (v/R)|^2.
\end{gather}

Writing $\gamma = \frac{\beta\beta_2}{1+\beta\beta_2} < 1$, we can iterate the above inequality to obtain
\begin{gather}
\int_{\refC \cap B_{\theta}} R^{2-n} |\partial_R(v/R)|^2 \leq c(\gamma) \theta^\mu \int_{\refC \cap B_{1/40}} R^{2-n} |\partial_R(v/R)|^2,
\end{gather}
where $\mu = -\log(\gamma)/\log(10) > 0$.  Using Lemma \ref{lem:ortho-implies-growth} at scale $\theta$ and \eqref{eqn:linear-decay-hyp2} at scale $1/4$ completes the proof of Theorem \ref{thm:linear-decay}.
\end{proof}


\end{comment}


\section{Inhomgeoneous blow-ups}\label{sec:blow-up}

We finish proving the excess decay Theorem \ref{thm:main-decay}.  We shall demonstrate how blow-up sequences generate compatible Jacobi fields, and how integrability allows to remove the linear part of the limiting field at any fixed scale.  This allows us to apply the linear decay of Theorem \ref{thm:linear-decay} to prove non-linear excess decay.

As before we continue to work with a fixed $\refC = \refC_0^2 \times \R^m$, with ${\refC_0}^2 \subset \R^{2+k}$ polyhedral.

\subsection{Blowing-up}

We need a notion of convergence under varying domains.  Consider the sequence of domains
\begin{gather}
\Omega_i = \{ (x', x_{m+1}) \in B_1^m \times \R : 0 \leq x_{m+1} \leq 1 + f_i(x') \} \subset \R^{m+1},
\end{gather}
where $f_i : B_1^m \to \R$ is $C^{1,\alpha}$, and $|f_i|_{C^{1,\alpha}} \to 0$.

Suppose we have $u_i : \Omega_i \to \R$, with uniformly bounded $|u_i|_{C^{1,\alpha}(\Omega_i)} \leq \Lambda$.  Define the functions $\phi_i : B_1^m \times [0, 1] \to \Omega_i$ by setting
\begin{gather}
\phi_i(x', x_{m+1}) = (x', (1 + f_i(x')) x_{m+1} ).
\end{gather}
Then $\phi_i$ is a diffeomorphism for large $i$, and we can consider the functions $\hat u_i : B_1^m \times [0,1] \to R$ defined by $\hat u_i = u_i \circ \phi_i$.

Now by Arzela-Ascoli and convergence of $f_i$, we can find a $C^{1,\alpha}$ function $u : B_1^m \times [0,1] \to \R$, with $|u|_{C^{1,\alpha}} \leq \Lambda$, so that:
\begin{gather}\label{eqn:varying-domain-conv}
\hat u_i \to u \text{  in $C^{1,\alpha'}(B_1^m \times [0,1])$}, \quad \text{ and } \quad u_i \to u \text{ in $C^{1,\alpha'}_{loc}(B_1^m \times (0, 1))$},
\end{gather}
for any $\alpha' < \alpha$.

\vspace{5mm}

Let us now take $(M_i, \bC_i, \eps_i, \beta_i)$ a blow-up sequence w.r.t. $\refC = \refC_0 \times \R^m$.  By Lemma \ref{lem:poly-graph}, there are numbers $\tau_i \to 0$ so that (for $i >> 1$) we can decompose
\begin{gather}
M_i \cap B_{3/4} = \graph_{\bC_i}(u_i, f_i, \Omega_i), \quad B_{1/2} \setminus B_{\tau_i}(\axis) \subset \Omega_i, \quad |u_i|_{C^{1,\mu}} + |f_i|_{C^{1,\mu}} \leq \tau_i,
\end{gather}
as per Definition \ref{def:poly-graph}, where the $u_i$, $f_i$ satisfy estimates \eqref{eqn:point-est-poly-1}, \eqref{eqn:point-est-poly-2}, \eqref{eqn:integral-est-poly-1}, \eqref{eqn:integral-est-poly-2}.

Similarly, we can decompose
\begin{gather}
\bC_i = \graph_{\refC}(\phi_i, g_i, U_i), \quad B_{3/4} \subset U_i, \quad |\phi_i|_{C^{1,\mu}} + |g_i|_{C^{1,\mu}} \leq \tau_i,
\end{gather}
where we use the fact $\bC_i$ is also conical to extend $U_i$.  Here $\phi_i$, $g_i$ also satisfy estimates \eqref{eqn:point-est-poly-1}, \eqref{eqn:point-est-poly-2}, \eqref{eqn:integral-est-poly-1}, \eqref{eqn:integral-est-poly-2} of Lemma \ref{lem:poly-graph}.

Since each $\bC_i$ is also polyhedral, we have that both $\phi_i$ and $g_i$ are \emph{linear} functions on the domains $U_i$ in $\refC$.  In particular, we can extend $\phi_i$ to be defined on each plane $P(i)\times \R^m$ associated to the wedges, and note that we can say (trivially) that
\begin{gather}\label{eqn:linear-smooth}
|\phi_i|_{C^\infty} + |g_i|_{C^\infty} \to 0.
\end{gather}

Let us define $\tilde \Omega_i(j) \subset P(i)$ to be the domains where
\begin{gather}
\Omega_i(j) = \{x' + \phi_i(x') : x' \in \tilde\Omega_i(j) \}.
\end{gather}
Since every $f_i, \phi_i, g_i \to 0$ in $C^{1,\mu}_{loc}$, each domain $\tilde\Omega_i(j)$ is converging locally in $C^{1,\mu}(B_{1/2} \setminus \axis)$ to $W(i) \times\R^m$.

Now consider the rescaled graphs $v_i(j) : \tilde\Omega_i(j) \to {\bC_i}^\perp$ defined by
\begin{gather}
v_i(j)(x') = \beta_i^{-1} u_i(j)(x' + \phi_i(x')). 
\end{gather}

From Lemma \ref{lem:poly-graph} and the definition of blow-up sequence, the $v_i$ satisfy:
\begin{gather}\label{eqn:holder-bounds-v}
\limsup_i \sum_{j=1}^d \int_{\Omega_i(j)} |v_i|^2 < \infty, \quad \sup_{\tilde\Omega_i(j)} r^{n+2}(r^{-1}|v_i| + |Dv_i| + r^{\alpha}[Dv_i]_{\alpha,C} )^2 \leq c(\refC, \alpha).
\end{gather}
Therefore, using \eqref{eqn:linear-smooth}, after passing to a subsequence (which we will also denote by $i$) we can find a function $v : \refC \cap B_{1/2} \to {\refC}^\perp$ so that for each $j = 1, \ldots, d$, we have $C^{1,\mu'}$ convergence $v_i(j) \to v(j)$ locally in the sense of \eqref{eqn:varying-domain-conv}.  In particular, we have
\begin{gather}
v(j) \in C^{1,\mu}_{loc}( ( (W(j) \setminus \{0\}) \times \R^m) \cap B_{1/2}).
\end{gather}

We can then make the following
\begin{definition}
Let $(M_i, \bC_i, \eps_i, \beta_i)$ be the subsequence which gives convergence to $v$ as outlined above.  We then say that $v$ is the \emph{Jacobi field generated by $(M_i, \bC_i, \eps_i, \beta_i)$}.
\end{definition}

We shall demonstrate in the following Proposition that $v$ is a compatible Jacobi field on $\refC$ with good estimates.

\begin{prop}\label{prop:blow-up}
Let $(M_i^{2+m}, \bC_i, \eps_i, \beta_i)$ be a blow-up sequence w.r.t $\refC$, generating Jacobi field $v : \refC \cap B_{1/2} \to \refC^\perp$.  Then $v$ is compatible (in the sense of Definition \ref{def:compatible}), and moreover satisfies the following estimates: for every $\rho \leq 1/4$, we have

A) Strong $L^2$ convergence:
\begin{gather}
\int_{\refC \cap B_\rho} |v|^2 = \lim_i \beta^{-2}_i \int_{M_i \cap B_\rho} d_{\bC_i}^2 ;
\end{gather}

B) Non-concentration:
\begin{gather}
\rho^{2+2-1/2} \int_{\refC \cap B_{\rho/2}} \frac{|v - \kappa_\rho^\perp|^2}{r^{2+2-1/2}} \leq c(\refC) \rho^{-n-2} \int_{\refC \cap B_\rho} |v|^2,
\end{gather}
where $\kappa_\rho : (0,\rho]\times B^m_\rho \to \R^{\ell+k}\times \{0\}$ is a chunky function satisfying $|\kappa_\rho|^2 \leq c(\refC) \rho^{-n} \int_{\refC \cap B_\rho} |v|^2$;

C) Growth estimates:
\begin{gather}
\int_{\refC \cap B_{\rho/10}} R^{2-n} |\partial_R (v/R)|^2 \leq c(\refC) \int_{\refC \cap B_\rho} |v|^2.
\end{gather}
\end{prop}

\begin{remark}
Even though $v$ is smooth, convergence to $v$ may be only $C^{1,\alpha}$.
\end{remark}

\begin{proof}
We first show compatibility.  Let $\{e_p\}_{p=1}^n$ be an ON basis for the plane $P(j) \times \R^m$.  Using the first-variation formula and the definition of $\phi_i$, $u_i$, one obtains directly that
\begin{gather}
\int_{W(j)\times \R^m} \sum_{p=1}^n (D_p u_i(j))(x + \phi_i(j)(x)) \cdot D_p \zeta(x) = \int_{\spt \zeta} O(|Du_i|^2 + |Du_i||D\phi_i| + |H_M|),
\end{gather}
for any $\zeta \in C^\infty_c( ((\mathrm{int} W(j)\times \R^m)\cap B_{3/4}, \R^{n+k})$.  Therefore, using \eqref{eqn:holder-bounds-v}, and the definition of $v_i$ and blow-up sequence, we get that
\begin{gather}
\int_{W(j)\times \R^m} \sum_{p=1}^n D_p v \cdot D_p \zeta = 0
\end{gather}
for all such $\zeta$.  We deduce that $v(j)$ is harmonic on $(\mathrm{int}W(j)\times \R^m)\cap B_{1/2}$.

Write $L = \cup_{j} L(j)$ for the lines of $\refC_0$.  Pick any $X = (x, y) \in ( (L \setminus \{0\}) \times \R^m) \cap B_{1/2}$.  In view of Remarks \ref{rem:wedge-cont} and \ref{rem:looks-like-Y}, we can choose a fixed $\rho = \rho(X, \bC)$, a constant $c = c(m, k, \rho)$ and a sequence of roatations $q_i \to q \in SO(n+k)$, so that
\begin{gather}
(q_i(\rho^{-1}(M_i - X)), \bY\times \R^{m+1}, c \eps_i, \beta_i)
\end{gather}
is a blow-up sequence w.r.t $\bY\times \R^{m+1}$, generating Jacobi field
\begin{gather}
\tilde v(Y) := \rho^{-1} (q \circ v)( X + \rho q^{-1}(Y)).
\end{gather}
By Lemma \ref{lem:compatible-Y}, $\tilde v$ satisfies the required $C^0$ and $C^1$ compatability conditions in $B_{1/2}$, and therefore $v$ satisfies these conditions in $B_{\rho/2}(X)$.  Compatibility of $v$ now follows from Lemma \ref{lem:jacobi-apriori-est}.

We now prove properties A), B), C).  Fix $\rho \leq 1/4$, and recall $v_i$ as the approximating sequence which converges to $v$.  We first observe that
\begin{gather}
\sum_{j=1}^d \int_{\Omega_i(j) \cap B_\rho \setminus B_\delta(\axis)} |v_i(j)|^2 = O(\tau_i) + (1 + o(1)) \beta_i^{-2} \int_{M_i \cap B_\rho \setminus B_\delta(\axis))} d_{\bC_i}^2,
\end{gather}
since the Jacobian of $u_i(x' + \phi_i(x'))$ is $1 + o(1)$, and $|u_i(j)| = d_{\bC_i}$ away from $B_{10\tau_i}(L\times \R^m)$.

Therefore, by the $C^{1,\mu}$ convergence of $\tilde\Omega_i(j)$ and $v_i(j)$ (as per \eqref{eqn:varying-domain-conv}) we have
\begin{gather}
\int_{\refC \cap B_\rho \setminus B_\delta(\axis)} |v|^2 = \lim_{i \to \infty} \beta_i^{-2} \int_{M_i \cap B_\rho \setminus B_\delta(\axis))} d_{\bC_i}^2.
\end{gather}

On the other hand, by estimates \eqref{eqn:thm-spine-est} and \eqref{eqn:thm-k-est} we have for any $\delta \geq \tau$ and $i >> 1$:
\begin{gather}\label{eqn:blow-up-1}
\sum_j \int_{\Omega_i(j) \cap B_\delta(\axis)\setminus B_{\tau}(L\times \R^m)} |u_i(j)|^2 \leq c(\refC) \delta^{2-1/2} E(M_i, \bC, 0, 1).
\end{gather}
Write $\Gamma = \limsup_i \beta_i^{-2} E_{\eps_i}(M_i, \bC_i, 0, 1)$.  Passing to the limit in \eqref{eqn:blow-up-1}, and then taking $\tau \to 0$, we deduce
\begin{gather}\label{eqn:no-conc-v}
\int_{\refC \cap B_\rho \cap B_\delta(\axis)} |v|^2 \leq c(\refC, \Gamma) \delta^{2-1/2}.
\end{gather}
Similarly, we have by estimate \eqref{eqn:thm-spine-est} that (for $i >> 1$)
\begin{gather}\label{eqn:no-conc-d}
\beta_i^{-2} \int_{M_i \cap B_\rho \cap B_\delta(\axis)} d_{\bC_i}^2 \leq c(\refC, \Gamma) \delta^{2-1/2}.
\end{gather}
Since \eqref{eqn:no-conc-v}, \eqref{eqn:no-conc-d} are valid for any fixed $\delta$ (provided $i$ sufficiently large), we deduce the strong $L^2$ convergence of A).

Let us prove B).  Fix $\tau > 0$.  We can apply Theorem \ref{thm:l2-est} at scale $\rho$ to deduce that, for each $i >> 1$, we have a chunky functon $\kappa_{\rho,i} : (0,\rho] \times B_\rho^m \to \R^{\ell+k}\times\{0\}$, with the bound
\begin{gather}
|\kappa_{\rho,i}| \leq c(\refC) \rho^{-n} \int_{M \cap B_\rho} d_{\bC_i}^2 + c(\refC) \rho \beta_i^2 \Gamma \eps_i, 
\end{gather}
so that
\begin{gather}\label{eqn:conv-part-B}
\rho^{2+2-1/2} \sum_{j=1}^d \int_{\Omega_j(i) \cap B_{\rho/2} \setminus B_\tau(L \times \R^m)} \frac{|\beta_i^{-1} u_i(j) - \beta_i^{-1} \kappa_{\rho, i}^\perp|^2}{r^{2+2-1/2}} \leq c(\refC) \rho^{-n-2} \beta_i^{-2} \int_{M \cap B_\rho} d_{\bC_i}^2 + c(\refC) \rho \Gamma \eps_i.
\end{gather}

By compactness of chunky functions, we can find a subsequence $i'$ and a chunky function $\kappa_\rho$ so that $\beta_i^{-1} \kappa_{\rho,i} \to \kappa_\rho$ pointwise, and uniformly on $B_\rho \setminus B_\tau(L\times \R^m)$ (for any fixed $\tau > 0$).  Using A), we can therefore take the limit in $i'$ on each side of \eqref{eqn:conv-part-B}, to deduce
\begin{gather}
\rho^{2+2-1/2} \int_{\refC \cap B_{\rho/2} \setminus B_\tau(L \times \R^m)} \frac{|v - \kappa_\rho|^2}{r^{2+2-1/2}} \leq c(\refC) \rho^{-n-2}\int_{\refC \cap B_\rho} |v|^2.
\end{gather}
Taking $\tau \to 0$ gives B).

We show C).  From \eqref{eqn:thm-point-est}, we have for any $\tau > 0$ and $i >> 1$,
\begin{gather}
\sum_i \int_{\Omega_i(j) \cap B_{\rho/10} \setminus B_\tau(L\times \R^m)} R^{2-n} |\partial_R(u_i(j)/R)|^2 \leq c(\refC) \rho^{-n-2} \int_{M_i \cap B_\rho} d_{\bC_i}^2 + c(\refC) \rho \beta_i^2 \Gamma \eps_i.
\end{gather}
Therefore, using the $C^1$ convergence of $v_i(j)$ away from $\partial W(j) \times \R^m$, and part A) we have
\begin{gather}
\int_{\refC \cap B_{\rho/2} \setminus B_\tau(L\times \R^m)} R^{2-n} |\partial_R(v_i(j)/R)|^2 \leq c(\refC) \rho^{-n-2} \int_{\refC \cap B_\rho} |v|^2.
\end{gather}
Now take $\tau \to 0$ to deduce C).
\end{proof}

We demonstrate that Jacobi fields obtained through inhomogeneous blow-up limits are compatible.

\begin{lemma}\label{lem:compatible-Y}
Suppose $(M_i^{1+m}, \bY\times \R^m, \eps_i, \beta_i)$ is a blow-up sequence w.r.t $\bY\times \R^m$, generating Jacobi field $v : \refC \cap B_{1/2} \to (\refC)^\perp$.  Then for every $y \in B_{1/2}^m$, there is a vector $V \in \R^{n+k}$ so that
\begin{gather}
v(j)(0, y) = \pi_{Q(j)^\perp}(V) \quad j = 1, 2, 3, \quad \text{and} \quad \sum_{j=1}^3 \partial_n v(j)(0, y) = 0.
\end{gather}
\end{lemma}

\begin{proof}
Fix some $y \in B_{1/2}^m$, and let $V_i$ be the (unique) point in $\sing M_i \cap (\R^{\ell+k}\times \{y\}) \cap B_1$.  So, we have
\begin{gather}
u_i(j)(f_i(j) (0, y)) = \pi_{Q(j)^\perp}(V_i),
\end{gather}
and from the $120^\circ$ angle condition we have
\begin{gather}\label{eqn:V-bound}
|V_i| \leq \sum_{j=1}^3 |u_i(j)(f_i(j)(0, y))|.
\end{gather}

From the blow-up procedure we have $u_i(j)(f_i(j)(0, y)) \to v(j)(0, y)$, and from \eqref{eqn:V-bound} we can pass to a subsequence $i'$ so that $V_{i'} \to V$.  Then we have
\begin{gather}
v(j)(0, y) = \pi_{Q(j)^\perp}(V).
\end{gather}
This proves the $C^0$-compatability.

We prove the $C^1$ condition.  Our proof follows \cite{simon1}, but we additionally exploit the stationarity of $\bY\times \R^m$ (as a technical aside, we mention that \cite{simon1} only requires stationarity away from the axis, while we stipulate stationary through the axis; for unions of half-planes this restricts not only the allowable surfaces but also the notion of integrability).  Let $\zeta(r, y)$ be any function with $\partial_r \zeta \equiv 0$ near $\{0\}\times \R^m$, and $\spt\zeta \subset B_{1/10}(X)$ for some $X$.  For ease of notation write $E_i = E(M_i, \bY\times \R^{1+m}, 1)$.

After rotation we can fix one of the $H(j) \equiv \R_+ \times \{0\}^k \times \R^m$.  So, coordinates on $H(j)$ are $(x^1, y^1, \ldots, y^m)$, and coordinates on $H(j)^\perp \equiv \R^k$ are $(x^2, \ldots, x^{1+k})$.  Ensuring $i >> 1$, we can assume $M_i$ is graphical over $H(j) \cap (B_{3/4} \setminus B_{\tau/2}(\partial H(j))$, with graphing function $u_i(j)$.  Write
\begin{gather}
U(j) = H(j)\cap (B_{1/2} \setminus B_{\tau}(\partial H(j))). 
\end{gather}

Let us drop the $i$ and $j$ indices momentarily.  Write $h^{pq}$ for the inverse of $h_{pq} = \delta_{pq} + D_p u \cdot D_q u$, and $\sqrt{h}$ for the determinant of $h_{pq}$.  Then we have
\begin{align*}
\int_{u(U)} \nabla x^1 \cdot \nabla \zeta
&= \underbrace{\int_{U} \sqrt{h} h^{11} \partial_{x^1} ( \zeta(\sqrt{x^2 + |u|^2}, y) )}_{=:I_1} \\
&\quad  + \underbrace{\int_U \sum_{p=1}^m \sqrt{h} h^{1, 1+p} \partial_{y^p} (\zeta(\sqrt{x^2 + |u|^2}, y)}_{=:I_2}) .
\end{align*}

Since the cross terms $|h^{1,1+p}| \leq c|Du|^2$, we can bound the second term directly as
\begin{align*}
\left| I_2 \right|
&\leq \int_{U} \left| \sum_{p=1}^m \sqrt{h} h^{1, 1+p} ( (\partial_r \zeta) \frac{u \partial_{y^p} u}{\sqrt{x^2 + |u|^2}} + (\partial_{y^p} \zeta))\right| \\
&\leq c(n, \beta) (|\partial_r \zeta| + |D_y \zeta|) \int_{U} |Du|^2 (1 + |u| |Du|) \\
&\leq c(n,\tau, \beta, \zeta) E_i . 
\end{align*}

The first term we don't bound quite explicitly.  Recalling that $|u| \leq \beta |x|$, we have
\begin{align*}
\left| I_1 - \int_{U} (\partial_r \zeta)(x, y) \right| 
&= \left|\int_{U} \sqrt{h} h^{11} (\partial_r \zeta)(\sqrt{x^2 + |u|^2}, y) \frac{x}{\sqrt{x^2 + |u|^2}} - \int_{U} (\partial_r \zeta)(x, y) \right| \\
&\leq \int_{U} | \sqrt{h} h^{11} - 1| \left| (\partial_r \zeta)(\sqrt{x^2 + |u|^2}, y) \right|  \\
&\quad + \int_{U} \left| (\partial_r \zeta)(\sqrt{x^2 + |u|^2}, y) \frac{x}{\sqrt{x^2 + |u|^2}} - (\partial_r \zeta)(x, y) \right| \\
&\leq c(n, \beta) |\partial_r \zeta| \int_{U} |Du|^2 + c(n, \beta) |\partial_r^2 \zeta| \int_{U(i)} |u|^2 \\
&\leq c(n, \beta, \tau, \zeta) E_i . 
\end{align*}

For $j = 2, \ldots, 1+k$, we also have
\begin{align}
\left| \int_{u(U)} \nabla x^j \cdot \nabla \zeta - \int_U D u^j \cdot D \zeta \right| 
&\leq \int_U |h^{pq} - \delta^{pq}| D_p u^j D_q (\zeta(\sqrt{x^2 + |u|^2}, y)) \\
&\quad + \int_U |D u^j| |D (\zeta(\sqrt{x^2 + |u|^2}, y)) - D \zeta(x, y) | \\
&\leq c(n, \beta) (|\partial_r \zeta| + |D_y \zeta| + |D^2 \zeta|) \int_U |Du|^2 + |u|^2 \\
&\leq c(n, \beta, \tau, \zeta) E_i. 
\end{align}

Therefore, turning indices back on we have the coordinate-free expression
\begin{align}
\int_{u_i(j)(U(j))} \pi_{\R^{1+k}\times\{0\}}(\nabla \zeta) \label{eqn:compat-sum-errors}
&= \int_{u_i(j)(U(j))} \sum_{p=1}^{1+k} (e_p \cdot \nabla \zeta) e_p \\
&= -\int_{U(j)} (\partial_r \zeta)(r, y) n(j) + \sum_q D_q u_i(j) \cdot D_q \zeta + R_i(j), 
\end{align}
Here $q$ sums over the coordinates on $H(j)$ (so, $x^1, y^1,\dots, y^m$), $n(j)$ is the outwards conormal of $\partial H(j) \subset H(j)$, and 
\begin{gather}
R_i(j) \leq c(n, \beta, \tau, \zeta) E_i. 
\end{gather}

We can identify any $U(j)$ and $U(j')$ by a rotation, and thereby view the integrand \eqref{eqn:compat-sum-errors} as defined on a fixed $U(1) \equiv U$.  Since $\sum_{j=1}^3 n(j) = 0$, this gives
\begin{align}
\sum_{j=1}^3 \int_{u_i(j)(U(j))} \pi_{\R^{1+k}\times\{0\}} (\nabla \zeta) 
&= \int_U  \sum_q D_q (\sum_{j=1}^3 u_i(i)) \cdot D_q \zeta + \sum_{j=1}^3 R_i(j),
\end{align}
where again $q$ sums over coordinates in $H(j)$.

On the other hand, provided $i$ is sufficiently large we can always ensure $\tau(\zeta)$ is sufficiently small so that $\partial_r \equiv 0$ on $V =  B_{5\tau}(\{0\} \times \R^m)$.  In particular, we have $\pi_{\R^{1+k}\times\{0\}}(D\zeta) = 0$ on $V$.  Therefore, we use the $L^2$-estimates of \cite[Theorem 3.1]{simon1} to deduce
\begin{align}
\left| \int_{M_i \cap V} e_i \cdot \nabla \zeta \right| 
&\leq \int_{M_i \cap V} | \pi_{\R^{1+k}\times\{0\}}(e_i) \cdot \pi_{M^T}(D\zeta)| \\
&= \int_{M_i \cap V} | \pi_{\R^{1+k}\times \{0\}}(e_i) \cdot (-\pi_{M^\perp}(D\zeta))| \\
&\leq \int_{M_i \cap V} | \pi_{M^\perp}(\pi_{\{0\}\times \R^m}(D\zeta))| \\
&\leq c(n, \zeta) \sqrt{t} \left( \int_{M_i \cap V \cap \spt \zeta} <M^\perp, \{0\}\times \R^m>^2 \right)^{1/2} \\
&\leq c(n, \zeta) \sqrt{t E_i} . 
\end{align}

Since we can ensure $|u_i(j)| \leq \beta |x| \leq |x|/100$ on $U(j)$, what this amounts to is that, for $e_p$ an ON basis of $\R^{1+k}\times \{0\}$, 
\begin{align}
R_1 = \sum_{p=1}^{1+k} \int_{M_i} \mathrm{div}(\zeta e_p) e_p 
&= \int_{M_i} \pi_{\R^{1+k}\times \{0\}}(\nabla \zeta) \\
&= \sum_{j=1}^3 \int_{u_i(j)(U(j))} \pi_{\R^{1+k}\times\{0\}}(\nabla \zeta) + S \\
&= \int_U \sum_q (\sum_{j=1}^3 D_q u_i(j))D_q \zeta + R_2 + S, \label{eqn:compat-2}
\end{align}
where $|R_1| + |R_2| \leq c(n, \zeta)  E_i$ and $|S| \leq c(n, \zeta) \sqrt{tE_i}$.

Multiply \eqref{eqn:compat-2} by $\beta_i^{-1}$, and by hypothesis $\beta_i^{-1} u_i(j) \to v(j)$ in $C^1$ on $U$, where $v$ is the generated Jacobi field.  Therefore, we obtain
\begin{gather}
0 = \int_U \sum_q \sum_{j=1}^3 D_q v(j) D_q \zeta + S
\end{gather}
for all $U$, and $|S| \leq c(n,\zeta)\sqrt{t}$.  Now take $t \to 0$, to deduce
\begin{gather}\label{eqn:compat-3}
0 = \int_{H} \sum_q \sum_{j=1}^3 D_q v(j) D_q \zeta,
\end{gather}
where we identify all the $H(j) \equiv H$ together via rotation, and $q$ sums over coordinates $(x^1, y^1, \ldots, y^m)$.

Let us write $\tilde v$ for the even extension of $\sum_{j=1}^3 v(j)$ to $Q \equiv Q(1) \equiv \R \times \{0\}^k \times \R^m$.  The above condition \eqref{eqn:compat-3} implies that
\begin{gather}\label{eqn:compat-4}
\int_Q \tilde v \Delta \zeta = 0
\end{gather}
for every $\zeta(r, y)$ with $\zeta(r, y) = \zeta(-r, y)$, and supported in $B_{1/10}(X)$ for some $X$.  But \eqref{eqn:compat-4} trivially holds for $\zeta$ which are odd in $r$, and therefore $\tilde v$ is weakly harmonic.  So in fact $\tilde v$ is smooth, and we deduce $\partial_r \tilde v = 0$.
\end{proof}

\subsection{Killing the linear part} We demonstrate that when $\refC_0$ is integrable (as per Definition \ref{def:integrable}), we can adjust the blow-up sequence to obtain a field that has no linear component.  Recall the notation that if $v$ is a compatible Jacobi field on $\refC$, then $v_\rho := v - \psi_\rho$, where $\psi_\rho$ is the $L^2(\refC \cap B_\rho)$-projection to $\cL$.

\begin{prop}\label{prop:kill-linear}
Let $(M_i, \refC, \eps_i, \beta_i)$ be a blow-up sequence w.r.t $\refC$, generating Jacobi field $v : \refC \cap B_{1/2} \to {\refC}^\perp$.  Suppose $\refC_0$ is integrable, and fix $\theta \in (0, 1/4]$.  Write $\Gamma = \limsup_i \beta_i^{-2} E_{\eps_i}(M_i, \refC, 1)$.

Then there is a constant $\gamma(\theta, \refC, \Gamma)$ so that the following holds: given any $\rho \in [\theta, 1/4]$, we can find a sequence of rotations $q_i \in SO(n+k)$, satisfying $|q_i - Id| \leq \gamma \beta_i$, so that $(M_i, q_i(\bC), \eps_i + \gamma \beta_i, \beta_i)$ is a blow-up sequence w.r.t $\refC$, generating the Jacobi field $v_\rho$.  In particular, we have the estimates:

A) Strong $L^2$ convergence:
\begin{gather}
\int_{\refC \cap B_\rho} |v_\rho|^2 = \lim_i \beta^{-2}_i \int_{M_i \cap B_\rho} d_{q_i(\bC)}^2 ;
\end{gather}

B) Non-concentration:
\begin{gather}
\rho^{2+2-1/2} \int_{\refC \cap B_{\rho/2}} \frac{|v_\rho - \kappa_{\rho, \psi}^\perp|^2}{r^{2+2-1/2}} \leq c(\refC) \rho^{-n-2} \int_{\refC \cap B_\rho} |v_\rho|^2,
\end{gather}
where $\kappa_{\rho,\psi} : (0,\rho]\times B^m_\rho \to \R^{\ell+k}\times \{0\}$ is a chunky function satisfying $|\kappa_{\rho,\psi}|^2 \leq c(\refC) \rho^{-n} \int_{\refC \cap B_\rho} |v_\rho|^2$;

C) Growth estimates:
\begin{gather}\label{eqn:growth-no-linear}
\int_{\refC \cap B_{\rho/10}} R^{2-n} |\partial_R (v/R)|^2 \leq c(\refC) \int_{\refC \cap B_\rho} |v_\rho|^2.
\end{gather}
\end{prop}

\begin{remark}
Of course $\partial_R (\psi_\rho/R) \equiv 0$, so \eqref{eqn:growth-no-linear} holds for both $v$ and $v_\rho$.
\end{remark}

\begin{remark}\label{rem:blow-up-rot}
Due to our particular notion of integrability (by rotations), we can always assume our initial blow-up sequence has $\bC_i \equiv \refC$ fixed, and thereby reduce to the hypothesis of Proposition \ref{prop:kill-linear}.  Proposition \ref{prop:kill-linear} holds also for general blow-up sequences (and the ``actual'' notion of integrability), using the fact that integrability is essentially an open condition on cones, but we will not need this.  See \cite{simon1} pages 601-602.
\end{remark}

\begin{proof}
Fix a $\rho \in [\theta, 1/4]$.  Using Proposition \ref{prop:blow-up} part A) we have
\begin{gather}
\rho^{-n-2} \int_{\refC \cap B_\rho} |\psi_\rho|^2 \leq \Gamma^2 \theta^{-n-2} ,
\end{gather}
and therefore, since $\psi_\rho$ is linear, we obtain
\begin{gather}
\sup_{\refC \cap B_1} |\psi_\rho| \leq c(\refC) \Gamma^2 \theta^{-n-2}.
\end{gather}

By integrability of $\refC_0$, the definition of $\cL$, and Theorem \ref{thm:1-homo-linear}, there is a skew-symmetric matrix $A_\rho : \R^{n+k} \to \R^{n+k}$ so that $\psi_\rho = \pi_{{\refC}^\perp} \circ A_\rho$, and $|A_\rho| \leq c(\refC, \Gamma, \theta)$.  We can therefore find a sequence of rotations $q_i \in SO(n+k)$, with $|q_i - Id| \leq c(\refC, \Gamma, \theta) \beta_i$, so that if we write
\begin{gather}
q_i(\bC) = \graph_{\refC}(\phi_i, g_i, U_i),
\end{gather}
then each $\phi_i(j) : P(j)\times \R^m \to P(j)^\perp$ is a linear function satisfying
\begin{gather}
\phi_i(j) = \beta_i \psi_\rho(j) + o(\beta_j).
\end{gather}

Now we have, for $i >> 1$, 
\begin{gather}
\int_{M \cap B_1} d_{q_i(\bC)}^2 \leq \int_{M \cap B_1} d_{\bC}^2 + c(\bC, \Gamma, \theta) \beta_i^2,
\end{gather}
and since increasing $\eps_i$ does not change the property of being a blow-up sequence, we see that $(M_i, q_i(\bC), \eps_i + \gamma \beta_i, \beta_i)$ is also a blow-up sequence.

We demonstrate that this blows up to $v_\rho$ as required.  As in Section \ref{sec:blow-up}, let us write
\begin{gather}
M_i = \graph_{q_i(\bC)}(u_i, f_i, \Omega_i), \quad M_i = \graph_{\bC}(u^*_i, f^*_i, \Omega^*_i) ,
\end{gather}
and define domains $\tilde \Omega_i(j) \subset W(j)$ by the condition that
\begin{gather}
\Omega_i(j) = \{ x' + \phi_i(x') : x' \in \tilde\Omega_i(j) \}.
\end{gather}

Now by elementary geometry we have that for every $x' \in \tilde\Omega_i(j) \cap \Omega^*_i(j)$, we have
\begin{gather}
u_i(x' + \phi_i(x')) = u^*_i(x') - \phi_i(x) + O( (|u_i| + |Du_i|) |D\phi_i|) = u^*_i(x') - \beta_i \psi_\rho + o(\beta_i).
\end{gather}
Since both $\tilde\Omega_i(j)$ and $\Omega^*_i(j)$ converge to the wedge $W(j)$ as $i \to \infty$, and since $\beta_i^{-1} u^*_i \to v$ by assumption, the blow-up of $u_i$ as per Section \ref{sec:blow-up} will yield the field $v_\rho = v - \psi_\rho$.
\end{proof}

\subsection{Non-linear decay: Proof of Theorem \ref{thm:main-decay}} Propositions \ref{prop:blow-up} and \ref{prop:kill-linear} allow us to use the linear decay of Jacobi fields as in Section \ref{sec:jacobi} to prove non-linear decay of $M$.

\begin{proof}[Proof of Theorem \ref{thm:main-decay}]
Fix $\theta \in (0, 1/4]$.  We first take $c_0(\refC) \equiv c(\refC)$ and $\gamma(\refC, \theta) \equiv \gamma(\refC, \theta, \Gamma = 1)$ to be the constants from Proposition \ref{prop:kill-linear}.  Now take $\mu(\refC) \equiv \mu(\refC, \beta = c_0, \alpha = 1/2)$ the constant from Theorem \ref{thm:linear-decay}.  We proceed by contradiction:

Suppose we had a sequence $M_i \in \cN_{\eps_i}(\refC)$ satisfying $E_{\eps_i}(M_i, \bC, 0, 1) \leq \eps_i^2$ and the $\eps_i/10$-no-holes condition, with $\eps_i \to 0$, but admitting for some $c_i \to \infty$ the bound
\begin{gather}
E_{\eps_i}(M_i, q(\refC), 0, \theta) \geq c_i \theta^\mu E_{\eps_i}(M_i, \refC, 0, 1) 
\end{gather}
for every $q \in SO(n+k)$ satisfying $|q - Id| \leq \gamma E_{\eps_i}(M_i, \refC, 0, 1)^{1/2}$.

Let us set $\beta_i^2 = E_{\eps_i}(M_i, \bC, 0, 1)$, and thereby obtain a blow-up sequence $(M_i, \refC, \eps_i, \beta_i)$, generating some Jacobi field $v : \refC \cap B_{1/2} \to \refC^\perp$.  By Proposition \ref{prop:kill-linear} and integrability of $\refC_0$, $v$ satisfies the hypotheses of Theorem \ref{thm:linear-decay} at scale $B_{1/2}$, with $\beta = c_0(\refC)$, and $\alpha = 1/2$.  Therefore we have the decay estimate
\begin{gather}
\theta^{-n-2} \int_{\refC \cap B_\theta} |v_\theta|^2 \leq c(\refC)\theta^\mu \int_{\refC \cap B_{1/2}} |v_{1/2}|^2 \leq c(\refC) \theta^\mu ,
\end{gather}
and a sequence of $q_i \in SO(n+k)$, with $|q_i - Id| \leq \gamma \beta_i$, so using the strong $L^2$-convergence of Proposition \ref{prop:kill-linear} A), we have for $i >> 1$
\begin{align}
E_{\eps_i}(M_i, q_i(\refC), 0, \theta) 
&\equiv \theta^{-n-2} \int_{M_i \cap B_\theta} d_{q_i(\refC)}^2 + \eps_i^{-1} \theta ||H_{M_i}||_{L^\infty(B_\theta)} \\
&\leq \left( 2\theta^{-n-2} \int_{\refC \cap B_\theta} |v_\theta|^2 \right) E_{\eps_i}(M_i, \bC, 0, 1) + \eps_i^{-1} \theta^\mu ||H_{M_i}||_{L^\infty(B_1)} \\
&\leq 4c(\refC) \theta^\mu E_{\eps_i}(M_i, \refC, 0, 1) .
\end{align}
For large $i$ this is a contradiction.
\end{proof}


 \end{comment}

\section{Equiangular nets in $\bbS^2$}\label{sec:nets}

We demonstrate that certain polyhedral cones are integrable, in the sense of Definition \ref{def:integrable}.  First, we demonstrate that $\bY\times \R$ and $\bT$ (under certain circumstances) admit no hole conditions.

\subsection{No-holes for $\bY$ and $\bT$}

The $\bY^1 \times \R^m$ cone is very special, in that closeness to thise cone always guarantees the existence of good density points.  No extra assumptions on the class or structure of the varifold are necessary.

\begin{prop}\label{prop:no-holes-Y}
There is an $\eps(m, k, \delta)$ so that if $M^{1+m} \subset \R^{1+k+m}$ lies in $M \in \cN_{\eps}(\bY\times \R^m)$, then $M$ satisfies the $\delta$-no-holes condition in $B_{1/2}$ w.r.t. $\bY \times \R^m$.
\end{prop}

\begin{proof}
By Lemma \ref{lem:global-graph}, provided $\eps$ is sufficiently small $M \cap B_{3/4}\setminus B_{\delta}(\axis)$ is a $C^{1,\alpha}$-perturbation of $\bY\times \R^m$.  We claim that
\begin{gather}
\sing M \cap (\R^{1+k}\times \{y\}) \cap B_{1/2} \neq \emptyset \quad \forall y \in B_{1/2}^m .
\end{gather}
Otherwise, since $\sing M$ is relatively closed, by Sard's theorem, we could choose a $y^*$ arbitrarily near $y$ so that $M \cap (\R^{1+k}\times \{y^*\}) \cap B_{1/2}$ would consist of a smooth $1$-manifold having three boundary components, which is impossible.

Therefore, using Almgren's stratification we have for $\haus^m$-a.e. $y \in B_{1/2}^m$ a singular point $X_y \in \sing M \cap (B_\delta(0)^{1+k}\times \{y\})$ which is $m$-symmetric.  So there is a tangent cone at $X_y$ which is either a multiplicity $\geq 2$ plane, or a union of $\geq 3$ half-planes, either of which has density $\geq \theta_{\bY}(0)$.
\end{proof}

\vspace{5mm}

Unfortunately, the tetrahedral cones $\bT^2 \times \R^m$ do not admit so nice a property, without imposing further restrictions: we can find piecewise-smooth varifolds of bounded mean curvature which look very close to $\bT$ at scale $B_1$, but which only have singularities of type $\bY\times \R$.  To rule this out one can enforce a boundary/orientability structure.

\begin{lemma}\label{lem:T-density}
Let $\bC = \bC_0^2 \times \R^m \subset \R^{3+m}$, where $\bC_0$ is $2$-dimensional, stationary and singular.  If (up to rotation) $\bC_0$ is not a multiplicity $1$ plane or the $\bY \times \R$, then we have
\begin{gather}
\theta_{\bC}(0) \geq \theta_{\bT}(0).
\end{gather}
\end{lemma}

\begin{proof}
If $\bC_0$ is planar, then it must be with multiplicity $\geq 2 > \theta_{\bT}(0)$.  If $\bC_0$ has $1$-degree of symmetry, then since we are not regular nor are we the $\bY$, then $\bC_0$ must consist of $\geq 4$ half-planes meeting along an edge, which also has multplicity $\geq 2$.

Suppose $\bC_0$ has no symmetries.  Consider the geodesic net $\Gamma := \bC_0 \cap \partial B_1 \subset \bbS^2$.  If any geodesic has multiplicity $\geq 2$, or any junction has $\geq 4$ vertices, then $\theta_{\bC}(0) \geq 2$ and we are done.  Let us suppose therefore that $\Gamma$ consists only of multiplicity-$1$ geodesics, which meet at $120^\circ$.

These nets are classified, and listed in the following subsection.  One can readily verify that the net with least length, aside from the circle and $\bY$, is the tetrahedral net.
\end{proof}

\begin{lemma}\label{lem:T-cross-section-sing}
Let $M^2$ be a set in $\R^3$ which coincides with $\bT^2$ in $B_1 \setminus B_\delta$.  Suppose $\haus^2 \llcorner M$ is an integral varifold with an associated cycle structure in $B_1$.  Then there is a point $x \in M \cap B_1$, so that $M$ near $x$ is \emph{not} a $C^1$ perturbation of $\R^2$ or $\bY \times \R$.
\end{lemma}

\begin{proof}
By assumption $M$ divides the annulus $B_{1} \setminus B_\delta$ into four regions $A_1, A_2, A_3, A_4$.  Any two $A_i$, $A_j$ share a boundary wedge $W \subset \bT$.

Suppose, towards a contradiction, that around every point $M$ is locally a $C^1$ perturbation of $\R^2$ or $\bY\times \R$.  Then $M \cap B_1$ consists of a finite collection of $C^1$ embedded surfaces $M_i$ meeting at $120^\circ$ along $C^1$ embedded curves $\gamma_i$.  Since $M$ coincides with $\bT$ outside $B_\delta$, we see that up to renumbering the curves $\gamma_1$, $\gamma_2$ start and end at vertices of $\bT \cap S^2$, while curves $\gamma_3, \gamma_4, \ldots$ must be closed.  See figure \ref{fig:no-holes} for an idealized picture.

\begin{figure}\label{fig:no-holes}
\centering
\includegraphics{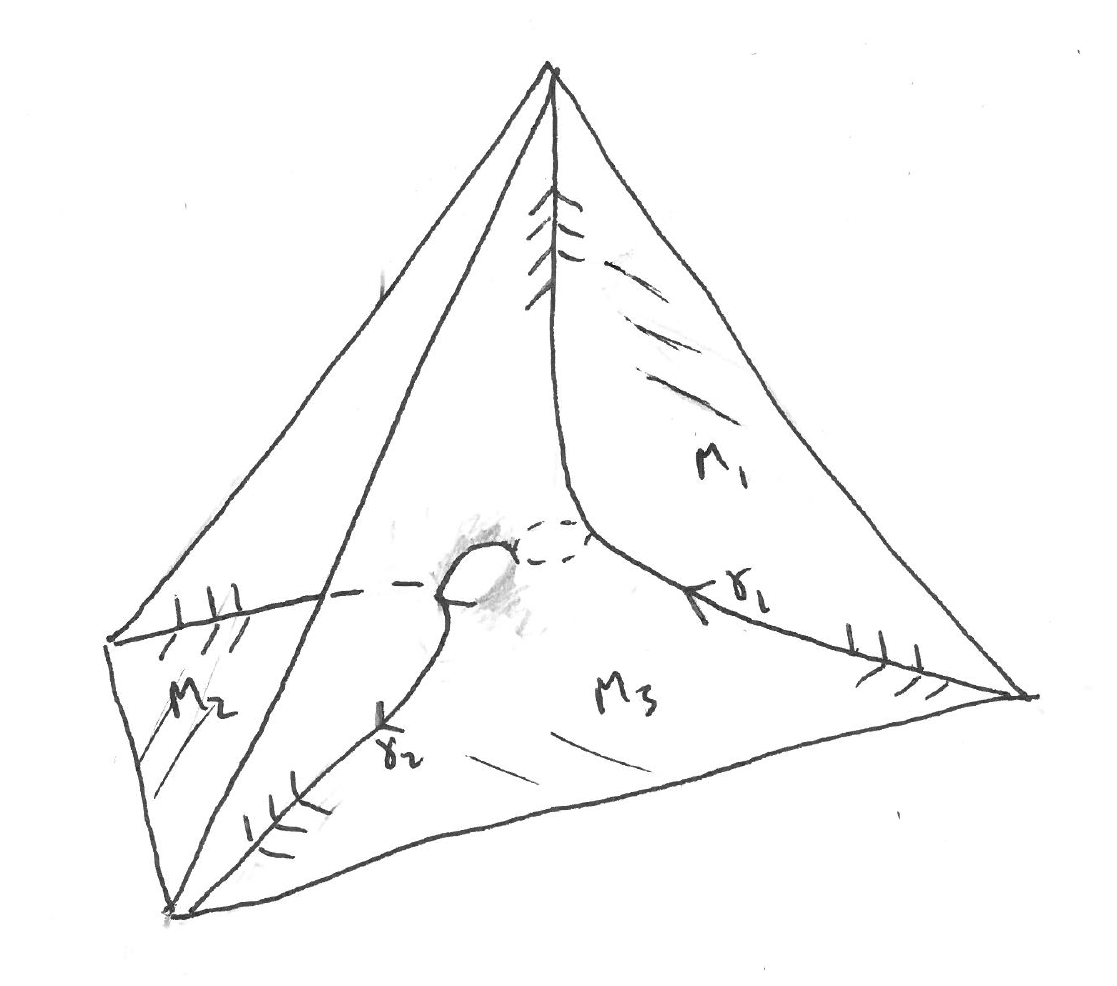}
\caption{A surfaces with only $\bY$-type singularities which coincides with $\bT$ outside a small ball.}
\end{figure}

We can assume $\gamma_1$ starts at the vertex adjoined by regions $A_1, A_2, A_3$, while $\gamma_1$ ends at the vertex adjoined by regions $A_1, A_2, A_4$.  A small tubular neighborhood of $\gamma_1$ is diffeomorphic $\bY\times \R$, and therefore if we push $\gamma_1$ away from any bounding surface in the conormal direction, the resulting curve $\hat \gamma_1$ induces a path connecting $A_i$ ($i = 1, 2, 3$) to some $A_j$ ($j = 1, 2, 4$).  After relabeling as necessary, we can thicken $\hat \gamma_1$ to obtain an open set $A$, disjoint from $M$, with $A \supset A_3 \cup A_4$.

Since each associated current is codimension $1$ and without boundary, we can assume WLOG that $\haus^2 \llcorner M$ is associated to a countable union of \emph{boundaries} $\partial [U_i]$, where $U_i$ are open sets, and we take the boundaries as $3$-currents.  From the above we have $A \subset U_i$ or $A \cap U_i = \emptyset$ for every $i$.  But now if $W$ is the boundary wedge shared by $A_3, A_4$, then the previous sentence implies
\begin{gather}
W \cap (B_{1/2} \setminus B_{2\delta}) \cap \spt \partial [U_i] = \emptyset \quad \forall i.
\end{gather}
And so $W$ \emph{cannot} be part of $M$.  This is a contradiction.
\end{proof}

\begin{remark}\label{rem:general-no-holes}
If one could show either curve $\gamma_1$ or $\gamma_2$ is unknotted (as in Figure \ref{fig:no-holes}), then one could construct a Lipschitz deformation of $M$ onto two faces of the solid tetrahedron (plus one edge).  This would prove Lemma \ref{lem:T-cross-section-sing} for general $(\mass, \eps, \delta)$-sets (at least for $\eps$ sufficiently small) without any extra orientation or codimension requirements.  Unfortunately, we have very little idea whether Lemma \ref{lem:T-density} holds in general codimension.
\end{remark} 

\begin{prop}\label{prop:no-holes}
There is an $\eps(m, \delta)$ so that the following holds.  Let $M^{n = 2+m} \subset \R^{3+m}$ be an integral varifold with associated cycle structure in $B_1$, and suppose $M \in \cN_{\eps}(\bT^2 \times \R^m)$.  Then $M$ satisfies the $\delta$-no-holes condition in $B_{1/2}$.
\end{prop}

\begin{proof}
By Lemma \ref{lem:poly-global-graph}, $M \cap B_{3/4} \setminus B_\delta(\axis)$ is a $C^{1,\alpha}$-perturbation of $\bT\times \R^m$, for $\eps(m, \delta)$ sufficiently small.  So
\begin{gather}\label{eqn:no-holes-T-1}
\sing M \cap B_{3/4} \subset B_\delta(\axis),
\end{gather}
and there is no loss in assuming $M \cap B_{3/4} \setminus B_\delta(\axis)$ coincides with $\bT \times \R^m$.

We claim that, for every $y \in B_{1/2}^m$, there is some singular point
\begin{gather}\label{eqn:no-holes-T-2}
X_y \in \sing M \cap (\R^{3}\times \{y\}) \cap B_{1/2}
\end{gather}
which is not a (multiplicity-1) $\bY\times \R^{1+m}$.  We prove this by contradiction.

First, observe that by Simon's regularity Theorem \ref{thm:Y-reg}, the set of singular points which are not a multiplicity-1 $\bY\times \R^{1+m}$ is relatively closed in $\sing M$, and hence closed.  Therefore, if the claim failed, it would fail for $y$ in some open set $U$.  Using Allard's and Simon's regularity we obtain that $M \cap (\R^{3}\times U)$ consists of embedded, multiplicity-one $C^1$ $n$-surfaces, meeting at $120^0$ along embedded $C^1$ $(n-1)$-surfaces.

Therefore by Sard's theorem, for a.e. $y \in U$, the $M \cap (\R^3 \times \{y\})$ consists of embedded $C^1$ surfaces meeting at $120^\circ$ along embedded $C^1$ curves, which coincides with $\bT^2$ in an annulus.  However, by slicing we also have that for a.e. $y \in U$, $\haus^2 \llcorner M \cap (\R^3 \times \{y\})$ has an associated cycle structure in $B_{1/4}(0, y)$, contradicting Lemma \ref{lem:T-cross-section-sing}.  This proves the claim.

The Proposition is completed by combining \eqref{eqn:no-holes-T-1} and the above claim with Lemma \ref{lem:T-density}.
\end{proof}

\subsection{Integrability}\label{sec:int}


We establish integrability of those polyhedral cones which arise from an equiangular geodesic net in $\bbS^2$.  As discussed in Remark \ref{rem:maybe-non-int}, it seems possible to us that in higher-codimension there exist non-integrable polyhedral cones (for either definition of integrability).  Indeed, even in the codimension-$1$ case we are unable to give a general abstract proof, but instead we make use of the classification of equiangular geodesics nets in $\bbS^2$ due to \cite{lamarle}, \cite{heppes} and proceed on a case-by-case basis.

\begin{theorem}
Suppose $\bC^2 \subset \R^3 \subset \R^{2+k}$ is a polyhedral cone.  Then $\bC$ is integrable in $\R^{2+k}$.  In particular the tetrahedron $\bT^2 \subset \R^{2+k}$ is integrable.
\end{theorem}


\begin{proof}
Fix a polyhedral cone $\bC^2 \subset \R^3 \subset \R^{2+k}$, composed of wedges $\cup_{i=1}^d W(i)$.  Write $\Gamma = \bC \cap \bbS^{1+k}$ for the corresponding equiangular geodesic net, and $\ell(i) \equiv W(i) \cap \bbS^{1+k}$ for the geodesic segments.  After relabeling as necessary we can assume $\ell(1), \ell(2), \ell(3)$ share a common vertex.

Let $v : \bC \to \bC^\perp$ be a linear, compatible Jacobi field.  We wish to show that $v = \pi_{\bC^\perp} \circ A$ for some skew-symmetric matrix $A : \R^{2+k} \to \R^{2+k}$.  From Proposition \ref{prop:baby-linear} we know this holds locally, in the sense that there is a skew-symmetric $A_0$, so that
\begin{gather}
v(i) = \pi_{\bC^\perp} \circ A_0 \quad i = 1, 2, 3.
\end{gather}
Therefore, by considering the field $v - \pi_{\bC^\perp}\circ A_0$, we can and shall reduce to the case when $v(1) = v(2) = v(3) = 0$.

In fact we shall prove that any linear, compatible Jacobi field $v$ satsifying $v(1) = v(2) = v(3) = 0$ must be identically zero.  It is reasonable to expect this to be true, as the $v(i)$s with their compatibility conditions effectively form a system of linear equations, and one can easily verify that the total number of variables equals the total number of conditions (equals $2kd$).  However an abstract counting argument seems insufficient to establish $v \equiv 0$, as the linear independence of this system depends strongly on both the global topology and geometry of the underlying net.  Thankfully, the possible nets $\Gamma$ are very well understood, and we can prove our assertion on a case-by-case basis.

\vspace{5mm}

Let us first assume $k = 1$.  For each $i$, fix a unit speed paramterization of $\ell(i)$, and write $\hat \ell(i)$ for the induced unit tangent vector.  We take $\hat\ell(i) \wedge \hat x$ to be the choice of unit normal to $W(i)$ (and hence an orientation on $W(i)^\perp$), where $\hat x$ is the unit position vector.

Define scalar functions $f(i) : \ell(i) \cong [0, \mathrm{length}(\ell(i))] \to \R$ by setting
\begin{gather}
f(i)(\theta) = v(i)(\theta) \cdot (\hat \ell(i) \wedge \hat x).
\end{gather}
Then each $f(i)$ completely determines $v(i)$, and takes the form
\begin{gather}\label{eqn:form-of-f}
f(i)(\theta) = a(i) \sin(\theta) + b(i) \cos(\theta), \quad \theta \in \ell(i) \cong [0, \mathrm{length}(\ell(i))],
\end{gather}
for real constants $a(i), b(i)$.

We shall prove that every $f(i)$ must be identically $0$.  Recall that by hypothesis we have
\begin{gather}\label{eqn:scalar-initial-conds}
f(1) = f(2) = f(3) = 0,
\end{gather}
while using Lemma \ref{lem:vect-vs-scalar-cond}, the $C^0$- and $C^1$-compatibility conditions on $v$ imply that
\begin{gather}\label{eqn:scalar-compat-conds}
\sum_{j=1}^3 (n(i_j) \cdot \hat \ell(i_j)) f(i_j)(p) = 0, \quad \text{ and } \quad f'(i_1)(p) = f'(i_2)(p) = f'(i_3)(p),
\end{gather}
whenever $\ell(i_1), \ell(i_2), \ell(i_3)$ share a common vertex $p$.  Here $n(i)$ is the outer conormal of $\ell(i)$, and $f'(i) \equiv \partial_{\hat \ell(i)} f(i)$ is the derivative in the direction $\hat\ell(i)$.
%

\vspace{5mm}

From the work of \cite{lamarle}, \cite{heppes}, and since $\bC \subset \R^3$ cannot have additional symmetries, then up to rotation $\Gamma$ can be only one of $8$ possible nets.  We prove integrability case-by-case by establishing that the corresponding system of $f(i)$s satisfying \eqref{eqn:scalar-initial-conds}, \eqref{eqn:scalar-compat-conds} must vanish.  In each case we give a topological diagram indicating numbering, orientation, and length (a single arrow indicates length $\theta_1$, a double arrow indicates $\theta_2$, etc.).  We will additionally use the following notation: if $p$ is the vertex joining edges $\ell(1), \ell(2), \ell(3)$ (e.g.), then we refer to $p$ by the triple $(1,2,3)$.

The possible nets (presented in the same order as in \cite{taylor}), with their corresponding proofs of integrability, are as follows.  Each edge length is given to $3$ decimal places.
\begin{enumerate}

\item \textbf{Regular tetrahedron}, having $6$ edges, each of length $\theta_1 = 109.471^\circ$.

We can apply the $C^1$ condition \eqref{eqn:scalar-compat-conds} at each end of $\ell(4)$ to obtain $f(4)'(0) = f(4)'(\theta_1) = 0$.  We deduce $f(4) \equiv 0$, and by symmetry we have $f(i) \equiv 0$ for all $i$.

\begin{figure}
\centering
\includegraphics{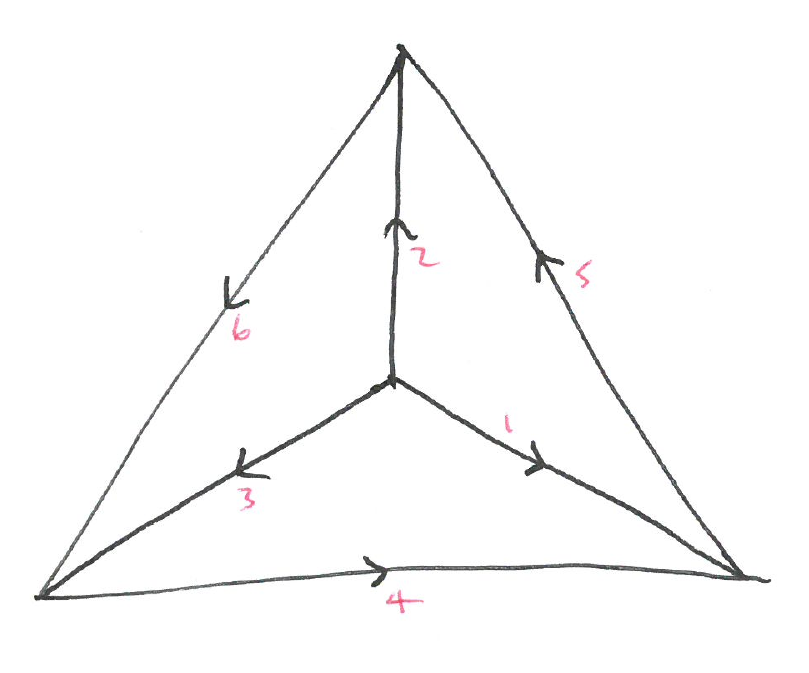}
\caption{Regular tetrahedron}
\end{figure}

\item \textbf{Regular cube}, having $12$ edges of length $\theta_1 = 70.529^\circ$.

Applying the $C^0$ and $C^1$ conditions \eqref{eqn:scalar-compat-conds}, and using that all edges have the same length, gives directly the relations
\begin{align}
&f(6) = -A \cos(\theta), \quad f(7) = A \cos(\theta), \quad f(9) = A \cos(\theta), \quad f(10) = -A \cos(\theta), \\
&\quad f(5) = -A\cos(\theta), \quad f(4) = A \cos(\theta),
\end{align}
where $A$ is the same constant.  But then, applying \eqref{eqn:scalar-compat-conds} at vertex $(4, 10, 12)$ gives the relation $A \cos(\theta_1) = -A \cos(\theta_1)$, which can only hold if $A = 0$.  By symmetry we deduce that every $f(i) \equiv 0$.

\begin{figure}
\centering
\includegraphics[width=3.5in]{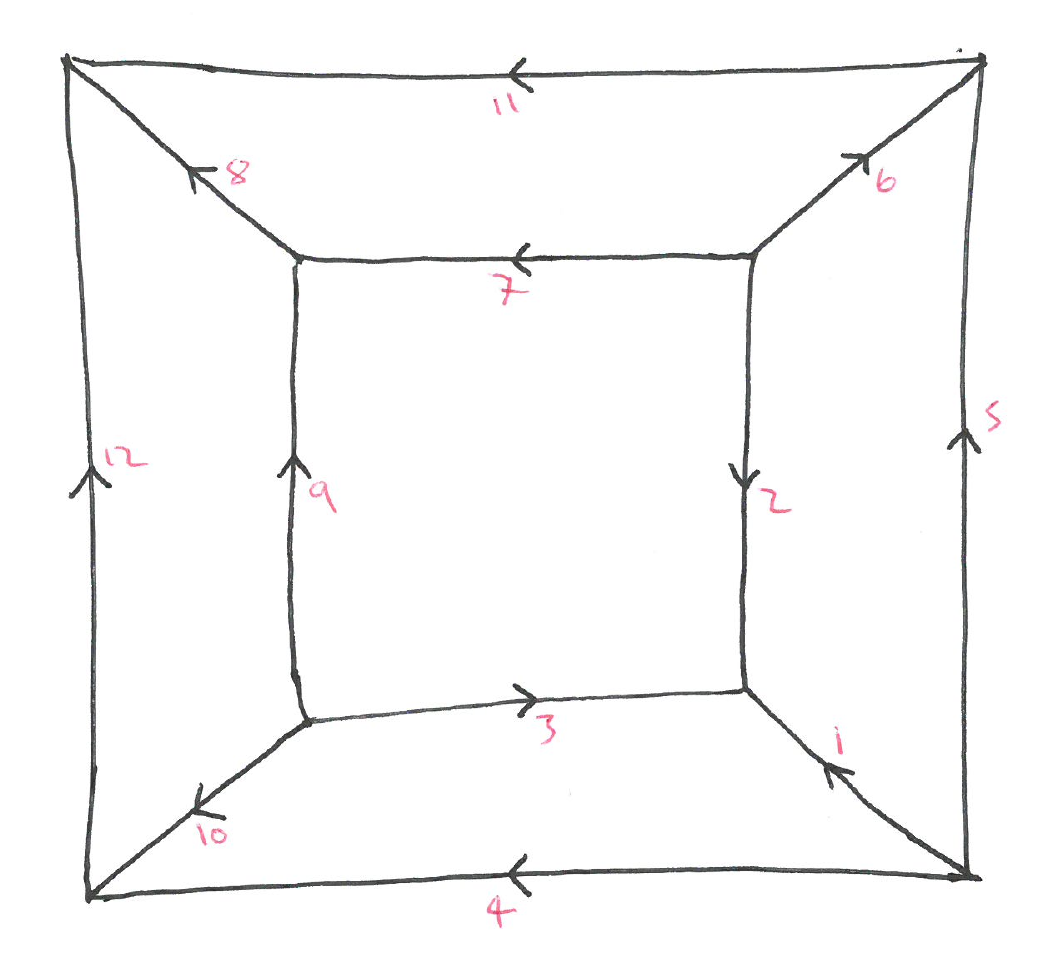}
\caption{Regular cube}
\end{figure}

\item \textbf{Prism over regular pentagon}, forming $15$ edges: ``with the pentagonal arcs having length $\theta_1 = 41.810^\circ$ and the other arcs being of length $\theta_2 = 105.245^\circ$.''

By the same reasoning as in the cube, taking into account the different lengths $\theta_1$, $\theta_2$, we have
\begin{align}
&f(6) = A \cos(\theta), \quad f(5) = -A \cos(\theta), \quad f(14) = A \cos(\theta), \quad f(11) = -A\cos(\theta),
\end{align}
for some constant $A$.  We can therefore apply the $C^1$ condition at each end of $\ell(9)$, to see that
\begin{gather}
f(9) = -A \sin(\theta_1) \sin(\theta) - A(\cos(\theta_1) + 1)\cos(\theta).
\end{gather}
Apply both conditions at vertex $7,5,4$ to obtain
\begin{gather}
f(7) = A \sin(\theta_2)\sin(\theta) - A(\cos(\theta_2) + \sin(\theta_2) \frac{\cos(\theta_1)}{\sin(\theta_1)} )\cos(\theta).
\end{gather}
These, together with $f(6)$, give three conditions on $f(8)$, and we obtain the relation
\begin{gather}
A(4 \sin \theta_2 \cos\theta_1 + 2\sin\theta_1\cos\theta_2 - \sin\theta_2) = 0.
\end{gather}
The term in the brackets is $-3.5$, to one decimal place.  We deduce that $A = 0$, and it's straightforward to verify that $f(i) \equiv 0$ for every $i$.

\begin{figure}
\centering
\includegraphics[width=4.5in]{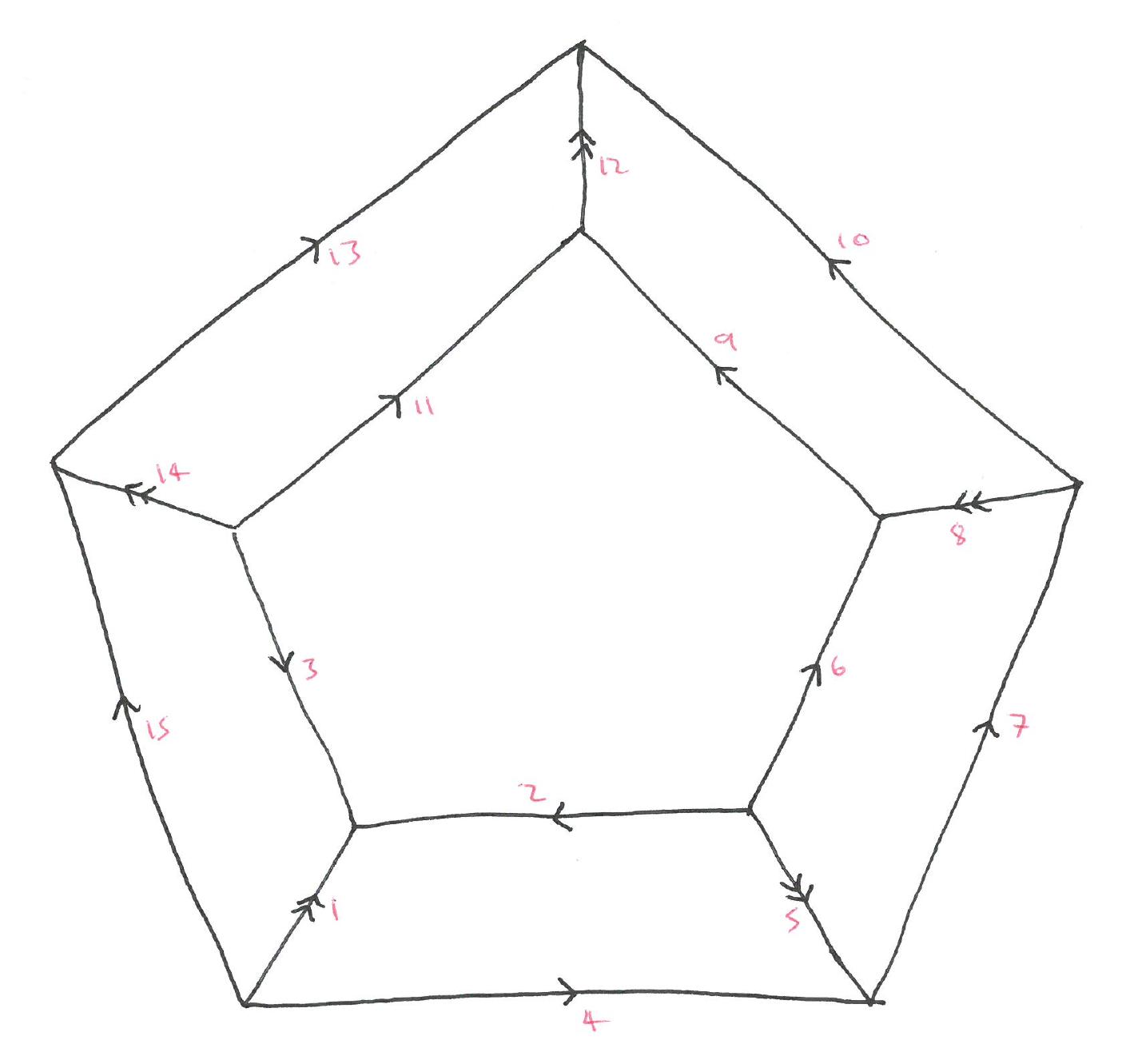}
\caption{Prism over regular pentagon}
\end{figure}

\item \textbf{Prism over a regular triangle}, forming $9$ edges: ``the triangular arcs being of length $109.471^\circ$ and the other arcs of length $38.942^\circ$.''

By same reasoning as the tetrahedron, we can apply the $C^1$ condition on each side of $\ell(7)$ to see $f(7) = 0$.  Apply both $C^0$- and $C^1$-condition at vertex $(2,6,7)$ to obtain $f(6)(0) = f(6)'(0) = 0$, and hence $f(6) = 0$.  Similarly, we have $f(8) = 0$.  We then deduce directly that $f(5) = f(9) = f(4) = 0$.

\begin{figure}
\centering
\includegraphics[width=3.5in]{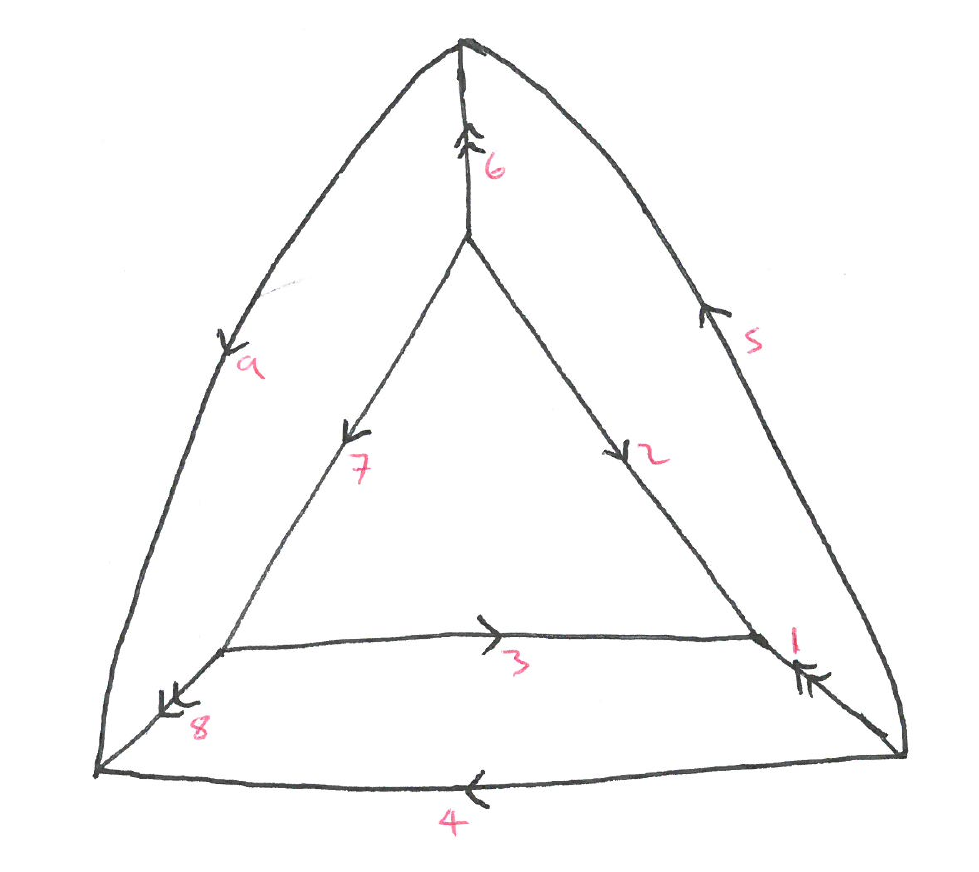}
\caption{Prism over regular triangle}
\end{figure}

\item \textbf{Regular dodecahedron}, having $30$ edges, each of length $\theta_1 =  41.810^\circ$.

\begin{figure}
\centering
\includegraphics[width=5.5in]{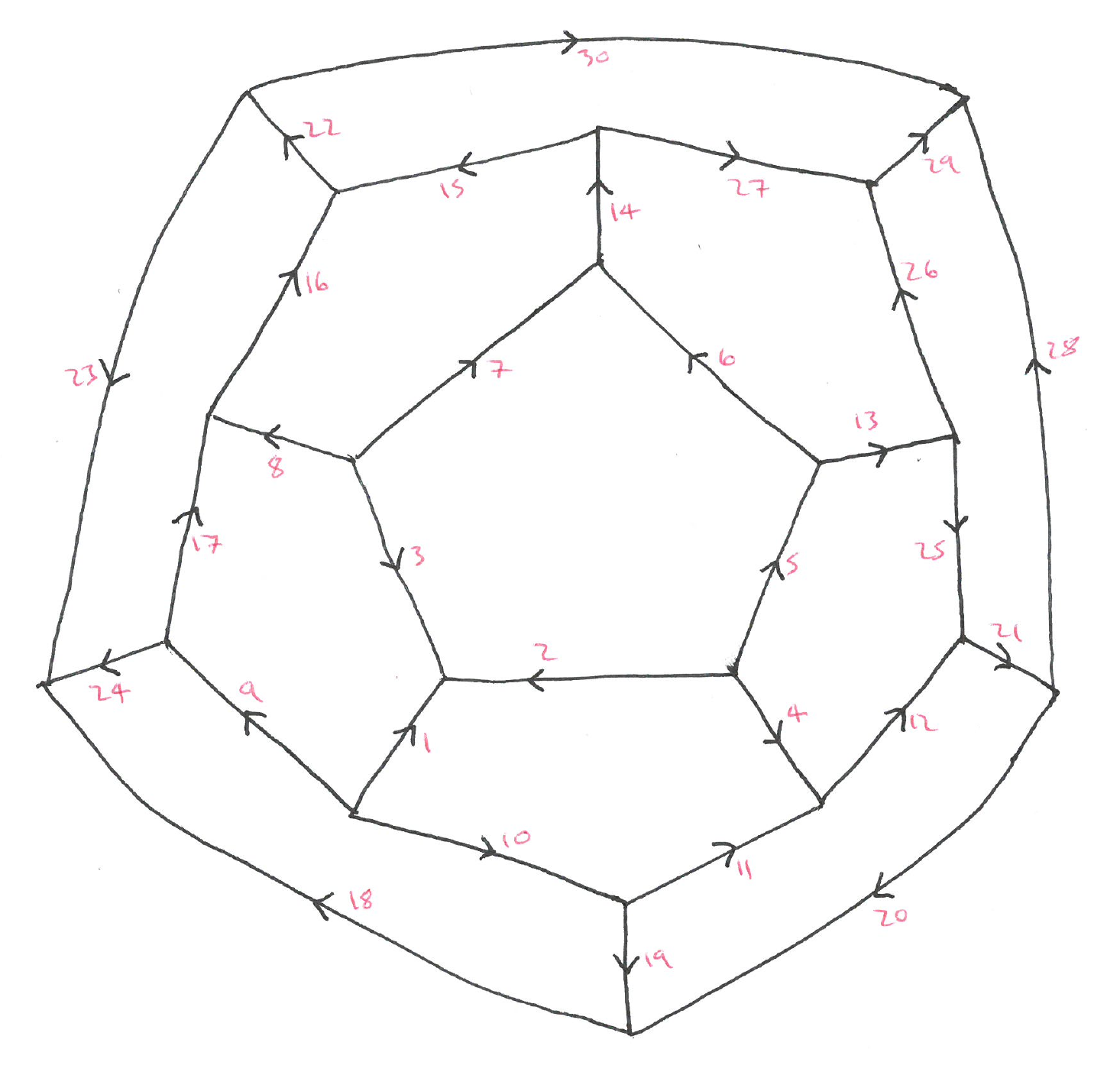}
\caption{Regular dodecahdron}
\end{figure}

We have immediately the equations
\begin{align}
&f(5) = A \cos\theta, \quad f(4) = -A\cos\theta, \quad f(10) = B\cos\theta, \quad f(9) = -B\cos\theta,\\
&\quad f(8) = G \cos\theta, \quad f(7) = -G\cos\theta,
\end{align}
for some constants $A, B, G$.  By symmetry it will suffice to show that $A = B = G = 0$.  We obtain, using the above and the compatibility conditions,
\begin{align}
&f(11) = -B\sin\theta_1 \sin\theta - (B\cos\theta_1 + A)\cos\theta, \\
&f(17) = B\sin\theta_1 \sin\theta + (G + B\cos\theta_1)\cos\theta ,\\
&f(6) = -A\sin\theta_1 \sin\theta - (A\cos\theta_1 + G)\cos\theta , \\
& f(13) = -A\sin\theta_1 \sin\theta + (G + 2A\cos\theta_1)\cos\theta , \\
& f(12) = A\sin\theta_1 \sin\theta - (2A\cos\theta_1 + B)\cos\theta, \\
& f(24) = B\sin\theta_1 \sin\theta - (2B\cos\theta_1 + G)\cos\theta, \\
& f(19) = -B\sin\theta_1\sin\theta + (A+2B\cos\theta_1)\cos\theta, \\
& f(25) = -(G + 3A\cos\theta_1)\sin\theta_1 \sin\theta - (G\cos\theta_1 + 3A\cos^2\theta_1 + 3A\cos\theta_1 + B)\cos\theta, 
\end{align}
and
\begin{align}
f(21) = (3A\sin\theta_1\cos\theta_1 + B\sin\theta_1)\sin\theta + (-G-2B\cos\theta_1 - 6A\cos^2\theta_1 - 3A\cos\theta_1 + A) \cos\theta.
\end{align}
And
\begin{align}
f(20) &= \left[ A(9\cos^2\theta_1 + 3\cos\theta_1\sin\theta_1 - \sin\theta_1) + 3B\sin\theta_1\cos\theta_1 + G\sin\theta_1 \right] \sin\theta \\
&\quad \left[ A(9\cos^3\theta_1 + 3\cos^2\theta_1 - \cos\theta_1 + 1) + B(3\cos^2\theta_1 + 3\cos\theta_1) + G\cos\theta_1 \right] \cos\theta .
\end{align}

From $19$ and $24$, we obtain
\begin{gather}
f(18) = -(3B\sin\theta_1 \cos\theta_1 + A\sin\theta_1)\sin\theta - (A\cos\theta_1 + 3B\cos^2\theta_1 + 3B\cos\theta_1 + G)\cos\theta.
\end{gather}
But we have additionally $20$, which implies the relation:
\begin{align}
2G = A(-9\cos^2\theta_1 - 6\cos\theta_1 + 1) + B(-9\cos^2\theta_1 - 6\cos\theta_1 + 1).
\end{align}

We work upwards.  We have
\begin{align}
&f(14) = G\sin\theta_1\sin\theta - (2G\cos\theta_1 + A)\cos\theta, \\
&f(16) = -G\sin\theta_1 \sin\theta + (B + 2G\cos\theta_1)\cos\theta, \\
&f(15) = (3G\cos\theta_1\sin\theta_1 + A\sin\theta_1)\sin\theta + (3G\cos^2\theta_1 + 3G\cos\theta_1 + A\cos\theta_1 + B)\cos\theta
\end{align}
And
\begin{align}
&f(26) = -(3A\sin\theta_1\cos\theta_1 + G\sin\theta_1)\sin\theta + (2G\cos\theta_1 + B + 6A\cos^2\theta_1 + 3A\cos\theta_1 - A)\cos\theta
\end{align}

We calculate $27$.  Using $14$, $15$, we obtain
\begin{gather}
f(27) = (3G\sin\theta_1\cos\theta_1 + A\sin\theta_1) \sin\theta + ( G - 6G\cos^2\theta_1 - 3G\cos\theta_1 - 2A\cos\theta_1 - B)\cos\theta.
\end{gather}
Combining this with $26$ gives the relation:
\begin{gather}
2B = A( -9 \cos^2\theta_1 - 6\cos\theta_1 + 1) + G (-9\cos^2\theta_1 - 6\cos\theta_1 + 1).
\end{gather}

Let us proceed to the left.  We have
\begin{align}
f(22) = -(3G\sin\theta_1\cos\theta_1 + B\sin\theta_1)\sin\theta + (6G\cos^2\theta_1+3G\cos\theta_1 - G + A + 2B\cos\theta_1)\cos\theta.
\end{align}
Using $22$ and $24$ we obtain
\begin{align}
&f(23) = \left[ G(-9\cos^2\theta_1\sin\theta_1 - 3\cos\theta_1\sin\theta_1 + \sin\theta_1) - A\sin\theta_1 - 3B\sin\theta_1\cos\theta_1 \right] \sin\theta \\
&\quad + \left[ G(-9\cos^3\theta_1 - 3\cos^2\theta_1 + \cos\theta_1 - 1) - A\cos\theta_1 - B(3\cos^2\theta_1 + 3\cos\theta_1)\right]\cos\theta.
\end{align}
But now using additionally $18$, we obtain the relation
\begin{gather}
2A = G(-9\cos^2\theta_1 - 6\cos\theta_1 + 1) + B(-9\cos^2\theta_1 - 6\cos\theta_1 + 1).
\end{gather}

We thus have the three equations
\begin{gather}\label{eqn:dodec-final}
A = \alpha(B + G), \quad B = \alpha(A + G), \quad G = \alpha(A + B),
\end{gather}
where $\alpha = 1.5$ (to one decimal place).  One easily verifies the only solution to \eqref{eqn:dodec-final} is when $A = B = G = 0$, and by symmetry we deduce that $f(i) = 0$ for every $i$.


\item \textbf{Two regular quadrilaterals and eight equal pentagons}, forming $24$ edges: ``each quadrilateral surrounded by four pentagons, and each pentagons surrounded by four pentagons and one quadrilateral, the quadrilateral arcs being of length $\theta_2 = 70.529^\circ$, the arcs adjacent to no quadrilateral vertex being of length $\theta_3 = 52.448^\circ$, and the remaining edges being of length $\theta_1 = 21.428^\circ$.''

\begin{figure}
\centering
\includegraphics[width=5.5in]{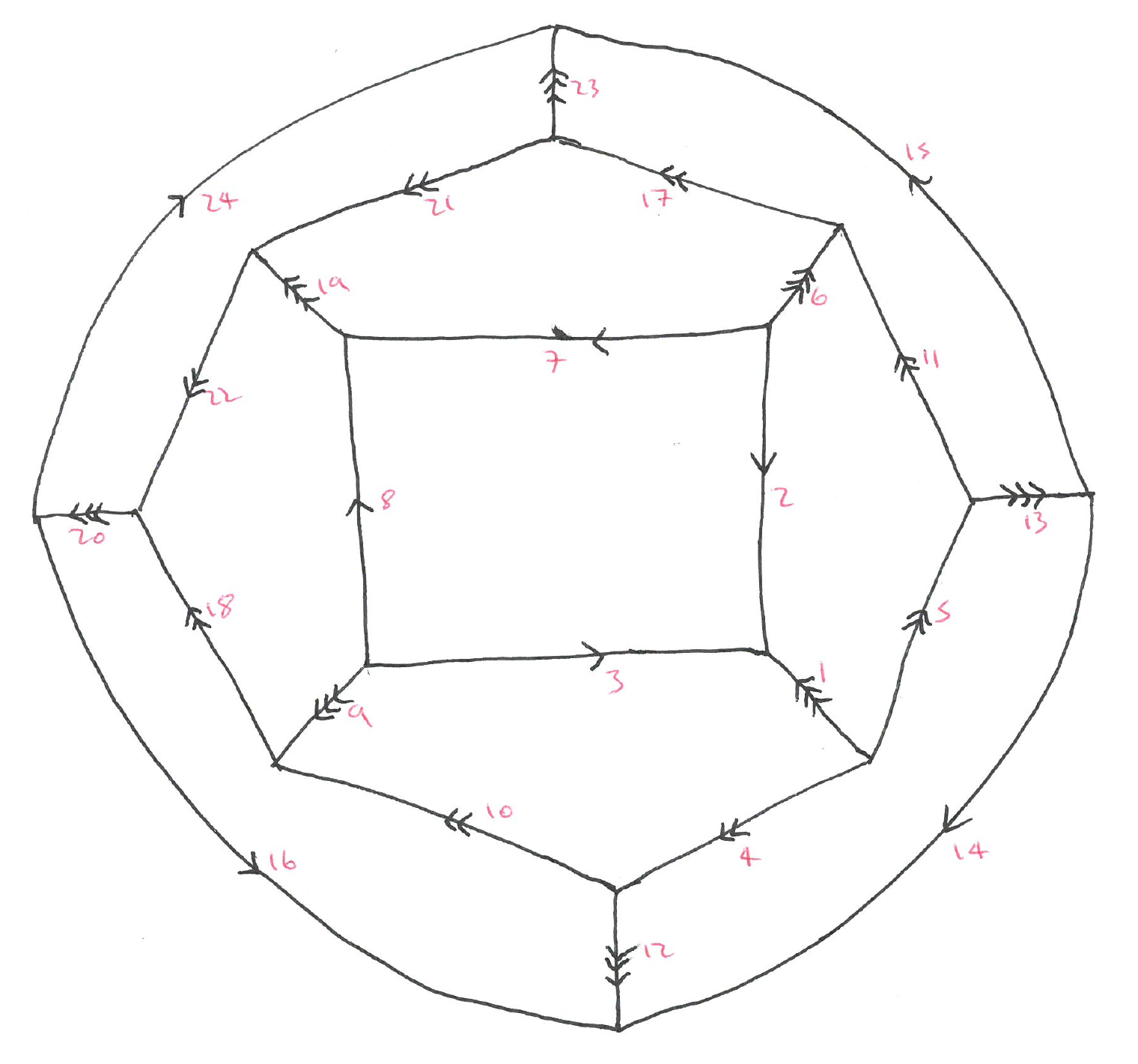}
\caption{Two regular quadrilaterals and eight equal pentagons}
\end{figure}


We have directly that
\begin{align}
&f(4) = -A\cos\theta, \quad f(5) = A\cos\theta, \\
&\quad f(6) = -B\cos\theta, \quad f(7) = B\cos\theta, \quad f(8) = B\cos\theta, \quad f(9) = -B\cos\theta,
\end{align}
for some constants $A, B$.  Using the compatability conditions at various vertices, we obtain
\begin{align}
&f(10) = A\sin\theta_3\sin\theta + (A\cos\theta_3 - B \frac{\sin\theta_1}{\sin\theta_3}) \cos\theta\\
&f(11) = -A\sin\theta_3\sin\theta - (A\cos\theta_3 + B \frac{\sin\theta_1}{\sin\theta_3})\cos\theta \\
&f(12) = A\sin\theta_3\sin\theta + (-2A\cos\theta_3 + B\frac{\sin\theta_1}{\sin\theta_3})\cos\theta \\
&f(13) = -A\sin\theta_3\sin\theta + (2A\cos\theta_3 + B\frac{\sin\theta_1}{\sin\theta_3})\cos\theta \\
&f(18) = B\sin\theta_1\sin\theta + (-B\cos\theta_1 - B \sin\theta_1 \frac{\cos\theta_3}{\sin\theta_3} + A)\cos\theta \\
&f(19) = -B\sin\theta_2\sin\theta + 2B\cos\theta_2\cos\theta \\
&f(17) = B\sin\theta_1 \sin\theta - (B\cos\theta_1 + B\sin\theta_1 \frac{\cos\theta_3}{\sin\theta_3} + A)\cos\theta
\end{align}

And we have
\begin{align}
&f(22) = \left[ -B\sin\theta_2 \cos\theta_1 - 2B\cos\theta_2\sin\theta_1 \right] \sin\theta \\
&\quad + \left[ B \frac{-\cos\theta_1\sin\theta_2\cos\theta_3 - 2\sin\theta_1\cos\theta_2\cos\theta_3 - 2\sin\theta_1\cos\theta_3}{\sin\theta_3} - B\cos\theta_1 + A \right]\cos\theta.
\end{align}
Using $17$ and $19$, we obtain
\begin{align}
&f(21) = (2B \sin\theta_1\cos\theta_3 + B\cos\theta_1\sin\theta_3 + A\sin\theta_3)\sin\theta \\
&\quad + \left[ B \frac{2\sin\theta_1\cos^2\theta_3  + \cos\theta_1\sin\theta_2 + 2\sin\theta_1\cos\theta_2}{\sin\theta_3} + B \cos\theta_1\cos\theta_3  + A \cos\theta_3 \right] \cos\theta
\end{align}
But then we can use the $C^0$ condition with $22$ to get the relation
\begin{align}
0 &= B \left[ 2\sin\theta_1\cos\theta_3\sin\theta_3 + 2\cos\theta_1 + 2\cos\theta_1\cos\theta_2 - \sin\theta_1\sin\theta_2 \right. \\
&\quad \left. +  \frac{\cos\theta_3}{\sin\theta_3} (-2\sin^3\theta_1 + 2\cos\theta0\sin\theta_1 + 4\sin\theta_1\cos\theta_2 ) \right] .
\end{align}
Notice the terms involving $A$ cancel!  One can readily calculate the term in the brackets is $= 3.3$ (to one decimal place), and therefore we must have $B = 0$.  We deduce
\begin{gather}
f(6) = f(7) = f(8) = f(9) = f(19) = 0.
\end{gather}

We now calculate
\begin{align}
&f(14) = A(-\sin\theta_3\cos\theta_1 - 2\cos\theta_3\sin\theta_1)\sin\theta \\
&\quad + A (\sin\theta_2)^{-1} \left[ -\cos\theta_1\cos\theta_2\sin\theta_3 - 2\sin\theta_1\cos\theta_2\cos\theta_3 - \cos\theta_1\sin\theta_3 - 2\cos\theta_3\cos\theta_1 \right] \cos\theta .
\end{align}
And
\begin{gather}
f(20) = -A\sin\theta_3\sin\theta + 2A\cos\theta_3\cos\theta.
\end{gather}

Since $B = 0$ we see $f(20)$ has precisely the same form as $f(13)$, and so by using $20$ and $12$ we see that $f(16)$ correspondingly has the same expression as $f(14)$.  Now we can additionally use the $C^0$ condition at vertex $(14, 12, 16)$ to get the condition
\begin{align}
&A \left[ -2\cos\theta_1\sin\theta_3 - 4\sin\theta_1\cos\theta_3 - 2\cos\theta_1\cos\theta_2\sin\theta_3 - 4\sin\theta_1\cos\theta_2 \right. \\
&\quad \left. + \sin\theta_1\sin\theta_2\sin\theta_3 - 2\cos\theta_1\sin\theta_2\cos\theta_3 \right] = 0.
\end{align}
The term in the brackets is $-3.8$ (to one decimal), and we deduce $A = 0$ also.  By symmetry we deduce $f(i) = 0$ for all $i$.



\item \textbf{Four equal quadrilaterals and four equal pentagons}, forming $18$ edges: ``each quadrilateral surrounded by three pentagons and one quadrilateral, and each pentagon by three quadrilaterals and two pentagons, and having the arcs held in common by two quadrilaterals (and the quadrilateral arcs opposite to them) being of length $\theta_3 = 83.802^\circ$ and the other quadrilateral arcs of length $\theta_2 = 58.257^\circ$ and all remaining edges of length $\theta_1 = 13.559^\circ$.''

\begin{figure}
\centering
\includegraphics[width=5in]{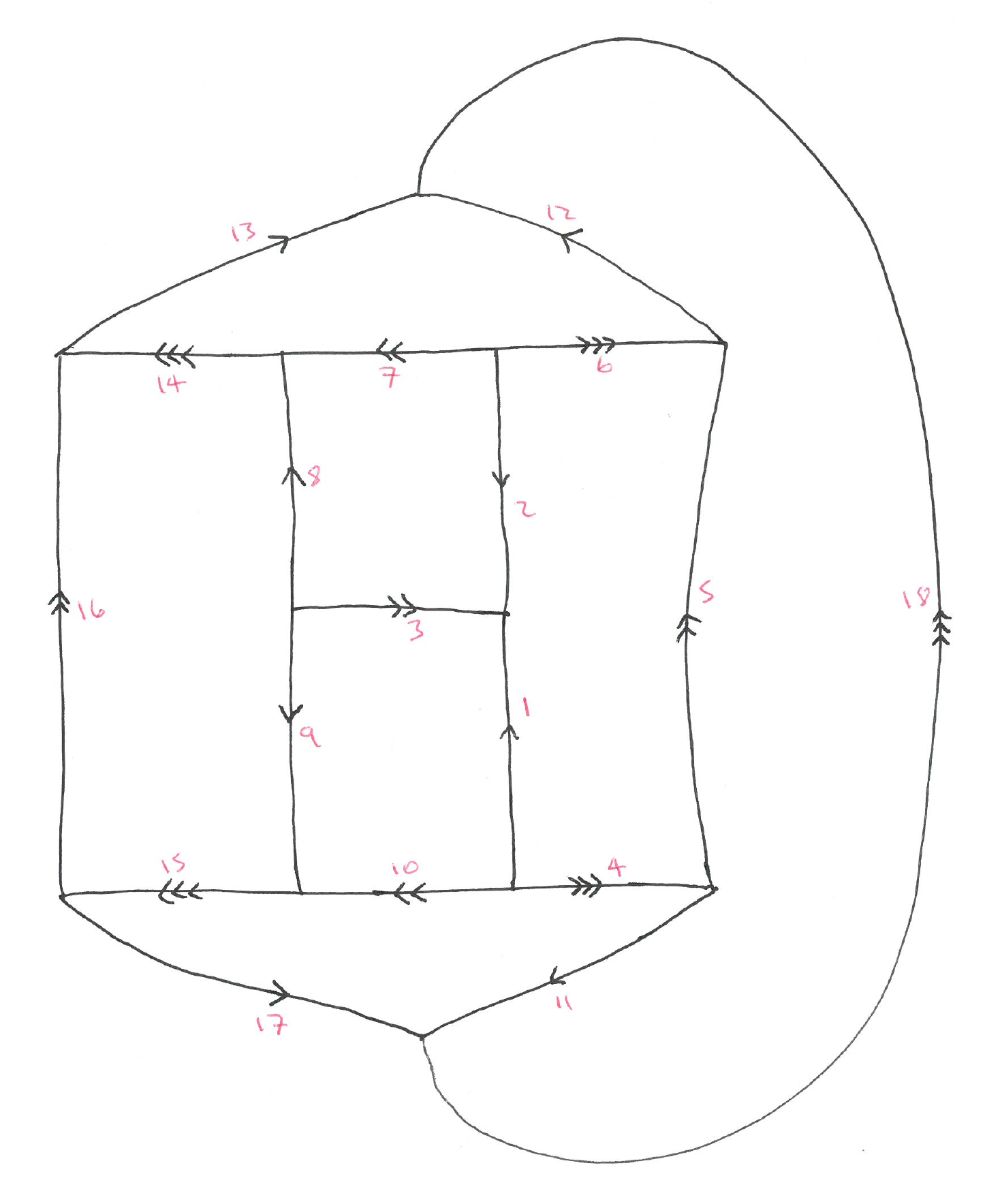}
\caption{Four equal quadrilaterals and four equal pentagons}
\end{figure}


Let us calculate.  We have directly
\begin{align}
&f(4) = A\cos\theta, \quad f(10) = -A\cos\theta, \quad f(9) = -A\frac{\sin\theta_3}{\sin\theta_2}\cos\theta, \quad f(8) = A\frac{\sin\theta_3}{\sin\theta_2}\cos\theta, \\
&\quad f(7) = A\cos\theta, \quad f(6) = -A\cos\theta,
\end{align}
for some constant $A$.  We have
\begin{align}
&f(5) = -A\sin\theta_1 \sin\theta - (A\sin\theta_1\frac{\cos\theta_3}{\sin\theta_3} + A \frac{\sin\theta_1}{\sin\theta_3})\cos\theta \\
&f(14) = -A\sin\theta_3 \sin\theta + (A \cos\theta_3 + A\frac{\sin\theta_3}{\sin\theta_2}\cos\theta_2)\cos\theta \\
&f(15) = A\sin\theta_3 \sin\theta - (A\cos\theta_3 + A\frac{\sin\theta_3}{\sin\theta_2} \cos\theta_2)\cos\theta
\end{align}
And
\begin{align}
&f(16) = A (\cos\theta_1\sin\theta_3 + \sin\theta_1\cos\theta_3 + \frac{\sin\theta_1\cos\theta_2\sin\theta_3}{\sin\theta_2})(\sin\theta + \frac{\cos\theta_3 + 1}{\sin\theta_3} \cos\theta).
\end{align}

We have
\begin{align}
&f(17) = A \left[ \cos\theta_1 \sin\theta_3 + \sin\theta_1\cos\theta_3 + \frac{\sin\theta_1\cos\theta_2\sin\theta_3}{\sin\theta_2} \right] \sin\theta \\
&\quad + A \left[ \sin\theta_1\sin\theta_3 - \cos\theta_1\cos\theta_3 - \frac{\cos\theta_1\cos\theta_2\sin\theta_3}{\sin\theta_2} \right. \\
&\quad\quad \left. - (\frac{\cos\theta_3 + 1}{\sin\theta_3})(\cos\theta_1\sin\theta_3 + \sin\theta_1\cos\theta_3 + \frac{\sin\theta_1\cos\theta_2\sin\theta_3}{\sin\theta_2} ) \right] \cos\theta.
\end{align}

Now using $4$, $5$, we obtain
\begin{gather}
f(11) = -A\sin\theta_1\sin\theta + A(\frac{\sin\theta_1}{\sin\theta_3} + \sin\theta_1\frac{\cos\theta_3}{\sin\theta_3} + \cos\theta_1) \cos\theta.
\end{gather}

But we additionally have a condition with $17$, giving us the relation:
\begin{align}
&A\left[ 2\sin\theta_1\cos\theta_2 + 2\sin\theta_1\cos\theta_2\cos\theta_3 + 2\frac{\sin\theta_1\sin\theta_2}{\sin\theta_3} + \frac{\sin\theta_1\sin\theta_3}{\sin\theta_2} - 3\sin\theta_1\sin\theta_2\sin\theta_3 \right. \\
&\quad \left. +2 \frac{\sin\theta_1\sin\theta_2\cos\theta_3}{\sin\theta_3}+ 2\cos\theta_1\sin\theta_2 + 2\cos\theta_1\cos\theta_2\sin\theta_3 + 2\cos\theta_1\sin\theta_2\cos\theta_3 \right] = 0.
\end{align}
Therefore we must have $A = 0$.  It follows directly that $f(i) = 0$.


\item \textbf{Three regular quadrilaterals and six equal pentagons}, forming $21$ edges: ``each quadrilateral surrounded by four pentagons and each pentagon by two quadrilaterals and three pentagons, with the quadrilateral edge being of length $\theta_2 = 70.529^\circ$, the pentagonal edge adjacent to just one quadrilateral vertex being of length $\theta_3 = 35.264^\circ$, and the remaining three edges of length $\theta_1 = 10.529^\circ$.''

\begin{figure}
\centering
\includegraphics[width=6in]{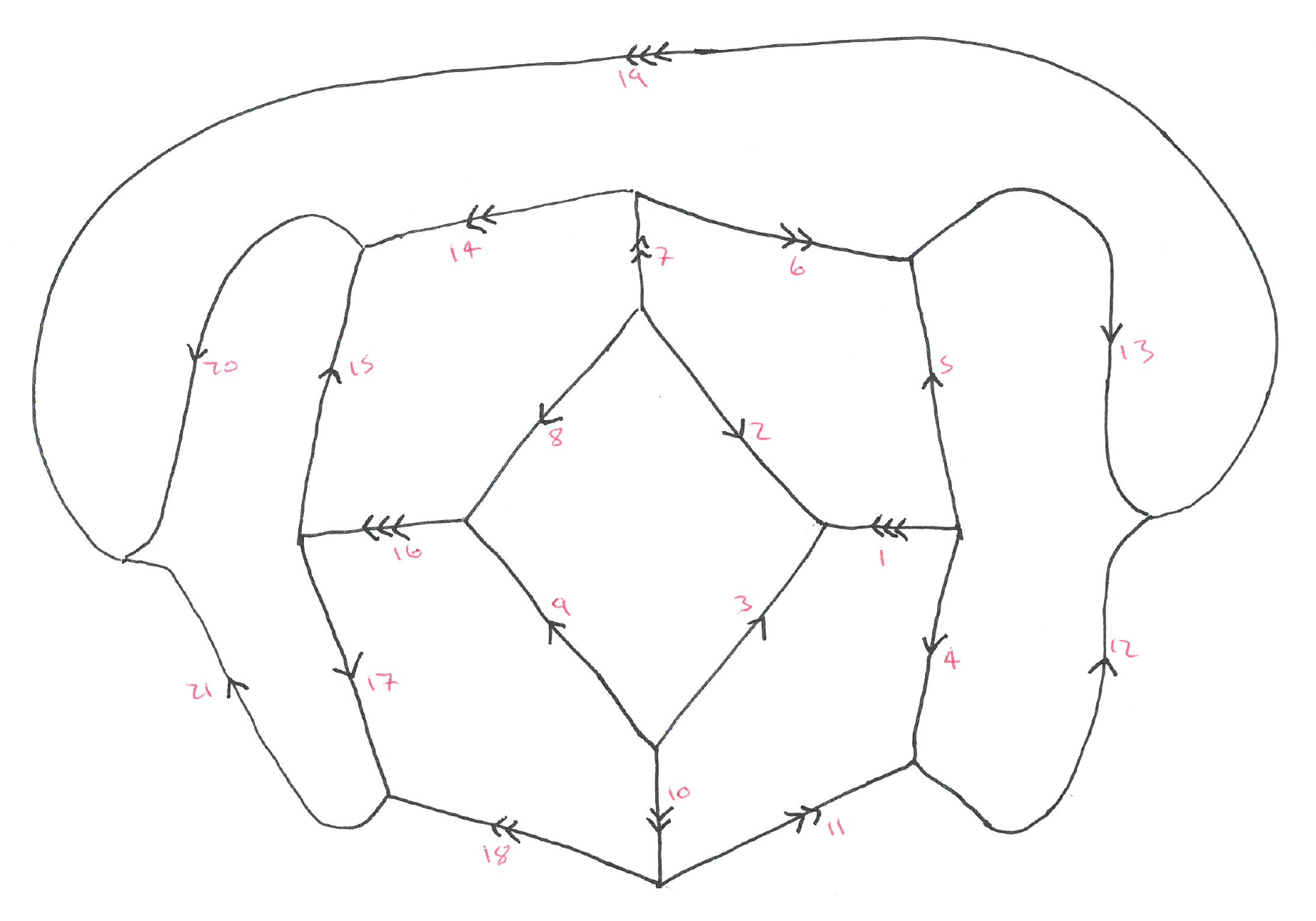}
\caption{Three regular quadrilaterals and six equal pentagons}
\end{figure}



We have directly that
\begin{align}
&f(7) = A\cos\theta, \quad f(8) = -A\cos\theta, \quad f(9) = -A\cos\theta, \quad f(10) = A\cos\theta \\
&\quad f(5) = B\cos\theta, \quad f(4) = -B\cos\theta,
\end{align}
for some constants $A, B$.  We obtain
\begin{align}
&f(6) = -A\sin\theta_3\sin\theta + (-A\cos\theta_3 + B\frac{\sin\theta_2}{\sin\theta_3})\cos\theta \\
&f(11) = -A\sin\theta_3\sin\theta - (A\cos\theta_3 + B \frac{\sin\theta_2}{\sin\theta_3})\cos\theta \\
&f(12) = B\sin\theta_2\sin\theta - (B\sin\theta_2 \frac{\cos\theta_3}{\sin\theta_3} + B\cos\theta_2 + A)\cos\theta \\
&f(13) = -B\sin\theta_2\sin\theta + (B\sin\theta_2\frac{\cos\theta_3}{\sin\theta_3} + B\cos\theta_2 - A) \cos\theta. 
\end{align}
But now we can use the $C^1$ condition at vertex $12, 19, 13$ to get the relation
\begin{gather}
B \left[ 2\cos\theta_2\sin\theta_3 + \sin\theta_2\cos\theta_3 \right] = 0,
\end{gather}
which necessitates that $B = 0$.

We proceed by calculating
\begin{align}
&f(16) = A\sin\theta_2\sin\theta - 2A\cos\theta_2\cos\theta \\
&f(14) = f(18) = -A\sin\theta_3 \sin\theta + 2A\cos\theta_3\cos\theta \\
&f(15) = f(17) = (A\cos\theta_1\sin\theta_2 + 2A\sin\theta_1\cos\theta_2)\sin\theta \\
&\quad + A\left[ \frac{\cos\theta_1\cos\theta_2\sin\theta_2 + 2\sin\theta_1\cos^2\theta_2 + 3\sin\theta_3\cos\theta_3}{\sin\theta_2} \right] \cos\theta.
\end{align}
But now we can apply the $C^0$ condition at vertex $16, 15, 17$ to get
\begin{gather}
A \left[ 6\cos\theta_3 \sin\theta_3 + 5\sin\theta_1\cos^2\theta_2 - \sin\theta_1 \right] = 0,
\end{gather}
which implies $A = 0$.  It then follows directly that $f(i) = 0$ for all $i$.

\end{enumerate}

This completes the proof of integrability when $k = 1$.  Suppose now $k \geq 2$.  We can handle the projection $\pi_{\R^3\times\{0\}} \circ v$ in precisely the same manner as above.  On the other hand, given any coordinate vector $e \in \{0\}\times \R^{k-1}$, let us define
\begin{gather}
f(i)(\theta) = e \cdot v(i)(\theta),
\end{gather}
and observe $f(i)$ takes the same form \eqref{eqn:form-of-f}.  By Lemma \ref{lem:vect-vs-scalar-cond} the compatibility conditions are now
\begin{gather}
f(i_1)(p) = f(i_2)(p) = f(i_3)(p), \quad \text{and} \quad \sum_{j=1}^3 (n(i_j) \cdot \hat \ell(i_j)) f'(i_j)(p) = 0,
\end{gather}
whenever $\ell(i_1), \ell(i_2), \ell(i_3)$ share a common vertex $p$.  Since $f'' + f = 0$, we see that the functions $f'(i)$ satisfy conditions \eqref{eqn:scalar-initial-conds}, \eqref{eqn:scalar-compat-conds}, and we can apply the proof above to deduce every $f'(i) = 0$.  This implies $f = 0$, and hence $\pi_{\{0\}\times \R^{k-1}} \circ v$ is zero also.
\end{proof}


\section{Corollaries}\label{sec:corollaries}

Given Theorem \ref{thm:main-decay} and some background results on $(\mass, \eps, \delta)$-minimizing sets, the proofs of our Corollaries are essentially standard.

\begin{proof}[Proof of Corollary \ref{cor:no-spine-reg}]
The argument is standard, but we include it for completeness.  Take $\delta_1(\bC)$ as in Theorem \ref{thm:main-decay}, and $\eps_1(\bC, \beta = 1/100, \tau = 1/100)$ as in Lemma \ref{lem:poly-graph}.  Ensure $\delta \leq \delta_1$.

If $M$ is such that $M \in \cN_\delta(\bC)$, and $\theta_M(0) \geq \theta_{\bC}(0)$, then $E_{\delta_1}(M, \bC, 1) \leq \delta_1^2$, and $M$ satisfies the $\delta$-no-holes condition in $B_{3/4}$ w.r.t. $\bC$, for all $\delta > 0$.  We deduce by Theorem \ref{thm:main-decay} there is a sequence of rotations $q_i$ so that
\begin{gather}
E_{\delta_1}(M, q_i(\bC), \theta^i) \leq 2^{-i} E_{\delta_1}(M, \bC, 1).
\end{gather}

It follows that $|q_i - q_{i+1}| \leq c(\bC) 2^{-i} E(M, \bC, 1)$, and in particular there is a rotation $q$ so that
\begin{gather}
\rho^{-n-2} \int_{M \cap B_\rho} d_{q(\bC)}^2 \leq c(\bC) \rho^{2\mu} E(M, \bC, 1)
\end{gather}
for all $\rho \leq 1$, and for some $\mu = \mu(\bC)$.  Ensuring $c(\bC) \delta \leq \eps_1$, we can apply Lemma \ref{lem:poly-graph} at any scale $B_\rho$, with $\mu \leq \alpha$ in place of $\alpha$, to obtain a uniform $C^{1,\mu}$ decomposition of $M$ over $\bC$.  That is, in the sense of Definition \ref{def:poly-graph}, we have $M \cap B_{1/2} = \graph_{\bC}(u, f, \Omega)$, where $u$ and $f$ admit the pointwise bounds
\begin{align}
r^{-1}|u(i)| + |Du(i)| + r^\mu [Du(i)]_{\mu, \bC} \leq c(\bC) r^{\mu} E(M, \bC, 1)^{1/2} \\
r^{-1} |f(i)| + |Df(i)| + r^\mu [Df(i)]_{\mu, \bC} \leq c(\bC)r^{\mu} E(M, \bC, 1)^{1/2}.
\end{align}
\end{proof}

\begin{proof}[Proof of Theorem \ref{thm:spine-reg}]
The argument is same as the proof given in Section \ref{sec:graph} for Theorem \ref{thm:Y-reg}, except using Proposition \ref{prop:no-holes} in place of \ref{prop:no-holes-Y}, and Simon's $\eps$-regularity in addition to Allard's.
\end{proof}

To prove Theorems \ref{thm:main-reg} and \ref{thm:clusters} we need a few background results.  First, we prove assertion 3) of Theorem \ref{thm:background-clusters}, as promised.
\begin{lemma}
The underlying varifold $M^n = \haus^n \llcorner (\partial^* \clus(1) \cup \ldots \cup \partial^* \clus(N))$ associated to a minimizing $N$-cluster (where $\partial^*$ denotes the reduced boundary) has bounded mean curvature, and no boundary.  As a corollary, $M = \haus^n \llcorner (\partial \clus(1) \cup \ldots\cup \clus(N))$, where $\partial$ denotes the \emph{topological} boundary.
\end{lemma}

\begin{proof}
For convenience write $V = \{ \ba \in \R^{N+1} : \sum_h a_h = 0 \}$.  From \cite[Theorem VI.2.3]{Alm-eps-delta}/\cite[Theorem IV.1.14]{maggi}, we have the following: for any $N$-cluster $\clus$, there are constants $\eta, c, R$ (depending only on $\clus$), and a $C^1$ function
\begin{gather}
\Psi : B_\eta^{N+1} \times \R^{n+1} \to \R^{n+1},
\end{gather}
with $\Psi_{\ba = 0} = Id$, which satisfies for any $\ba \in B_\eta^{N+1} \cap V$:
\begin{gather}\label{eqn:vol-preserving-prop}
\spt (\Psi_{\ba} - Id) \subset B_R, \quad |\Psi_{\ba}(\clus(h)) \cap B_R| = |\clus(h) \cap B_R| + a_h, \quad | D \Psi_{\ba} - Id| \leq c \sum_{h=1}^N |a_h|.
\end{gather}
Of course we can also assume $B_R$ contains all the bounded chambers $\{ \clus(h) \}_{h=1}^N$.

Now suppose $\clus$ is a minimizing $N$-cluster, take $\Psi$ as above, and consider an arbitrary $C^1$ vector field $X$ supported in $B_R$ generating flow $\phi_t$.  Define the function $F : \R \times (B_\eta^{N+1} \cap V) \to \R^N$ by setting
\begin{gather}
F^h(t, \ba) = |\phi_t(\Psi_{\ba}(\clus(h)))| - |\clus(h)| .
\end{gather}
Choosing coordinates on $V$ via the map
\begin{gather}
\bb \in \R^N \mapsto (-\sum_{i=1}^N b_i, b_1, \ldots, b_N) \in \R^{N+1} \cap V,
\end{gather}
we obtain that
\begin{gather}
F(0, 0) = 0, \quad \partial_{b_h} F^k|_{(0, 0)} = \delta_{kh} \quad \forall k, h = 1, \ldots, N.
\end{gather}
Therefore, by the implicit function theorem we can find a $C^1$ curve $\ba : (-\eps, \eps) \to B_\eta^{N+1} \cap V$, so that $\ba(0) = 0$ and $F(t, \ba(t)) \equiv 0$.  In other words, the variation $\phi_t \circ \Psi_{\ba(t)}$ preserves the volume vector of $\clus$.

If $Y$ is the initial velocity vector field for $\Psi_{\ba(t)}$, then by \eqref{eqn:vol-preserving-prop} we have $|DY| \leq c \sum_{h=1}^N |a'_h(0)|$.  On the other hand, since $D_t F(t, \ba(t)) = 0$, we have for each $h$:
\begin{gather}
0 = \int_{\partial^* \clus(h)} (X + Y) \cdot \nu = \int_{\partial^* \clus(h)} X \cdot \nu + a'_h(0).
\end{gather}

Therefore, since $\clus$ is minimizing for volume-vector-preserving deformations, 
\begin{gather}
\int_M div_M(X) = -\int_M div_M(Y) \leq c \int_M |X|.
\end{gather}
This shows that $\delta M$ forms a bounded linear operator on $L^1(\mu_M)$, which implies $M$ has no boundary and bounded $H_M$.
\end{proof}

Next, we give a general ``sheeting'' theorem for $(\mass, \eps,\delta)$-minimizing varifolds, which effectively says that this class forms a multiplicity-one class.  This is well-known, and essentially the same as \cite[Corollary II.2]{taylor}.
\begin{lemma}\label{lem:almost-min-mult}
Let $M_i^n = \haus^n \llcorner \spt M_i$ be a sequence of (multiplicity-one) integral varifolds in $U \subset \R^{n+k}$ without boundary, such that: the $M_i$ have uniformly bounded mean curvature and mass, and each $\spt M_i$ is $(\mass, \eps, \delta)$-minimizing in $U$ (for uniform $\eps$, $\delta$).

If $M_i \to M$ as varifolds in $U$, then $M = \haus^n \llcorner \spt M$, and $\spt M$ is $(\mass, \eps, \delta)$-minimizing in $U$.  In particular, if $\bC$ is any tangent cone for $M$, then $\bC$ has multiplicity-one and $\spt \bC$ is $(M, 0, \infty)$-minimizing.
\end{lemma}

\begin{proof}
Since $M$ is integral, at $\mu_M$-a.e. $x$ we have an approximate tangent plane $P$.  Fix such an $x$, and suppose towards a contradiction that $\theta_M(x) = q > 1$.  By monotonicity, for sufficiently small $r$ and $i >> 1$, both $B_r(x) \cap \spt M$ and $B_r(x) \cap \spt M_i$ lie in an $\eta r$-neighborhood of $x + P$.  Therefore, if we construct a $C^1$ deformation which pushes $B_{r/2}(x)$ into $B_{r/2}(x) \cap (x + P)$, we save $\geq c(n) (q-1) r^n$ amount of area in $M_i$.  This contradicts $(\mass, \eps, \delta)$-minimality.

That $M$ is $(\mass, \eps, \delta)$-minimizing follows directly from the facts: a) any piecewise $C^1$ mapping $\phi$ induces a continuous map $\phi_\sharp$ on the space of integral varifolds; and b) any Lipschitz deformation on $M$ can be well-approximated by piecewise-$C^1$ deformations. 
\end{proof}

The last crucial fact we need is Taylor's classification of $2$-dimensional, $(\mass, 0, \infty)$-minimizing cones in $\R^3$.  The classification for $1$-d cones is trivial.  The following Lemma is a straightforward consequence of \cite[Proposition II.3]{taylor}.
\begin{lemma}\label{lem:almost-min-tangent}
Let $\bC^n$ be an $(\mass, 0, \infty)$-minimizing cone in $\R^{n+k}$.  If $\bC = \bC_0^1 \times \R^{n-1}$, then up to rotation $\bC_0 = \bY$.  If $k = 1$ and $\bC = \bC_0^2\times \R^{n-2}$, then (up to rotation) $\bC_0$ is either $\R^2$, $\bY\times \R$, or $\bT$.
\end{lemma}

Lemma \ref{lem:almost-min-tangent} highlights the importance of the cones $\bY\times \R$ and $\bT$: up to factors of $\R^m$, they are the only singular cones arising in the top three strata of $(\mass, \eps, \delta)$-minimizing sets.  Moreoever, they always occur with multiplicity one.  From these facts Theorem \ref{thm:clusters} follows in a straightforward way from our decay Theorem \ref{thm:main-decay} and no-holes Proposition \ref{prop:no-holes}.

\begin{proof}[Proof of Theorem \ref{thm:main-reg}/Theorem \ref{thm:clusters}]
Recall the definitions of $k$-strata and $(k,\eps)$-strata as given in Section \ref{sec:prelim}.  Let us define $M_k = S^k(M)$ to be the $k$-th stratum, for $k = n-3, \ldots, n$.  Conclusions 1), 2), 3) follow immediately from Lemmas \ref{lem:almost-min-mult}, \ref{lem:almost-min-tangent}, and the $\eps$-regularity Theorems of Allard (Theorem \ref{thm:allard}), Simon (Theorem \ref{thm:Y-reg}), and Theorem \ref{thm:spine-reg}.

More generally, the aforementioned Lemmas and Theorems show each stratum $S^m(M)$ (for $m = n, n-1, n-2, n-3$) is closed in the following sense: suppose $M_i$ is a family of varifolds satisfying the hypotheses of Theorem \ref{thm:main-reg} with uniform bounds on mass, mean curvature, and uniform $\eps, \delta$.  If $M_i \to M$, and $x_i \in S^m(M_i)$ converge to $x \in U$, then $x \in S^m(M)$.

We claim that, for every compact $K \subset U$, there is an $\eps> 0$ so that $S^{n-3} \cap K \subset S^{n-3}_\eps$.  This is an easy consequence of the closedness of the strata.  Otherwise, if the claim was false, we would have sequences $x_i \to x \in K \cap S^{n-3}$, $\eps_i \to 0$, and $r_i \in (0, \min\{d(x_i, \partial U), 1\})$, for which $M$ is $(n-2, \eps_i)$-symmetric in $B_{r_i}(x_i)$.  Let $M_i = r_i^{-1}(M - x_i)$.  Then the $M_i$ have uniformly bounded mass and first-variation in $B_1$, each $M_i$ is $(n-2, \eps_i)$-symmetric in $B_1$, while $0 \in S^{n-3}(M_i)$.  

Passing to a subsequence, we have varifold convergence $M_i \to \bC$, where $\bC$ is a $(n-2)$-symmetric cone.  But by the closedness property, $0 \in S^{n-3}(\bC)$.  This is a contradiction.  Conclusion 4) is now a consequence of Naber-Valtorta \cite{naber-valtorta}.
\end{proof}

\section{Appendix}

\subsection{Linear algebra} We require some elementary linear algebra.  The following Lemma relates vectorial and scalar compatability conditions.  Notice how the scalar conditions in different cases are dual to each other.
\begin{lemma}\label{lem:vect-vs-scalar-cond}
Let $\omega_1, \omega_2, \omega_3$ be unit vectors, with $\omega_1 + \omega_2 + \omega_3 = 0$, and take vectors $v_1, v_2, v_3$ so that $v_i \perp \omega_i$ for each $i$.  Write $P^2$ for the $2$-plane spanning $\omega_i$.

\begin{enumerate}
\item[A)]
We can write $\pi_P(v_i) = \alpha_i e^{i\pi/2}\omega_i$.  Then 
\begin{gather}
\pi_{P}(v_i) = \pi_{<\omega_i>^\perp}(u) \text{ for some fixed $u$} \iff \sum_i \alpha_i = 0,
\end{gather}
and
\begin{gather}
\sum_i \pi_P(v_i) = 0 \iff \alpha_1 = \alpha_2 = \alpha_3.
\end{gather}

\item[B)] Suppose $\pi_{P^\perp}(v_i) = \alpha_i v$ for some fixed $v \in P^\perp$.  Then
\begin{gather}
\pi_{P^\perp}(v_i) = \pi_{<\omega_i>^\perp}(u) \text{ for some fixed $u$} \iff \alpha_1 = \alpha_2 = \alpha_3,
\end{gather}
and
\begin{gather}
\sum_i \pi_{P^\perp}(v_i) = 0 \iff \sum_i \alpha_i = 0.
\end{gather}
\end{enumerate}
Here $<\omega_i>^\perp$ denotes the orthogonal complement to the line spanned by $\omega_i$.
\end{lemma}

\begin{proof}
Since part B) is obvious, let us concentrate on part A).  For ease of notation we can swap the role of $\omega_i$ and $e^{i\pi/2}\omega_i$.  Let us identify $P$ with $\R^2$, and the $\omega_i$ with $1, e^{2\pi i/3}, e^{4\pi i/3}$.

The ``only if'' direction of the first statement is obvious.  Conversely, given $\alpha_i$ with $\sum_i \alpha_i$, define
\begin{gather}
u = \alpha_1 \omega_1 + \frac{1}{\sqrt{3}}(\alpha_2 - \alpha_3) e^{i\pi/2} \omega_1.
\end{gather}
Trivially $\pi_{\omega_1}(u) = \alpha_1$, and we calculate
\begin{align}
\pi_{\omega_2}(u)
&= \alpha_1 (\omega_2 \cdot \omega_1) + \frac{1}{\sqrt{3}}(\alpha_2 - \alpha_3) (\omega_2 \cdot (e^{i\pi/2}\omega_1)) \\
&= \frac{-1}{2} \alpha_1 + \frac{1}{2}(\alpha_2 - \alpha_3) \\
&= \alpha_2.
\end{align}
By a symmetric calculation we have $\pi_{\omega_3}(u) = \alpha_3$ also.

We prove the second assertion of A).  We have $e_2 \cdot \omega_2 = -e_2 \cdot \omega_3 = \sqrt{3}/2$, and $e_1 \cdot \omega_2 = e_1\cdot \omega_3 = -1/2$.  Therefore,
\begin{align}
\sum_i \alpha_i \omega_i = 0  
&\iff \sqrt{3}/2 (\alpha_2 - \alpha_3) = 0 \text{ and } \alpha_1 - \frac{1}{2}(\alpha_2 + \alpha_3) = 0 \\
&\iff \alpha_1 = \alpha_2 = \alpha_3. \qedhere
\end{align}
\end{proof}

\begin{lemma}\label{lem:sum-is-zero}
Suppose $\omega_1, \omega_2, \omega_3$ are unit vectors, with $\omega_1 + \omega_2 + \omega_3 = 0$.  Let $v_1, v_2, v_3$ be vectors, such that $v_i \perp \omega_i$ for each $i$.

Then the following are equivalent:
\begin{enumerate}
\item[A)] $v_1 + v_2 + v_3 = 0$;

\item[B)] There is a skew-symmetric $A$, which is zero on the orthogonal complement of $\mathrm{span}(v_1, v_2, v_3, \omega_1, \omega_2, \omega_3)$, such that $Av_i = \omega_i$;

\item[C)] For any vector $u$, we have $\sum_i v_i \cdot \pi_{<\omega_i>^\perp}(u) = 0$.  Here $<\omega_i>^\perp$ is the orthogonal complement to the line spanned by $\omega_i$.
\end{enumerate}
\end{lemma}

\begin{proof}
We show A) implies B).  The converse B) $\implies$ A) is trivial.  If $P^2$ is the plane containing the points $0, \omega_1, \omega_2$, then clearly $\omega_3$ must lie in $P$ also.  Therefore, after a suitable rotation, we can identify $P^2$ with $\R^2$, and the $\omega_i$ with $1, e^{2\pi i/3}, e^{4\pi i/3} \in \R^2$.

Let $v_i^T$ and $v_i^\perp$ be the orthogonal projections of $v_i$ to $P$ and $P^\perp$ respectively.  Define the matrix
\begin{gather}
A_{ij} = \sum_{\ell=1}^3 \left( \omega_\ell \wedge \left( \frac{v_\ell^T}{3/2 + \sqrt{3}/2} + \frac{v_\ell^\perp}{3/2} \right) \right)(e_j, e_i),
\end{gather}
where $e_i$ is the standard basis of $\R^{n+k}$.  Of course in Euclidean space we can identify vectors and covectors via the standard inner product.  Clearly $A_{ij}$ is skew-symmetric.

By symmetry it will suffice to show $A\omega_1 = v_1$.  First, since $\sum_i v_i^T = 0$ and $v_i^T \cdot \omega_i = 0$, then one can easily check that
\begin{gather}
v_i^T = \alpha e^{i \pi/2} \omega_i \quad i = 1, 2, 3,
\end{gather}
for some fixed $\alpha \in \R$, i.e. each $v_i^T$ is a $90^0$ rotation of $\alpha \omega_i$.  We therefore have
\begin{align}
(3/2 + \sqrt{3}/2) (A \omega_1)^T 
&= v_1^T + (\omega_2 \cdot \omega_1) v_2^T + (\omega_3 \cdot \omega_1) v_3^T - (v_2\cdot \omega_1) \omega_2 - (v_2 \cdot \omega_1) v_3 \\
&= v_1^T - \frac{1}{2} (v_2^T + v_3^T) - \alpha ( (e^{i\pi/2} \omega_2)\cdot \omega_1) \omega_2 - \alpha ((e^{i\pi/2} \omega_3) \cdot \omega_1) \omega_3 \\
&= \frac{3}{2} v_1^T + \frac{\alpha}{2} (\omega_2 - \omega_3) \\
&= \frac{3}{2} v_1^T + \frac{\sqrt{3}}{2} \alpha e^{i\pi/2} \omega_1 \\
&= (3/2 + \sqrt{3}/2) v_1^T.
\end{align}

Similarly, we have
\begin{align}
\frac{3}{2} (A \omega_1)^\perp 
&= v_1^\perp + (\omega_2 \cdot \omega_1) v_2^\perp + (\omega_3 \cdot \omega_1) v_3^\perp = v_1^\perp - \frac{1}{2} (v_2^\perp + v_3^\perp) = \frac{3}{2} v_1^\perp.
\end{align}
This shows $A\omega_1 = v_1$.

We show A) $\iff$ C).  With $P$ as above, we trivially have that
\begin{gather}
\pi_{<\omega_i>^\perp}(u) = u \quad \forall u \in P^\perp.
\end{gather}
Therefore $\sum_i v_i^\perp = 0$ if and only if $\sum_i v_i \cdot \pi_{<\omega_i>^\perp}(u) = 0$ for all $u \in P^\perp$.

On the other hand, given $u \in P$, and our assumption $v_i \perp \omega_i$, then we can write
\begin{gather}
\pi_{<\omega_i>^\perp}(u) = \beta_i e^{i\pi/2} \omega_i, \quad v_i^T = \alpha_i e^{i\pi/2} \omega_i,
\end{gather}
where $\beta_i \in \R$ satisfy $\sum_i \beta_i = 0$, and $\alpha_i \in \R$.  Then, using Lemma \ref{lem:vect-vs-scalar-cond}, we have
\begin{align}
\sum_i v_i^T = 0 
&\iff \alpha_1 = \alpha_2 = \alpha_3 \\
&\iff \sum_i \alpha_i \beta_i = 0 \quad \forall \beta_i \text{ such that} \sum_i \beta_i = 0 \\
&\iff \sum_i v_i \cdot \pi_{<\omega_i>^\perp}(u) = 0 \quad \forall u \in P^T.
\end{align}
This completes the proof.
\end{proof}

\subsection{Two variation inequalities}\label{sec:variation}

We sketch the proof of the estimates \eqref{eqn:lem-density-1} and \eqref{eqn:lem-density-2}.  Both are minor modifications of the derivation given in \cite{simon1}.

\begin{lemma}
Let $\bC = \bC^\ell_0 \times \R^m$, and take $M \in \cN_\eps(\bC)$ with $\theta_M(0) \geq \theta_\bC(0)$.  Let $\phi: \R \to \R$ be any smooth function satisfying $\phi' \leq 0$, $\phi \equiv 1$ on $[0, 1/10]$, and $\phi \equiv 0$ on $[2/10, \infty)$.  Then we have
\begin{gather}\label{eqn:app-density-1}
\frac{1}{2} n (1/10)^n \int_{M \cap B_{1/10}} \frac{|X^\perp|^2}{R^{n+2}} \leq \int_M \phi^2(R) - \int_{\bC} \phi^2(R) + c(\bC, \phi) ||H_M||_{L^\infty(B_1)},
\end{gather}
and
\begin{align}\label{eqn:app-density-2}
\ell \left( \int_M \phi^2(R) - \int_{\bC} \phi^2(R) \right) &\leq \left( \int_M 2\phi |\phi'| r^2/R - \int_{\bC} 2\phi |\phi'| r^2/R \right) \\
&\quad  + \int_M 2\phi (\phi')^2 |(x, 0)^\perp|^2 + c(\bC, \phi) ||H_M||_{L^\infty(B_1)}.
\end{align}
\end{lemma}

\begin{proof}
Write $\Lambda = ||H||_{L^\infty(B_1)}$.  By the monotonicity formula (see e.g. \cite{simon:gmt}) we have
\begin{gather}
e^{\Lambda \rho}\theta_M(0, \rho) - \theta_M(0) \geq \int_{M \cap B_\rho} \frac{|X^\perp|^2}{R^{n+2}} \quad \forall \rho < 1.
\end{gather}

By plugging (a $C^1$ approximation to) the vector field $(x, y) 1_{B^{n+k}_\rho}$ into the first variation \eqref{eqn:first-variation}, and using the coarea formula, we obtain
\begin{gather}
\frac{n}{\rho} \mu_M(B_\rho) - c\Lambda \leq D_\rho \int_{M \cap B_\rho} |\nabla^T R|^2 \leq D_\rho \mu_M(B_\rho).
\end{gather}
Therefore, taking $\eps \geq \Lambda$ small, by the monotonicity formula and our assumption $\theta_M(0) \geq \theta_{\bC}(0)$ we have
\begin{align}
\frac{1}{2} n \rho^{n-1} \int_{M \cap B_\rho} \frac{|X^\perp|^2}{R^{n+2}} 
&\leq D_\rho \mu_M(B_\rho) - e^{-\Lambda\rho} n\rho^{n-1} \theta_\bC(0) + c\Lambda \\
&\leq D_\rho( \mu_M(B_\rho) - \mu_\bC(B_\rho)) + c(\bC)(\Lambda + (1 - e^{-\Lambda}))
\end{align}
Now multiply by $\phi^2(\rho)$ and integrate in $\rho \in [0, 1]$ to obtain \eqref{eqn:app-density-1}.

We prove \eqref{eqn:app-density-2}.  Plugging the vector field $(x, 0) \phi^2(R)$ into the first variation, and rearranging, gives
\begin{gather}\label{eqn:app-density-3}
\int_M (\ell + \frac{1}{2} <M^\perp, \{0\}\times \R^m>^2)\phi^2 \leq \int_M 2\phi|\phi'| r^2/R + 2(\phi')^2 |(x, 0)^\perp|^2 + c(\bC, \phi) \Lambda .
\end{gather}
On the other hand, using Fubini and integrating by parts in $r$, gives
\begin{gather}\label{eqn:app-density-4}
\ell \int_{\bC} \phi^2 = \int_\axis \int_0^\infty \phi(\sqrt{r^2 + |y|^2})^2 \,\ell r^{\ell-1} \theta_{\bC}(0) dr dy  = \int_\bC -2\phi \phi' r^2/R.
\end{gather}
Now subtract \eqref{eqn:app-density-3} from \eqref{eqn:app-density-4}.
\end{proof}

\subsection{Graphicality for $\bC_0$ smooth}\label{sec:graphical-smooth}

We prove the analogue of decomposition Lemma \ref{lem:poly-graph} when $\bC_0^\ell$ is \emph{smooth}, which allows us in certain circumstances to remove the multiplicity-one hypothesis of \cite{simon1}.  The proof is essentially the same as for Lemma \ref{lem:poly-graph}, but simpler.

In this section we always assume $\bC^n = \bC_0^\ell\times \R^m$, for $\bC_0$ smooth.  Recall the torus
\begin{gather}
U(\rho, y, \gamma) = \{ (\xi, \eta) \in \R^{\ell+k}\times \R^m: (|\xi| - \rho)^2 + |\eta - y|^2 \leq \gamma \rho^2 \},
\end{gather}
and the ``halved-torus''
\begin{gather}
U_+(\rho, y, \gamma) = U(\rho, y, \gamma) \cap \{ (\xi, \eta) : |\xi| \geq \rho \}.
\end{gather}

We first demonstrate global graphical structure, but without good estimates.
\begin{lemma}\label{lem:global-graph}
For any $\beta, \tau > 0$ there is an $\eps_1(\refC, \beta, \tau)$ so that the following holds.  Take $M \in \cN_{\eps_1}(\refC)$.  Then there is a domain $\Omega \subset \bC$, and smooth function $u : \Omega \to \bC^\perp$, so that
\begin{gather}\label{eqn:global-graph}
M \cap B_{3/4} \setminus B_\tau(\{0\}\times \R^m) = \graph_C(u), \quad r^{-1} |u| + |\nabla u| \leq \beta.
\end{gather}
\end{lemma}

\begin{proof}
This is essentially a direct Corollary of Lemma \ref{lem:mult-one}.  If the Lemma failed, we would have a counter-example sequence $M_i$.  Passing to a subsequence, we have multiplicity-$1$ convergence $M_i \to \refC$, on compact subsets of $B_1$.  Therefore, by Allard convergence is smooth in $B_1 \setminus (\{0\}\times \R^m)$.
\end{proof}

\begin{lemma}\label{lem:tiny-graph}
For any $\beta > 0$ there is an $\eps_2(\refC, \beta)$ so that the following holds.  Take $M \in \cN_{1/10}(\refC)$. Take $\rho \leq 1/2$, and $\eta \in B_{3/4}^m(0)$, and suppose
\begin{equation}\label{eqn:tiny-u}
M \cap U_+(\rho, y, 1/16) = \graph_\bC(u), \quad r^{-1} |u| + |\nabla u| \leq 1/10, 
\end{equation}
and
\begin{equation}\label{eqn:tiny-dist}
\rho^{-n-2} \int_{M \cap U(\rho, y, 1/4)} d_\bC^2 + \rho ||H_M||_{L^\infty(U(\rho, y, 1/4))} \leq \eps_2 .
\end{equation}
Then we have
\begin{gather}
M \cap U(\rho, y, 1/8) = \graph_C(u), \quad r^{-1} |u| + |\nabla u| \leq \beta.
\end{gather}
\end{lemma}

\begin{proof}
By dilation invariance, we can suppose $\rho = 1/2$.  Suppose the Lemma is false, and consider a counterexample sequence $M_i$, $y_i$, $\eps_i \to 0$.  Passing to a subsequence, the $y_i \to y \in B_{3/4}^m$, and in $U(\rho, y, 1/5)$ the $M_i$'s converge to some stationary varifold supported in $\refC$.  The multiplicity in each disk is constant, but by the graphicality assumption we converge with multplicity one inside $U_+(\rho, y, 1/16)$.

Therefore the convergence is with multiplicity $1$, and therefore by Allard we satisfy the conclusions of the Lemma when $i >> 1$.
\end{proof}

\begin{lemma}[Graphicality for smooth $\refC_0^\ell \times \R^m$]\label{lem:graph-smooth}
Given any $\beta, \tau > 0$, there is an $\eps(\refC, \beta, \tau)$ so that: if $M \in \cN_{\eps}(\refC)$, then there are open sets $U \subset M$, $\Omega \subset \bC$, with $U \supset M \cap B_{3/4} \setminus B_\tau(\axis)$, and a function $u : \Omega \to \refC^\perp$, so that
\begin{gather}
M \cap U = \graph_{\bC}(u), \quad r^{-1} |u| + |\nabla u| \leq \beta, 
\end{gather}
and
\begin{gather}
\int_{\Omega} r^2 |\nabla u|^2 + \int_{M \cap B_{3/4} \setminus U} r^2 \leq c(\refC, \beta) E(M, \bC, 1) .
\end{gather}
Note that $c$ is \emph{independent} of $\tau$.
\end{lemma}

\begin{proof}
We can assume $\beta \leq 1/10$.  Ensure $\eps \leq \eps_1(\refC, \beta, \tau)$ and $\eps \leq \eps_2(\refC, \beta)$, the constants from Lemmas \ref{lem:global-graph}, \ref{lem:tiny-graph}.  So, from Lemma \ref{lem:global-graph}, $M \cap (B_{1/2} \setminus B_\tau(\axis)) = \graph_{\bC}(u)$ with $u : \Omega \subset \bC \to \bC^\perp$  satisfying estimates \eqref{eqn:global-graph}.

Given $y \in B_{3/4}^m$, define
\begin{gather}
r_y = \inf \{ r' : \text{\eqref{eqn:tiny-u} holds for all $r' < \rho < 3/4$} \}.
\end{gather}
By Lemma \ref{lem:global-graph} $r_y \leq \tau$.  Necessarily by Lemma \ref{lem:tiny-graph}, \eqref{eqn:tiny-dist} must fail at $\rho(\eta)$, and therefore
\begin{gather}
r_y^{n+2} \eps_2 \leq \int_{M \cap U(r_y, y, 1/4)} d_{\bC}^2 + r_y^{n+3} ||H||_{L^\infty(U(\rho, y, 1/4))}.
\end{gather}
In particular, by monotonicity we have
\begin{gather}
\int_{M \cap B_{20 r_y)}(0, \eta)} r^2 \leq c(\refC, \beta) \int_{M \cap U(r_y, y, 1/4)} d_{\bC}^2 + c(\bC, \beta) r_y^{n+3} ||H||_{L^\infty(B_1)}.
\end{gather}

Let $U$ be the region
\begin{gather}
U = \{ (x, y) \in M \cap B_{3/4} : |x| > \rho(y) \}, 
\end{gather}
so that $U \subset B_{3/4} \setminus B_\tau(\axis)$, and $M \cap U = \graph_{\bC}(u)$.

Take a Vitali subcover $\{B_{2\rho_i}(0, y_i)\}_i$ of $\{B_{2 r_y}(y)\}_{y \in B_{3/4}^m}$, and then by construction $\{B_{10\rho_i}(0, y_i)\}_i$ covers $\mu_M$-a.e. $B_{3/4} \setminus U$, and the $U(\rho_i, y_i, 1/4) \subset B_{2\rho_i}(0, y_i)$ are disjoint.  We deduce that
\begin{align}
\int_{M \cap B_{3/4} \setminus U} r^2 
&\leq \sum_i \int_{M \cap B_{20\rho_i}(0, y_i)} r^2 \\
&\leq \sum_i c\int_{M \cap U(\rho_i, y_i, 1/4)} d_{\bC}^2 + \sum_i c \rho_i^{n+3} ||H||_{L^\infty(B_1)} \label{eqn:global-graph-vitali} \\
& \leq c(\refC, \beta) E(M, \bC, 1) . 
\end{align}

Given $(x, y) \in \Omega$ with $d((x, y), \partial \Omega) < |x|/2$, then there are $(x', y') \in \partial \Omega$ with $|x| < 2|x'|$.  We have $(x', y') + u(x', y') \in B_{10\rho_i}(0, y_i)$ for some $i$, and since $|u(x', y')| \leq |x'|/10$, we have
\begin{gather}
|x| < 2|x'| < 20 \rho_i.
\end{gather}
We deduce that $\cup_i B_{20 y_i}(0, \eta_i)$ covers $\Omega' = \{ (x, y) \in \Omega : d((x, y), \partial \Omega) < |x|/2\}$.

Therefore, since $|\nabla u| \leq \beta$ we have from \eqref{eqn:global-graph-vitali} that
\begin{gather}
\int_{\Omega'} r^2 |\nabla u|^2 \leq c(\refC, \beta) E(M, \bC, 1). 
\end{gather}

If $(x, y) \in \Omega \setminus \Omega'$, then we can use Allard and smallness of $\beta$ to give bounds 
\begin{gather}
\int_{\bC \cap B_{|x|/4}(x, y)} r^2 |\nabla u|^2 \leq c \int_{\bC \cap B_{|x|/2}(x, y)} |u|^2 + c |x|^{n+3} ||H_M||_{L^\infty(B_{|x|}(x, y))}.
\end{gather}
Choose an appropriate Vitali subcover of $\{B_{|x|/4}(x, y) : (x, y) \in \Omega \setminus \Omega'\}$, then the resulting cover will have overlap bounded by $c(n)$, and therefore we have
\begin{gather}
\int_{\Omega \setminus \Omega'} r^2 |\nabla u|^2 \leq c E(M, \bC, 1). 
\end{gather}
\end{proof}

\bibliographystyle{plain}
\bibliography{references}

\end{document}